\documentclass[a4paper,reqno,oneside,10pt]{amsart}
\usepackage[utf8]{inputenc}
\usepackage[english]{babel}
\usepackage{csquotes}		

\usepackage{geometry}

\usepackage{graphicx}
\usepackage[dvipsnames]{xcolor}
\usepackage{comment}

\usepackage[shortlabels]{enumitem} 

\usepackage{amsfonts}
\usepackage{amssymb}
\usepackage{amsthm}
\usepackage{amsmath}
\usepackage{dsfont}
    \let\mathbbm\mathds

\usepackage{mathtools}
	\mathtoolsset{centercolon} 
\usepackage{mathrsfs}
\usepackage{xfrac}

\usepackage[scr=boondox]{mathalfa}
\usepackage{tikz}
    \usetikzlibrary{hobby}
    \usetikzlibrary{arrows}

\usepackage{pgfplots}

\renewcommand{\paragraph}[1]{%
  \medskip\noindent\textit{\underline{#1.}}
}

\makeatletter
\renewcommand\subsubsection{\@startsection{subsubsection}{3}%
  \z@{.5\linespacing\@plus.7\linespacing}{-.5em}%
  {\normalfont\itshape}}
\makeatother

\renewenvironment{proof}[1][\proofname]{%
  \par\medskip
  \noindent\textbf{{#1.}}\;\;%
}{%
  \hfill\qedsymbol\par\bigskip
}

\usepackage{aliascnt}

\newtheorem{proposition}{Proposition}[section]

\newaliascnt{theorem}{proposition}
\newtheorem{theorem}[theorem]{Theorem}
\aliascntresetthe{theorem}

\newaliascnt{lemma}{proposition}
\newtheorem{lemma}[lemma]{Lemma}
\aliascntresetthe{lemma}

\newaliascnt{corollary}{proposition}
\newtheorem{corollary}[corollary]{Corollary}
\aliascntresetthe{corollary}

\newaliascnt{question}{proposition}

\aliascntresetthe{question}

\theoremstyle{definition}

\newaliascnt{definition}{proposition}
\newtheorem{definition}[definition]{Definition}
\aliascntresetthe{definition}

\newaliascnt{assumptions}{proposition}

\aliascntresetthe{assumptions}

\newaliascnt{remark}{proposition}
\newtheorem{remark}[remark]{Remark}
\aliascntresetthe{remark}

\newtheorem*{note}{Note}

\newaliascnt{example}{proposition}
\newtheorem{example}[example]{Example}
\aliascntresetthe{example}

\usepackage[colorlinks=true,allcolors=MidnightBlue,pdfencoding=auto, psdextra]{hyperref}
\usepackage{cleveref}

\crefname{proposition}{proposition}{propositions}
\Crefname{proposition}{Proposition}{Propositions}

\crefname{theorem}{theorem}{theorems}
\Crefname{theorem}{Theorem}{Theorems}

\crefname{lemma}{lemma}{lemmas}
\Crefname{lemma}{Lemma}{Lemmas}

\crefname{corollary}{corollary}{corollaries}
\Crefname{corollary}{Corollary}{Corollaries}

\crefname{question}{question}{questions}
\Crefname{question}{Question}{Questions}

\crefname{definition}{definition}{definitions}
\Crefname{definition}{Definition}{Definitions}

\crefname{assumptions}{assumptions}{assumptions}
\Crefname{assumptions}{Assumptions}{Assumptions}

\crefname{remark}{remark}{remarks}
\Crefname{remark}{Remark}{Remarks}

\crefname{example}{example}{examples}
\Crefname{example}{Example}{Examples}

\crefname{appendix}{appendix}{appendices}
\Crefname{appendix}{Appendix}{Appendices}

\numberwithin{equation}{section}

\title[An elliptic regularization approach to the Stefan problem]{An elliptic regularization approach to the Stefan problem}
\date{\today}

\newcommand{\eps}{\varepsilon}
\newcommand{\de}{\partial}
\newcommand{\dif}{\,\text{d}}
\renewcommand{\phi}{\varphi}
\newcommand{\Interior}{\text{Int}}

\newcommand{\Jc}{\mathcal{J}}
\newcommand{\Fc}{\mathcal{F}}
\newcommand{\Oc}{\mathcal{O}}

\newcommand{\Hc}{\mathscr{H}}

\newcommand{\Uc}{\mathcal{U}}
\newcommand{\RR}{\mathbb{R}}
\newcommand{\R}{\mathbb{R}}
\newcommand{\NN}{\mathbb{N}}

\DeclareMathOperator{\dive}{div}
\newcommand{\weak}{\rightharpoonup}
\newcommand{\ind}{\mathbbm{1}}

\begin{document}

\author[F.~Paiano]{Filippo Paiano}

\author[B.~Velichkov]{Bozhidar Velichkov}

\address {Filippo Paiano \newline \indent
Dipartimento di Matematica, Universit\`a di Pisa \newline \indent
Largo Bruno Pontecorvo, 5, 56127 Pisa, Italy}
\email{filippo.paiano@phd.unipi.it}

\address {Bozhidar Velichkov \newline \indent
Dipartimento di Matematica, Universit\`a di Pisa \newline \indent
Largo Bruno Pontecorvo, 5, 56127 Pisa, Italy}
\email{bozhidar.velichkov@unipi.it}


\begin{abstract}
In this paper, we develop the theory for the two-phase Stefan problem with finite energy, possibly non-empty \emph{mushy region}, and space-dependent melting temperature. 
Specifically, we prove the existence of weak solutions with an elliptic regularization scheme.
Our existence theorem provides information about the regularity of the solutions: we prove that the temperature of weak solutions is in $H^1$ for all times, that the enthalpy is well defined and bounded for all times, and that both the enthalpy and the temperature are weakly continuous in time.
Finally, we establish a comparison principle for weak solutions on general unbounded domains and use it to show that every weak solution is recovered by the approximation scheme.
\end{abstract}
\maketitle

\tableofcontents

\section{Introduction}\label{s:introduction}

In this paper we prove existence, uniqueness and comparison principle for weak solutions to the two-phase Stefan problem. We do this via an elliptic regularization scheme, in which the weak solutions are obtained as the limit of minimizers to a family of approximating functionals $\mathcal F_\eps$.
We prove that the scheme converges to a couple $(u,\mu)$, which solves the two-phase Stefan problem in a weak integral sense. We also show that the enthalpy $\mu$ of the limit solution $(u,\mu)$ is weakly-$\ast$ continuous in time. This property allows to give a definition of a weak solution (enthalpy solution), which includes the notion of initial enthalpy and has naturally associated comparison principles, which allow to prove the uniqueness of the weak solution and the consistency of the elliptic regularization scheme (the fact that every weak solution is the limit of the elliptic regularization scheme).  Finally, we notice that our elliptic regularization allows to obtain weak solutions with finite Dirichlet energy and possibly non-empty mushy region. \medskip

\noindent The rest of the introduction is dedicated to the main results of the paper.
\begin{itemize}
\item In \Cref{sub:notation} we introduce the notation that we will use throughout the paper. 
\item In \Cref{sub:weak-solutions-definitions} we give the definition of \emph{weak/enthalpy solution} of the Stefan problem. 
\item In \Cref{sub:introduction-elliptic} we introduce the elliptic regularization scheme and the main theorem of the paper (\Cref{t:main}). We also state the two main existence results for weak solutions (\Cref{cor:existence-weak-solutions-Dirichlet} and \Cref{cor:existence-weak-solutions-Neumann}), which are a direct consequence of \Cref{t:main}. 
\item  \Cref{sub:introduction-comparison} is dedicated to the comparison principles (\Cref{t:comparison-dirichlet,t:comparison-neumann}), the uniqueness of the weak solutions and the consistency of the elliptic regularization scheme (\Cref{cor:every-weak-solution-is-the-limit-of-u-eps}).
\item In \Cref{sub:about-the-proofs} we briefly discuss the key ideas in the proof of \Cref{t:main}.  
\item \Cref{sub:plan-of-the-paper} contains the plan of the paper.
\end{itemize}

\subsection{Notation}\label{sub:notation}
Throughout this paper, we will always work with sets $D\subset\R^d$ and $E \subset \R^d\times[0,+\infty)$.\medskip

\paragraph{Space-time sets}
We use $E\subset \R^d \times[0,+\infty)$ to indicate a space-time set. Moreover,
\begin{itemize}
	\item $H^1(E)$ is the usual Sobolev space in $\R^{d+1}$; 
	\item for $0\le t_1\le t_2$ we say that $E$ is \emph{cylindrical} if it is of the form $E = D \times (t_1,t_2]$ for some $D\subset \R^d$, and we define the following boundary set:
	\begin{itemize}
		\item the \emph{lateral boundary of $E$} as $\de_L E := \de D \times(t_1,t_2)$;
		\item the \emph{parabolic boundary of $E$} as $\de_P E := \de_L E \cap E(t_1)$.
	\end{itemize}
\end{itemize}

\paragraph{Spatial sets} Given an open set $D\subset\R^d$ we define
\begin{itemize}
	\item $D(t) := D\times\{t\}$ for all $t\in [0,+\infty)$ ;	
	\item  $D_t:= D \times [0,t)$ for all $t\in (0,+\infty]$ ; in particular, $D_\infty:= D \times[0,+\infty)$; 
	\item $H^1(D)$ is the usual Sobolev space in $\R^d$;
	\item $H^1_{0,P}(D_t)$ is the closure of $C^\infty_c(D\times(0,+\infty))$ in $H^1(D_t)$, so that the functions in $H^1_{0,P}(D_t)$ are zero on the parabolic part of the boundary 
	$$\partial_P D_\infty=(D\times\{0\})\cup(\partial D\times(0,+\infty));$$
	\item $H^1_{0,L}(D_t)$ is the closure of $C^\infty_c(D\times\R)$ in $H^1(D_t)$; so the functions in $H^1_{0,L}(D_t)$ are zero on the lateral  part of the boundary 
	$$\partial_LD_\infty=\partial D\times(0,+\infty).$$
\end{itemize}
Moreover, given $(x_0,t_0)\in \R^d\times(0,+\infty)$ and $r>0$, we will use the following notation:
\begin{itemize}
	\item $B_r(x_0) \subset \R^d$ is the ball (in space) of radius $r>0$ centered in $x_0$;
	\item $C_r(x_0) := B_r(x_0) \times[0,+\infty)$ is the infinite space-time cylinder over $B_r(x_0)$.
\end{itemize}

\paragraph{Temperature}
We will call \emph{temperature} a function $u : D \times[0,+\infty) \to \R$, and we write 
\[u = u^+ + u^-,\]
where $u^+\ge 0$ is the positive part of $u$ and $u^-\le 0$ is the  negative part of $u$, and we define the positive and negative phases as
	\[
	\Omega^\pm_u := \big\{(x,t) \in D\times[0,+\infty) : \pm u(x,t) >0 \big\}.
	\] 
With a slight abuse of notation, we define the positive and negative phases at time $t\ge0$ as
\[
    \Omega_u^\pm(t) := \big\{x \in D : \pm u(x,t) >0 \big\}.
\]

\paragraph{Integral and differential operators} 
For any Lebesgue measurable
$$D\subset\R^d\ ,\quad  E \subset D \times[0,+\infty)\quad\text{and}\quad u : D \times[0,+\infty)\to\R,$$ 
we adopt the following notation:
\begin{itemize}
	\item with $|D|$ and $|E|$, we indicate both the $d$ and $(d+1)$ dimensional Lebesgue measure, while we use $\Hc^{d-1}(D)$ for the $(d-1)$-Hausdorff measure of $D\subset\R^d$;
	\item we use all the following convention for the integrals at fixed time, depending on the circumstances
	\[
	\int_{E(t)} u \dif x = \int_{E}u(x,t) \dif x = \int_{E} u \dif x \bigg|_{\tau=t};
	\]
	\item we omit the $D\times[0,+\infty)$ term in the double integrals
	\[
	\iint u \dif x \dif t := \iint_{D\times[0,+\infty)} u (x,t) \dif x \dif t;
	\]
	\item if $D$ is open and $E$ is open in $D\times(0,+\infty)$, $k,\ell \in \NN$, then we write $u\in C^{k,\ell}_{x,t}(E)$ if $u$ is $k$ times differentiable in $D\times(0,+\infty)$ in the space variables, $\ell$ times in the time variable and all these derivatives are continuous; we write $u\in C^{k,\ell}_{x,t}\left(\overline E\right)$ if the derivatives are continuous up to $\partial E$;
	\item we adopt the following conventions
	\[
	D u = D_{x,t} u, \qquad \nabla u = \nabla_x u,\qquad \text{and}\qquad \Delta u = \Delta_x u.
	\]
\end{itemize}

\paragraph{Functional setting} 
Let $D \subset \R^d$ be an open set. We define 
the functional space
\[
\Uc :=
\Big\{
u : D \times [0,+\infty) \to \R \ : \ 
u \in H^1(D_T)\ \text{for all}\ T>0
\Big\}.
\]
We say that a sequence $\{u_j\}_{j\in\NN} \subset \Uc$ \emph{converges weakly to $u \in \Uc$}, and we write
\[
u_j \xrightharpoonup[j\to+\infty]{} u \qquad \text{in}\quad \Uc,
\]
if and only if, for all $T>0$,
\[
u_j \xrightharpoonup[j\to+\infty]{} u \qquad \text{weakly in}\quad H^1(D_T).
\]
Given a function $g\in H^1(D)$, we define the spaces $\Uc_{\mathcal D}$ and $\Uc_{\mathcal N}$ as follows:
\begin{equation}\label{e:def:space-U}
\begin{array}{rl}
	\displaystyle	\Uc_{\mathcal D}(D,g) &:= 
	\Big\{
	u \in \Uc: (u - g) \in H^1_{0,P}(D_t)\ \text{ for all }\ t\in(0,+\infty)
	\Big\},\medskip\\
	\Uc_{\mathcal N}(D,g) &:= 
	\Big\{
	u \in \Uc : u(x,0) = g(x)
	\Big\}.
\end{array}	
\end{equation}
These correspond to the Dirichlet boundary problem (with time-independent boundary data) and to the homogeneous Neumann problem, which, from a physical point of view, 
models adiabatic processes.\medskip

\subsection{Weak (enthalpy) solutions of the Stefan problem - definitions}\label{sub:weak-solutions-definitions}

Throughout the paper we will use the terms {\it weak solution} and {\it enthalpy solution} to the Stefan problem, as well as {\it weak formulation} and {\it enthalpy formulation} of the Stefan problem, as synonyms. The definition of an enthalpy solution that we employ throughout this paper is the following.

\begin{definition}[Enthalpy solution]\label{def:enthalpy_solution}
	Let $D\subset\R^d$ be an open set, $T \in (0,+\infty]$ and $u_M : D \times[0,+\infty) \to \R$ a measurable function.
	A couple $(u,\mu)$ of Lebesgue measurable functions $u:D_T\to\R$ and $\mu:D_T\to\R$ is an \emph{enthalpy solution} (or equivalently \emph{weak solution}) of the Stefan problem in $D_T$ (with respect to the melting temperature $u_M$) if the following hold:
	\begin{enumerate}[{\itshape (i)}]
		\item\label{item:as:regularity} \textit{Regularity of the temperature in time.}
		$u \in L^2_{loc}([0,T);H^1(D))$ and $\de_t u \in L^2_{loc}([0,T);L^2(D))$.
		\item\label{item:as:continuity} \textit{Continuity of the mushy coefficient in time.} The function $\mu:[0,T)\to L^\infty(D)$ is continuous with respect to the weak-$\ast$ topology in $L^\infty(D)$, that is:
		\[
		\int_{D} \mu(x,t) \, \eta(x) \dif x = \lim_{\tau \to t} \int_{D}\mu(x,\tau) \, \eta(x) \dif x,
		\]
		for all $t\in[0,T)$ and all $\eta\in L^1(D)$.
		\item\label{item:as:compatibility}\textit{Compatibility.} For all $t\in[0,T)$, there exists $N_t \subset D$ such that $|D\setminus N_t|=0$ and
		\begin{equation}\label{e:compatibility_enthalpy-phases}\tag{C-ST}
			\mu(x,t) = 
			\begin{cases}
				\begin{array}{rl}
					1 & \text{if}\quad u(x,t)>u_M(x,t),\\
					-1 & \text{if} \quad u(x,t)<u_M(x,t),
				\end{array}
			\end{cases} \qquad\text{for every}\quad x\in N_t.
		\end{equation}
		\item \label{item:as:weakformulation} \textit{Weak formulation.}
		The following integral identity holds
		\begin{equation}\label{e:weak_solution}\tag{H-ST}
			\int_{D(t)}(u+\mu)\eta \dif x \bigg|_{t=t_1}^{t_2} 
			= \int_{t_1}^{t_2}\int_D (u+\mu) \de_t \eta 
			-\nabla u \cdot \nabla \eta \dif x \dif t,
		\end{equation}
		for all $0\le t_1\le t_2 < T$ and all {\it admissible test functions} $\eta$, where:
		\begin{itemize}
			\item $\eta \in C^\infty_c(D\times\R)$ in the Dirichlet case;
			\item $\eta \in C^\infty_c(\R^d\times\R)$ for the Neumann problem.
		\end{itemize}
	\end{enumerate}
\end{definition}

Weak solutions have already been defined and studied in the literature, see for instance \cite{Kamenomostskaya,Oleinik60,Ladyzhenskaya_Solonnikov_Uraltseva68:Parabolic_equation_BOOK,Friedman_1968,CannonDiBenedetto1980:ExistenceStefan,GotzZaltzman:NonincreaseMushyInhomogeneousPb} and the discussion in \Cref{sub:weak-solutions-history}. The main difference in \Cref{def:enthalpy_solution}, with respect to the classical definitions of weak solutions, is in the requirement that the mushy coefficient $\mu$ is defined for every time $t\ge0$ and that the map $t\mapsto\mu(\cdot,t)$ is weakly-$\ast$ continuous. Indeed, the weak formulation \eqref{e:weak_solution} 
only requires that 
$$u \in L^2_{loc}((0,+\infty);H^1(D))\qquad\text{and}\qquad  \mu(\cdot,t) \in L^\infty(D)\quad\text{for all}\quad t\ge0,$$
in which case \eqref{e:weak_solution} holds for almost-every $0\le t_1<t_2<+\infty$. The main reason, for which we develop a theory for weak solutions satisfying the condition (ii) from \Cref{def:enthalpy_solution}, is the role of this  condition in the comparison principles \Cref{t:comparison-dirichlet} and \Cref{t:comparison-neumann}. Precisely, it is well-known that in order to have a comparison principle for weak solutions $u_1, u_2$ of the (two-phase) Stefan problem, it is not sufficient to have ordered initial temperatures $u_1(x,0)\le u_2(x,0)$, but is necessary to have also information about the initial mushy coefficients $\mu_1$ and $\mu_2$ (see for instance  \cite{Ladyzhenskaya_Solonnikov_Uraltseva68:Parabolic_equation_BOOK, CannonDiBenedetto1980:ExistenceStefan} and example \Cref{ex:false-onephase}); essentially, this is due to the fact that the function $\mu$ encodes the information about the evolution of the domains $\{u(\cdot,t)>0\}$ and $\{u(\cdot,t)<0\}$. We will prove the existence of weak solutions (in the sense of \Cref{def:enthalpy_solution}) in \Cref{t:main} via an elliptic regularization scheme, which allows to derive the weak-$\ast$ continuity of the mushy coefficient from the convergence of the approximating sequence.

\begin{remark}[About the melting temperature $u_M$]\label{oss:about-the-drift-term-F}
	Suppose that $(u,\mu)$ us an enthalpy solution in the sense of \Cref{def:enthalpy_solution} with melting temperature $u_M = u_M(x)$ that does not depend on the time variable.
	Suppose that $u_M$ can be written in the form $u_M = u^R_M + u_M^S$, where:
	\begin{equation}\label{e:melting-temperature}
		\nabla u_M^R \in L^2(D;\R^d)\qquad\text{and}\qquad u_M^S : D \longrightarrow \R \quad \text{is measurable}.
	\end{equation}
	Then, setting $w:=u-u_M^R$, we get that the couple  $(w,\mu)$ satisfies
	\[
	\int_{D} (w + \mu) \eta \dif x\bigg|_{t=t_1}^{t_2} = \int_{t_1}^{t_2} \int_D (w + \mu) \de_t \eta - \nabla w \cdot \nabla \eta - \nabla u_M^R \cdot \nabla \eta \dif x\dif t,
	\]
	for all test functions $\eta$ ($\eta \in C^\infty_c(D\times\R)$ in the Dirichlet case and $\eta \in C^\infty_c(\R^d\times\R)$ in the Neumann case).
	Moreover, we can rewrite the compatibility condition for $(u,\mu)$ as follows:
	\[
	\mu(x,t) = 
	\begin{cases}
		\begin{array}{rl}
			1 &\text{if}\quad w(x,t) > u^S_M(x)\\
			-1 &\text{if}\quad w(x,t) < u^S_M(x),
		\end{array}
	\end{cases}\qquad\text{for almost every}\quad x\in D.
	\]
In order to have a theory for a class of problems, which is invariant with respect to this family of transformations, we 
recast \Cref{def:enthalpy_solution} including a heat source term $F \in L^2(D;\R^d)$. In the setting of the previous remark, $F$ is the weak gradient of the "regular part" of $u_M$, that is: $F=\nabla u_M^R$.	
\end{remark}

\begin{definition}[Enthalpy solution with heat source $F$]\label{def:enthalpy_solution-F}
	Let $D\subset \R^d$ an open set and $T\in (0,+\infty]$. Let $u_M: D \to \R$ be a Lebesgue measurable function and $F \in L^2(D;\R^d)$.
	We say that a pair of measurable functions $(u,\mu)$ is a \emph{enthalpy solution of the Stefan problem with heat source} $\dive F$ with Dirichlet (Neumann) boundary conditions on $\partial D$, if the conditions  \ref{item:as:regularity} -- \ref{item:as:compatibility} of \Cref{def:enthalpy_solution} hold and if 
	\begin{equation}\label{e:weak-equation-F}\tag{H-STF}
		\int_{D}(u+\mu)\eta \dif x\bigg|_{t=t_1}^{t_2} = \int_{t_1}^{t_2}\int_D (u+\mu)\de_t\eta - \nabla u \cdot \nabla \eta - F \cdot \nabla \eta \dif x \dif t
	\end{equation}
	for all $0\le t_1 < t_2 < T$ and all admissible test functions $\eta$; as in \Cref{def:enthalpy_solution} we use test functions $\eta \in C^\infty_c(D\times\R)$ in the Dirichlet case and test functions $\eta \in C^\infty_c(\R^d\times\R)$ in the Neumann case.
\end{definition}

\begin{remark}
The presence of a heat source $F$ does not lead to any particular complications in the proof of our main result \Cref{t:main}. On the other hand, solution with such terms  appear naturally in some physical models, so we decided to extend our theory in order to include this wider class of problems.
\end{remark}	

\subsection{Existence of weak solutions via an elliptic regularization scheme}\label{sub:introduction-elliptic}
In this section we define the family of the approximating problems, which will take part in the elliptic regularization scheme. The choice of the {\it two-sided Heaviside approximations} (see \Cref{def:two-sided-heaviside-approximation} here below) will be essential in the proof of our main result \Cref{t:main} in \Cref{sec:convergence_Stefan}. 
The main features of this approximating family are listed in the following definition.
\begin{definition}[Two-sided Heaviside approximations]\label{def:two-sided-heaviside-approximation}
Let $D\subset\R^d$ be an open set, $u_M:D\to\R$ a Lebesgue measurable function, and  $h:D\to\R$ a measurable function satisfying $|h|\le1$ in $D$.	We say that a family of functions $\{p_\eps\}_{\eps>0}$
$$p_\eps:\mathbb R\times D\to\R,$$
is a {\it two-sided Heaviside approximation centered in $(u_M,h)$} if: 
\begin{enumerate}[\rm (1)]
\item for all $(z,x)\in\R\times D$ we have that 
\begin{equation}\label{e:pe-convergence}
\lim_{\eps\to0}	p_{\eps}(z,x) = \ind_{(u_M(x),+\infty)}(z) - \ind_{(-\infty,u_M(x))}(z) + h(x) \ind_{\{u_M(x)\}}(z);
\end{equation}
\item for each $x\in D$, the function $p_\eps(\cdot,x):\R\to\R$ is $C^\infty(\R)$;
\item for all $(z,x)\in\R\times D$ we have that $|p_\eps(z,x)|\le 1$;
\item for each $\eps>0$, there is $L_\eps>0$ such that 
\begin{equation}\label{e:pe-Lipschitz-bound}
|\partial_zp_{\eps}(z,x)|\le L_\eps\quad\text{for all}\quad (z,x)\in \R\times D.
\end{equation}
\end{enumerate}	
\end{definition}
We can now define the family of approximating functionals that we will use in the elliptic regularization scheme. 
Suppose that we have:
\begin{itemize}[--] 
	\item an open set $D\subset\mathbb R^d$;
	\item a measurable function $u_M:D\to\R$, which we call {\it melting temperature};
	\item a function $h\in L^\infty(D)$ with $\|h\|_{L^\infty(D)}\le 1$, which we call {\it initial enthalpy};
	\item a vector field $F\in L^2(D;\R^d)$, which we call {\it heat source}.
\end{itemize} 
Given $D$, $u_M$, $h$, $F$ as above, and a two-sided Heaviside approximation \begin{center}$\{p_\eps:\R\times D\to\R\}_{\eps>0}$ centered at $(u_M,h)$,\end{center} we define the functional 
\[
\Fc_{\eps} : \Uc \longrightarrow \R \cup \{+\infty\}
\]
as
\begin{equation}\label{e:efunctional}
	\Fc_{\eps}(u) 
	:= \iint \frac{e^{-t/\eps^4}}{\eps^4} \Biggl\{ \eps^4 \biggl[ |\de_t u|^2 + |\sqrt{\eps} \,\de_t(p_{\eps}(u,x))|^2 \biggr] + |\nabla u|^2 + 2F \cdot \nabla u \Biggr\} \dif x \dif t,
	\tag{$\Fc_\eps$}
\end{equation}
if the integral converges, and $\Fc_{\eps}(u):=+\infty$ otherwise.

\begin{remark}[Existence of minimizers $u_\eps$]
Notice that for all $u\in \Uc$
 it holds
\[
|\nabla u(x,t)|^2 + 2F(x) \cdot \nabla u(x,t) \ge - |F(x)|^2,
\]
and so $\Fc_\eps(u) \ge -  \|F\|_{L^2(D)}^2$ for all $u\in \Uc$. In particular, $\Fc_\eps(u)\in (-\infty,+\infty]$ for all $u\in \Uc$. Furthermore, in \Cref{prop:existence-J-epsilon} we will show that, for any fixed $g\in H^1(\Omega)$, there are minimizers $u_\eps$ of $\mathcal F$ in both the Dirichlet class $\mathcal U_{\mathcal D}(D,g)$ and the Neumann class $\mathcal U_{\mathcal N}(D,g)$.
\end{remark}

\begin{theorem}[Convergence of the regularization scheme and limit problem]\label{t:main}
	Let $D$ be an open set in $\R^d$, $g\in H^1(D)$, $h\in L^\infty(D)$ with $|h(x)|\le 1$, $u_M : D \to \R$ measurable, and $F \in L^2(D;\R^d)$. Then, there exists an two-sided Heaviside approximation $\{p_\eps\}_{\eps>0}$ centered in $(u_M,h)$ such that the following holds. 
	Given a family of minimizers $u_\eps$ of the functional \eqref{e:efunctional} in $\Uc_{\mathcal D}(D,g)$ (resp. in $\Uc_{\mathcal N}(D,g)$), there exists a sequence $\{\eps_n\}_{n\in\NN}$ such that $\eps_n\to0$ and the following properties hold:
	\begin{enumerate}[label={\itshape(\roman*)}]
		\setlength{\itemsep}{8pt}
		
		\item\label{item:main_thm:convergence} \textbf{Convergence.} There exist two functions,
		\[
		u \in \Uc_{\mathcal D}(D,g) \quad (\text{resp.}\ u \in \Uc_{\mathcal N}(D,g)) \qquad \text{and} \qquad \mu \in L^\infty(D \times [0, +\infty)),
		\]
		representing the \emph{temperature} and the \emph{mushy coefficient} respectively, such that:
		\begin{enumerate}[label={\itshape (i.\alph*)}]    
			\setlength{\itemsep}{8pt}
			
			\item\label{item:main_thm:convergence:a} Convergence of the temperature. $u_{\eps_n} \xrightharpoonup[n\to +\infty]{} u$ weakly in $\Uc$.
			
			\item\label{item:main_thm:convergence:b} Convergence of the mushy coefficient. $p_{\eps_n}(u_{\eps_n},\cdot) \xrightharpoonup[n\to+\infty]{} \mu$ weakly-$\ast$ in $L^\infty(D\times[0,+\infty))$.
			
			\item\label{item:main_thm:convergence:c} Fixed-time convergence. For all $t_0 \ge 0$, we have 
			\begin{equation}\label{e:main_thm:convergence_mu}
				p_{\eps_n}(u_{\eps_n}(\cdot,t_0),\cdot) \xrightharpoonup[n\to+\infty]{\ast} \mu(\cdot,t_0) \qquad\text{weakly-$\ast$ in }\quad L^\infty(D).
			\end{equation}
			In particular, the mushy coefficient $\mu$ is well-defined at every time $t_0$ and $\mu(\cdot,t_0) \in L^\infty(D)$.

		\end{enumerate}
		
		\item \label{item:main_thm:enthalpy} \textbf{Regularity properties of $(u,\mu)$.} \medskip
		\begin{enumerate}[label={\itshape(ii.\alph*)}]
			\setlength{\itemsep}{8pt}
			
			\item\label{item:main_thm:temperature} Energy bound. For all $t\ge0$, $u(\cdot,t) \in H^1(D)$. Moreover $\nabla u(\cdot,t) \in L^\infty((0,+\infty); L^2(D;\R^d))$ and it is bounded by the initial data, that is,
			\begin{equation}\label{e:main_thm:temperature}
				\sup_{t\ge0}\| \nabla u(\cdot,t) \|_{L^2(D)}^2 \le 8\,(\|\nabla g \|_{L^2(D)}^2 + \| F \|_{L^2(D)}^2) \qquad\text{for all}\quad t\ge0.
			\end{equation}
			
			\item\label{item:main_thm:enthalpy:a} Continuity. The function $\mu: [0,+\infty)\to L^\infty(D)$ is continuous with respect to the weak-$\ast$ topology of $L^\infty(D)=(L^1(D))^\ast$; i.e., for all $t_0 \ge 0$ and $\eta \in L^1(D)$, we have
			\[
			\int_{D(t_0)} \mu\,\eta \dif x = \lim_{t \to t_0} \int_{D(t)} \mu \,\eta \dif x.
			\]
			
			\item\label{item:main_thm:enthalpy:b} Compatibility. For all $t \ge 0$, there exists a set $N_t \subset D$ such that  $|D\setminus N_t|=0$,
			\[
			|\mu(x,t)| \le 1 \quad \text{for all}\quad x \in N_t,
			\]
			and
			\begin{equation*}
				\mu(x,t)=
				\begin{cases}
					\begin{array}{rl}
						-1 & \quad \text{for all}\quad x \in N_t\quad\text{such that}\quad u(x,t)< u_M(x),\\
						1  & \quad \text{for all}\quad x \in N_t\quad\text{such that}\quad u(x,t)>u_M(x).
					\end{array}
				\end{cases}
			\end{equation*}
			
		\end{enumerate}
		
			\item\label{item:main_thm:initial-mushy-coefficient} \textbf{Initial mushy coefficient.} The mushy coefficient $\mu$ at time zero is determined by $h$. Precisely, for Lebesgue almost-every $x\in D$ we have
			\[
			\mu(x,0) = \begin{cases}
				\begin{array}{cl}
							1& \text{if}\quad x\in  \Omega_g^+=\{g>0\}\cap D,\\
							-1& \text{if}\quad x\in  \Omega_g^-=\{g<0\}\cap D,\\
				h(x)& \text{if}\quad x \in D\setminus (\Omega_g^+\cup \Omega_g^-).
				\end{array}
				\end{cases}
			\]
		
		\item\label{item:main_thm:limit_eq} \textbf{The limit problem.} The pair $(u,\mu)$ satisfies the integral identity \eqref{e:weak-equation-F} from \Cref{def:enthalpy_solution-F}.
	\end{enumerate}
\end{theorem}

The above \Cref{t:main} contains an existence theorem for weak solutions $(u,\mu)$ of the Stefan problems with Dirichlet and Neumann boundary conditions. In the next two corollaries we give the precise statement in the case $h\equiv0$, which generates solutions with initial mushy coefficient
$$\mu(x,0)=\ind_{\Omega_g^+}(x)-\ind_{\Omega_g^-}(x),$$
which is the one classically associated to the initial temperature $u(0,x)=g(x)$.

\begin{corollary}[Existence of weak solutions with Dirichlet boundary conditions]\label{cor:existence-weak-solutions-Dirichlet}
Let $D$ be an open set in $\R^d$ and let $u_M:D\times[0,+\infty)\to\R$ be a Lebesgue measurable function. Let $g\in H^1(D)$ be a given initial datum and $F\in L^2(D;\R^d)$. Then, there are functions $u\in\mathcal U_{\mathcal D}(U,g)$ and $\mu\in L^\infty(D\times[0,+\infty))$ such that $(u,\mu)$ is an enthalpy solution of the 
Stefan problem (in the sense of \Cref{def:enthalpy_solution-F}) with:
\begin{itemize}[--]
\item Dirichlet boundary conditions on $\partial D$: $u(\cdot,t)-g\in H^1_{0}(D)$ for all $t\ge 0$;
\item initial temperature $u(x,0)=g(x)$;
\item initial mushy coefficient 
$\mu(x,0)=\ind_{\Omega_g^+}(x)-\ind_{\Omega_g^-}(x).$
\end{itemize}
\end{corollary}

\begin{corollary}[Existence of weak solutions with Neumann boundary conditions]\label{cor:existence-weak-solutions-Neumann}
	Let $D$ be an open set in $\R^d$ and let $u_M:D\times[0,+\infty)\to\R$ be a Lebesgue measurable function. Let $g\in H^1(D)$ be a given initial datum and $F\in L^2(D;\R^d)$. Then, there are functions $u\in\mathcal U_{\mathcal N}(U,g)$ and $\mu\in L^\infty(D\times[0,+\infty))$ such that $(u,\mu)$ is an enthalpy solution of the 
	Stefan problem (in the sense of \Cref{def:enthalpy_solution-F}) with:
	\begin{itemize}[--]
			\item Neumann boundary conditions on $\partial D$, that is, $u(\cdot, t)\in H^1(D)$ for all $t\ge 0$ and \eqref{e:weak-equation-F} holds for all test functions $\eta\in C^\infty_c(\R^d\times[0,+\infty))$;
		\item initial temperature $u(x,0)=g(x)$;
		\item initial mushy coefficient 
		$\mu(x,0)=\ind_{\Omega_g^+}(x)-\ind_{\Omega_g^-}(x).$
	\end{itemize}
\end{corollary}

\begin{remark}[On the history of the weak solutions]
The first existence results for weak (distributional) solutions of the Stefan problem were established by Kamin \cite{Kamenomostskaya} (for $d \le 3$) and Oleinik \cite{Oleinik60} (for $d > 3$), who showed 
the existence of functions $(u,\mu)$, defined in bounded smooth domains $D$, satisfying $u(\cdot,t) \in L^\infty(D)$, for all $t \ge 0$, and 
\begin{equation}\label{e:Stefan_russian}
	\int_{D} (u+\mu)\eta \dif x\bigg|_{t=t_1}^{t_2} = \int_{t_1}^{t_2} \int_D \left( (u+\mu) \partial_t \eta + u \Delta \eta \right) \dif x \dif t 
	- \int_{t_1}^{t_2}\int_{\partial{D}} g\,\partial_{\nu_{\partial{D}}} \eta \dif \mathcal{H}^{d-1}x \dif t,
\end{equation}
for all test functions $\eta \in C^\infty_c(D\times\R)$. 
The existence of enthalpy solutions with $u\in H^1$ are due to Ladyzhenskaya, Solonnikov, and Ural’tseva \cite[Ch.~V, \S~9]{Ladyzhenskaya_Solonnikov_Uraltseva68:Parabolic_equation_BOOK}, to Friedman \cite{Friedman_1968} (under some further assumptions on the boundary data), and later to Cannon and DiBenedetto \cite{CannonDiBenedetto1980:ExistenceStefan}. In all \cite{Ladyzhenskaya_Solonnikov_Uraltseva68:Parabolic_equation_BOOK,Friedman_1968,CannonDiBenedetto1980:ExistenceStefan} the authors considered $\mu = \mu(u)$ as a multivalued function of the temperature, and therefore the existence of weak solutions is obtained under the condition that the zero set is empty at the initial time ($|\{g=0\}| =0$) in which case the initial mushy coefficient is determined almost-everywhere from the compatibility condition.
In the early 90s, G\"otz and Zaltzman \cite{GotzZaltzman:NonincreaseMushyInhomogeneousPb} (see also a simplified proof by Andreucci \cite{Andreucci2004:StefanNotes}) managed to prove an existence theorem with possibly non-empty mushy region, but required the initial temperature in $L^\infty(D)$. In \Cref{t:main}
 we not only prove the existence of $(u,\mu)$ satisfying an integral equation, but we show that the mushy coefficient $\mu$ is weakly-$\ast$ continuous in time and is well defined at time $0$, which is fundamental for the validity of the comparison principle as we will see in the next subsection. 
\end{remark}

\begin{remark}[On the $H^1$ regularity of the temperature] In \Cref{t:main}, our regularity assumption over the initial temperature ($g\in H^1(D)$) is stronger than the one in \cite{Ladyzhenskaya_Solonnikov_Uraltseva68:Parabolic_equation_BOOK,CannonDiBenedetto1980:ExistenceStefan} (where the assume $g\in L^2(D)$) and the one in \cite{GotzZaltzman:NonincreaseMushyInhomogeneousPb} (where $g\in L^\infty(D)$), but leads to a stronger regularity result. Indeed, \Cref{t:main} provides 
\begin{equation}\label{e:intro:L-infty-in-time-L-2-in-space}
	{\nabla u \in L^\infty((0,+\infty);L^2(D;\R^d))},
\end{equation}
while in the classical theory of weak solutions for the Stefan problem we only had 
$$\nabla u \in L^2_{loc}([0,+\infty);L^2(D;\R^d)).$$ 
%
Regularity results under stricter assumptions on the initial data were also obtained by Had{\v z}i{\'c} and Shkoller in \cite{HadzicShkoller2017:Gibbs-ThomsonStefan} or by Ding, Du, and Guo in \cite{DingDuGuo2021:StefanProblemFisher-unbounded}.
However, up to our knowledge, the estimate \eqref{e:intro:L-infty-in-time-L-2-in-space} is new for unbounded sets $D\subset\R^d$. We also notice that \eqref{e:intro:L-infty-in-time-L-2-in-space} is optimal since it is the precisely
the regularity of the solutions to heat equation.
\end{remark}

\begin{remark}[Solutions with variable melting temperature]
In \Cref{cor:existence-weak-solutions-Dirichlet} and \Cref{cor:existence-weak-solutions-Neumann} we establish  existence results for solutions with general measurable melting temperatures $u_M$, which as far as we know are new in the literature. We notice that, if the melting temperature is sufficiently regular, $u_M\in H^1(D)$, this can be recovered from the result of Cannon and DiBenedetto for equations with right-hand side. 
\end{remark}

\subsection{Comparison principles and consistency of the regularization scheme}\label{sub:introduction-comparison}
This section is dedicated to the comparison principles for weak solutions in the classes $\mathcal U_{\mathcal D}(D,g)$ and $\mathcal U_{\mathcal N}(D,g)$.  Before we state our main theorems we define the notions of enthalpy  subsolutions and enthalpy supersolutions.

\begin{definition}[Dirichlet and Neumann enthalpy sub/supersolutions]\label{def:enthalpy-sub-supersol}
	Let $T^\ast \in (0,+\infty]$, $D\subset\R^d$ be an open set, $u_M : D \to\R$ be a Lebesgue measurable function, and $F \in L^2(D;\R^d)$. 
	Let also $u$ and $\mu$ be measurable functions in $D\times(0,+\infty)$.
	We say that a couple $(u,\mu)$ is an \emph{enthalpy subsolution} (resp. \emph{supersolution}) \emph{with melting temperature $u_M$, heat source $F$ and Dirichlet/Neumann boundary conditions} if it satisfies the conditions \ref{item:as:regularity}-\ref{item:as:compatibility} of \Cref{def:enthalpy_solution} and if the following integral inequality holds
	\begin{equation}\label{e:subsolution-enthalpy}
		\int_{D(t)} (u+\mu) \eta \dif x\biggl|_{t=t_1}^{t_2} \le \int_{t_1}^{t_2}\int_{D} (u+\mu) \de_t\eta - \nabla u \cdot \nabla \eta - F\cdot \nabla \eta \dif x \dif t\,,\qquad (\text{resp.}\; \ge)
	\end{equation}
	for all $0\le t_1\le t_2< T^\ast$ and all nonnegative  admissible test functions $\eta\ge 0$,  where:
	\begin{itemize}
		\item $\eta \in C^\infty_c(D\times\R)$ in the Dirichlet case;
		\item $\eta \in C^\infty_c(\R^d\times\R)$ in the Neumann case.
	\end{itemize}
\end{definition}

\begin{theorem}[Comparison principle for the Stefan problem in the Dirichlet case]\label{t:comparison-dirichlet}
	Let $T^\ast\in(0,+\infty]$, $D\subset\R^d$ be an open set, $u_M :D \to \R$ a Lebesgue measurable function, and $F \in L^2(D;\R^d)$. 
	Let $(u_1,\mu_1)$ and $(u_2,\mu_2)$ be respectively a Dirichlet enthalpy subsolution and a Dirichlet enthalpy supersolution  in $D_{T^\ast}$ with melting temperature $u_M$ and heat source $F$. Assume that
	\[
	(u_2- u_1)^- \in H^1_{0,P}(D_{T^*}) \quad \text{and} \quad \mu_1(x,0) \le \mu_2(x,0) \quad \text{for almost every} \quad x\in D.
	\]
	Then, 
	for all $T\in(0,T^\ast)$, we have:
	\[
	u_1(x,T)\le u_2(x,T) 
	\quad\text{and}\quad
	\mu_1(x,T) \le \mu_2(x,T)\quad \text{for almost every}\quad x\in D.
	\]
\end{theorem}

\begin{theorem}[Comparison principle for the Stefan problem in the Neumann case]\label{t:comparison-neumann}
	Let $D\subset\R^d$ be an open with boundary $\de D$, which is locally $C^{1,1}$ regular. Let $T^\ast\in(0,+\infty]$, $u_M :D \to \R$ be a Lebesgue measurable function, and $F \in L^2(D;\R^d)$. 
	Let $(u_1,\mu_1)$ and $(u_2,\mu_2)$ be a Neumann enthalpy subsolution and a Neumann enthalpy supersolution in $D_{T^\ast}$ with melting temperature $u_M$ and heat source $F$. Furthermore, assume that
	\[
	u_1(x,0) \le u_2(x,0) \quad \text{and} \quad \mu_1(x,0) \le \mu_2(x,0) \quad \text{for almost every} \quad x\in D.
	\]
	Then, for all $T\in(0,T^\ast)$, we have
	\[
	u_1(x,T)\le u_2(x,T) 
	\quad\text{and}\quad
	\mu_1(x,T) \le \mu_2(x,T)\quad \text{for almost every}\quad x\in D.
	\]
\end{theorem}

As a consequence of the comparison principle we obtain that every (Dirichlet or Neumann) weak solution $(u,\mu)$ can be obtained through the elliptic regularization scheme from \Cref{t:main}. In particular, this also means that the solutions $(u,\mu)$ obtained via the elliptic regularization scheme do not depend on the sequence $(\eps_n)_{n\ge 1}$ from \Cref{t:main}. 

\begin{corollary}[Consistency of the elliptic regularization scheme]\label{cor:every-weak-solution-is-the-limit-of-u-eps}
	Let $D\subset\R^d$ be an open set, ${u_M : D \to \R}$ be a Lebesgue measurable function, $F\in L^2(D;\R^d)$, and $h\in L^\infty(D)$ with $|h(x)|\le 1$. Then, for every  $g\in H^1(D)$, there are unique enthalpy solutions $(u,\mu)$ of the Stefan problem in the Dirichlet class ($u\in\Uc_{\mathcal D}(D,g)$) and  in the Neumann class ($u\in\Uc_{\mathcal N}(D,g)$) with:
	\begin{itemize}
	\item[--] melting temperature $u_M$;
	\item[--] heat source $F$;
	\item[--] initial temperature $u(x,0)=g(x)$;
	\item[--] initial enthalpy 	$\mu(x,0) =\ind_{\Omega_g^+}(x)- \ind_{\Omega_g^-}(x)+h(x)\ind_{D\setminus (\Omega_g^+\cup \Omega_g^-)}(x).$
	\end{itemize}
	 Furthermore, if $\{u_\eps\}_{\eps>0}$ is a family of minimizers of \eqref{e:efunctional} in $\Uc_{\mathcal D}(D,g)$ (respectively in $\Uc_{\mathcal N}(D,g)$), then:
	\begin{enumerate}[(i)]
		\item $u_{\eps} \xrightharpoonup[\eps\to 0]{} u$ weakly in $\Uc$;
	\item  $p_{\eps}(u_{\eps},\cdot) \xrightharpoonup[n
	\eps\to0]{} \mu$ weakly-$\ast$ in $L^\infty(D\times[0,+\infty))$;
	\item $p_{\eps}(u_{\eps}(\cdot,t_0),\cdot) \xrightharpoonup[\eps\to0]{\ast} \mu(\cdot,t_0)$  weakly-$\ast$ in $L^\infty(D)$, for all $t_0 \ge 0$.
	\end{enumerate}
\end{corollary}

We notice that the weak-$\ast$ continuity of the mushy coefficient in zero is essential for the validity of the comparison principle. It is indeed the mushy coefficient which encodes the information about the evolution of the free boundary and allows to distinguish solutions with the same initial temperature, but with qualitatively different dynamics, as we point out in the following example. 
\begin{example}\label{ex:false-onephase}
First notice that, by \cite{Park2026:radial3Dsharp}, there are radial solutions $u_1:\R^2\times[0,+\infty)\to\R$  of the Stefan problem  with $\{u_1(\cdot,t)>0\}=B_{R(t)}$, where: 
	\begin{itemize}
		\item $R:[0,+\infty)\to(0,+\infty)$ is positive, bounded and strictly increasing;
		\item $u_1:D\times[0,+\infty)\to\R$  is non-negative and radial;
		\item $u_1$ is a classical solution of the one-phase Stefan problem, that is: 
		$$\begin{cases}
			\partial_tu_1(t,x)=\Delta u_1(x,t)\quad\text{for all}\quad (x,t)\quad\text{such that}\quad |x|<R(t),\\
			R'(t)=|\nabla u_1(R(t),t)|\quad\text{for all}\quad t>0.
		\end{cases}$$ 
	\end{itemize}
	In particular, setting 
	$$\mu_1(x,t)=\begin{cases} 
		1\quad\text{if}\quad |x|< R(t),\\
		0\quad\text{if}\quad |x|\ge R(t),
	\end{cases}
	$$
	we get a couple $(u_1,\mu_1)$, which is an enthalpy solution of the Stefan problem in $\R^2$ with $u_M=h=0$.

On the other hand, let $u_2:\R^2\times[0,+\infty)$ be the solution to the heat equation in $\R^2$ with initial datum $u_2(x,0)\le u_1(x,0)$. Then, we have 
	\[
	\int_{\R^2} u_2(x,t) \, \eta(x,t) \dif x\bigg|_{t_1}^{t_2} = \int_{t_1}^{t_2} u_2\,\de_t\eta - \nabla u_2 \cdot\nabla \eta \dif x \dif t,
	\]
	for all test functions $\eta\in C^\infty_c(\R^2\times[0,+\infty))$ and all $0\le t_1<t_2<+\infty$. Since, 
	\[
	\int_{\R^2} \eta(x,t) \dif x\bigg|_{t_1}^{t_2} = \int_{t_1}^{t_2} \de_t\eta(x,t) \dif x \dif t,
	\]
	taking as mushy coefficient $\mu_2\equiv1$ in $\R^2\times[0,+\infty)$ and summing up the two equations, we get that the couple $(u_2,\mu_2)$ is a solution of the Stefan problem \eqref{e:weak_solution}. We notice that, since $u_2$ is a solution of the heat equation in $\R^2$, the free interface $D\cap\partial\{u_2(\cdot,t)>0\}$ disappears instantly. In particular, we have that the temperatures $u_1$ and $u_2$ are ordered ($u_1\ge u_2$) at time $t=0$, but not at times $t>0$. This is not in contradiction with the comparison principle \Cref{t:comparison-dirichlet} since $\mu_1\le \mu_2$ at time zero.
\end{example}

\subsection{About the proof of the main theorem}\label{sub:about-the-proofs}

In the proof of \Cref{t:main} we will need a specific control on the behavior of the approximating functions $p_\eps$ up to their second derivatives. We will define the two-sided Heaviside approximation $p_\eps$ in \Cref{sec:functional} starting from general functions $u_M:D\to\R$ and $h\in L^\infty(D)$ with $|h|\le 1$. The key properties of our approximating functions are the following:
\begin{itemize}
	\item {\it Rate of convergence.} For every $\eps>0$, we have 
	\[
	|p_{\eps}(u_M(x),x) - h(x)|\le \eps\quad\text{for all}\quad x\in D,
	\]
	and
	$$p_{\eps}(u_M(x),x) = h(x)\quad\text{for all}\quad x\in D\quad\text{such that}\quad |h(x)|\le 1-\eps.$$
	
	\item  {\it First order estimates.} The derivative is of order $\eps^{-1}$, that is:
	\[
	\de_z p_\eps(z,x) \sim \frac1\eps \quad \text{if } \big|\big(z+\eps h(x)\big)-u_M(x)\big| \le \eps, \qquad \text{and} \qquad \de_z p_\eps(z,x) \sim 0 \quad \text{otherwise}.
	\]
	\item {\it Second order estimates.}  The function $z\mapsto p_{\eps}(z,x)$ is piecewise affine, except for two intervals of size $\eps^2$ each. In particular, the following second order estimates hold:
	\[
	|\de_{zz} p_{\eps}(z,x)| \sim \frac{1}{\eps^3} \quad \text{if } \big|\big|\big(z+\eps h(x)\big)-u_M(x)\big|-\eps\big|\le \eps^2, 
	\qquad \text{and} \qquad
	|\de_{zz} p_{\eps}(z,x)| \sim 0 \quad \text{otherwise}.
	\]
\end{itemize}

\noindent The approximating solutions $u_\eps$ in the elliptic regularization scheme are minimizers of the functional
	\begin{equation*}
		\Fc_{\eps}(u) 
		:= \iint \frac{e^{-t/\eps^4}}{\eps^4} \Biggl\{ \eps^4 \biggl[ |\de_t u|^2 + |\sqrt{\eps} \,\de_t(p_{\eps}(u,x))|^2 \biggr] + |\nabla u|^2 + 2F \cdot \nabla u \Biggr\} \dif x \dif t. 
	\end{equation*}
The weak $H^1$ convergence of the temperatures $u_\eps$ to a function $u$ follows by a uniform $H^1$ bound obtained via the argument of Serra and Tilli from \cite{Serra_Tilli_2012:wave} (see \Cref{sec:energy-estimates}). The exponential term $e^{-\sfrac{t}{\eps^4}}$ in $\mathcal F_\eps$ forces the elliptic problems to become parabolic in the limit and, arguing as in \cite{Serra_Tilli_2012:wave}, it is easy to show that $u$ satisfies the heat equation where the temperature $u$ is strictly below or strictly above the melting temperature $u_M$, that is:
$$\partial_tu=\Delta u+\text{div} F\quad\text{in}\quad\{u\neq u_M\}.$$
The conditions satisfied by $u$ on the free boundary $\partial\{u\neq u_M\}$, are encoded in the mushy coefficient $\mu(x,t)$, which we will obtain as a limit of the functions $$\mu_\eps(x,t):=p_\eps(u_\eps(x,t),x).$$ 
In fact, by construction, the family  $\{\mu_\eps\}_{\eps>0}$ is uniformly bounded in $L^\infty(D\times(0,+\infty))$, so  (up to subsequences) there exists $\mu \in L^\infty(D_\infty)$ such that 
\[
\mu_\eps \xrightharpoonup[\eps\to0]{} \mu \quad \text{weakly-$\ast$ in}\quad  L^\infty (D\times(0,T)),
\]
for all $T>0$. Of course, this information by itself is not enough to guarantee the continuity of $\mu$ in time. It is also clear that, it is not possible to obtain a stronger convergence of $\mu_\eps$ via uniform $H^1$ estimates in the spirit of \cite{Serra_Tilli_2012:wave} since the solutions of the Stefan problem typically have mushy coefficients $\mu$, which are characteristic functions (which means that they cannot be in $H^1$).\medskip

The convergence of $\mu_\eps$ and the properties of the limit function $\mu$ are the core of the proof of \Cref{t:main}, and are also where the specific choice of $p_\eps$ and the functional $\mathcal F_\eps$ come into play. 
The term $|\sqrt{\eps} \,\de_t(p_{\eps}(u,x))|^2$ in the functional $\mathcal F_\eps$ has a special role and is formally the one that allows to recover (in a weak form) the following set of Stefan conditions at the limit: 
\begin{equation*}
	\begin{cases}
		\begin{array}{rl}
	\partial_tu_+=\frac{1}{1-h(x)}|\nabla u_+|^2&\text{on}\quad\partial\{u>u_M\}\setminus \partial\{u<u_M\}\cap D,\medskip\\
		\partial_tu_-=\frac{1}{h(x)-1}|\nabla u_-|^2&\text{on}\quad\partial\{u<u_M\}\setminus \partial\{u>u_M\}\cap D,\medskip\\
			\partial_tu_++\partial_tu_-=\frac{1}2\Big(|\nabla u_+|^2-|\nabla u_-|^2\Big)&\text{on}\quad\partial\{u>u_M\}\cap \partial\{u<u_M\}\cap D.
			\end{array}
\end{cases}
\end{equation*}
Precisely, thanks to the factor $\sqrt{\eps}$, the outer variation of \eqref{e:efunctional} (see \Cref{lemma:outer-variation-epsilon}) can be written as
\[
\iint \Big(u_\eps + p_\eps(u_\eps,x)\Big) \de_t \eta - \nabla u_\eps \cdot \nabla \eta - F\cdot\nabla \eta \dif x \dif t = \Oc(\eps),
\]
for any fixed test function $\eta \in C^\infty_c(D\times(0,+\infty))$. Passing formally to the limit as $\eps\to0$, we get exactly the weak formulation of the Stefan problem \eqref{e:weak-equation-F} (we refer to \Cref{sec:convergence_Stefan} for the complete argument).

Proving the validity of the integral equation \eqref{e:weak-equation-F} does not complete the analysis as  (see \Cref{def:enthalpy_solution},  \Cref{def:enthalpy_solution-F} and the discussion in \Cref{sub:weak-solutions-definitions}) we also need to prove that $\mu(\cdot,t)$ is well-defined for every $t \ge 0$ and that the function $t\mapsto \mu(\cdot,t)$ is weakly-$\ast$ continuous as an $L^\infty(D)$-valued function. both of the above properties are not guaranteed by the weak-$\ast$ convergence in space-time. 
This issue does not appear in most of the elliptic regularization schemes present in the literature but is characteristic for evolution problems with free boundary, as recently observed in \cite{AudritoSanzPerela2025:singular-elliptic-regularization}. 
Notice that, differently from \cite{AudritoSanzPerela2025:singular-elliptic-regularization}, the free boundary of solutions of the Stefan problem does not satisfy any density estimates (see for instance  \cite{KingLaceyVasquez1995:AnglesHeleShaw,Athanasopoulos_Caffarelli_Salsa_1996a,Athanasopoulos_Caffarelli_Salsa_1996b,CaffarelliSalsa:GeomApproachToFreeBoundary,ChoiKim2006:WaitingTime}); so we develop a different approach.

The idea of the proof of the convergence of  $\mu_\eps(\cdot,t)$ is to formally test the equation for $u_\eps$ with test functions of the form $\eta \ind_{[t_1,t_2]}$ with $\eta \in C^\infty_c(D\times\R)$.
Unfortunately, this cannot be done directly as it would require a control over the $H^1(D)$-norms of the traces $u_\eps(\cdot,t)$, which we do not have.
In order to avoid the necessity of such stronger estimates on $u_\eps$, in the proof of \Cref{t:main}, we replace 
$\ind_{[t_1,t_2]}$ with a smooth $\eps$-approximation $\ind^\eps_{[t_1,t_2]}$, defined in such a way that
\begin{equation}\label{e:interval-introduction}
	\de_t \ind^{\eps}_{[t_1,t_2]} \sim \frac{1}{\eps^6}
	\left(
	\ind_{[t_1,t_1+\eps^6]} - \ind_{[t_2-\eps^6,t_2]}
	\right).
\end{equation}
The specific choice of the scale $\eps^6$ allows to control the time derivatives of $\mu_\eps$. Using this control, we are able to define $\mu(\cdot,t_0)$ as a weak-$\ast$ limit of $\mu_\eps(\cdot, t_0)$ for every time $t_0\ge0$.
More precisely, we first show that $\mu(\cdot,t_0)$ is the weak-$\ast$ limit of (right) time-averages of $p_\eps(u_\eps(x,t),x)$, that is,
\begin{align}
	\int_{D} \mu(\cdot, t_0) \,\eta(x,t) \dif x
	&=
	\lim_{\eps\to0}\;\frac{1}{\eps^6} \int_{t_0}^{t_0+\eps^6} \int_{D} p_\eps(u_\eps(x,t),x)\,\eta(x,t)  \dif x \dif t\\
&	= 
	\lim_{\eps\to0} \int_{t_0}^{t_0+\eps^6} \int_{D} p_\eps(u_\eps(x,t),x)\,\eta(x,t) \, \de_{t}\ind_{[t_0,+\infty)}^\eps(t) \dif x \dif t,
\end{align}
for all test functions $\eta$ in both $C^\infty_c(D\times\R)$ and $C^\infty_c(D)$.
At this point, we use the $H^1$-regularity of $u_\eps$ to freeze the above mean value at level $t_0$, finding that
\[
p_\eps(u_\eps(\cdot,t_0),\cdot)  \xrightharpoonup[\eps\to0]{\ast} \mu(\cdot,t_0) \qquad \text{weakly-$\ast$ in} \quad L^\infty(D)\quad \text{for all}\quad t_0\ge0.
\]
Finally, for the continuity of $\mu$ and the compatibility conditions \ref{item:as:compatibility} in \Cref{def:enthalpy_solution} , we use again \eqref{e:interval-introduction} together with the (uniform) $\sfrac12-$H\"older continuity of the traces $u_\eps(\cdot,t)$ obtained thanks to the uniform $H^1$ bounds on $u_\eps$ in space-time.

\begin{remark}[On the elliptic regularization scheme]
The existence of weak solutions via elliptic regularization has been investigated in the context of different hyperbolic and parabolic problems. The first results go back to works of J.L. Lions \cite{LionsJL65:elliptic-regularization,LionsMagenes68:BookVol1} and Ole{i}nik \cite{Oleinik64:elliptic-regularization} and since then the method has been applied to numerous parabolic problems (see for instance \cite{AudritoSerraTilli:SegregatedSolutions} and the references therein). In the hyperbolic setting, the elliptic regularization scheme was used by Serra and Tilli in \cite{Serra_Tilli_2012:wave}, where they proved a conjecture of De Giorgi's \cite{DeGiorgi1996:ConjecturesEvolution} inspired by the work of Ilmanen. In the setting of geometric evolution and free boundary problems such schemes have been used for the mean curvature flow \cite{Ilmanen1994:EllipticRegularization}, in segregation problems \cite{AudritoSerraTilli:SegregatedSolutions}, and free boundary evolution problems \cite{AudritoSanzPerela2025:singular-elliptic-regularization}. The Stefan problems combine the parabolic nature of the temperature with the hyperbolic law of the evolution of the free boundary. Our family of elliptic functionals takes into account this mixed character of the Stefan problem and, in particular, provides a method for the elliptic regularization of free boundary problems, in which the evolution of the free boundary is determined by a transport-like equation.
\end{remark}

\subsection{Plan of the paper}\label{sub:plan-of-the-paper}
In \Cref{sec:functional} we introduce the functionals involved in the elliptic regularization scheme. In \Cref{sec:energy-estimates} we prove the main $H^1$-energy estimates, which allow us to prove the temperature's convergence and regularity. 
In \Cref{sec:convergence_Stefan} we prove our main result  (\Cref{t:main}): the convergence of the approximating sequence and the \emph{mushy coefficients}.
In \Cref{sec:initial_enthalpy} we prove the comparison principles (\Cref{t:comparison-dirichlet} and \Cref{t:comparison-neumann}), which allow to prove the uniqueness of the solutions and the consistency of the elliptic regularization scheme (\Cref{cor:every-weak-solution-is-the-limit-of-u-eps}).
\Cref{sec:history} contains a brief discussion on the history of the Stefan problem: the original classical formulations (\Cref{sub:history-classical-formulation}), the development of the notion of a weak solution (\Cref{sub:weak-solutions-history}), a physical interpretation and a brief discussion on the mushy region (\Cref{sub:mushy-physical}). Finally, \Cref{sub:new-classical} is dedicated to the classical formulation associated to the weak solutions from \Cref{sub:weak-solutions-definitions}.

\section{Two-sided Heaviside approximation and approximating problems}\label{sec:functional}

Let $D\subset \R^d$ be an open set, $g \in H^1(D)$, $h\in L^\infty(D)$ with $|h(x)|\le1$, $F \in L^2(D;\R^d)$, $u_M : D\to \R$ measurable, $\Uc$, $\Uc_{\mathcal D}(D,g)$, $\Uc_{\mathcal N}(D,g)$, and \eqref{e:efunctional} as defined in \Cref{sub:introduction-elliptic}.
In this section we discuss the properties of $p_{\eps}(z,x)$ and we solve the minimization problem associated to \eqref{e:efunctional}.

\subsection{The definition of \texorpdfstring{$p_{\eps}$}{pe}}\label{sub:definition-p-f-k-epsilon}
Let $p : \R \to \R$ be defined as follows
\[
p(z) := 
\begin{cases}
    \begin{array}{ll}
        1 & \text{if}\quad z \ge 1 \\
        z & \text{if} \quad |z| \le 1\\
        -1 & \text{if}\quad z \le -1,
    \end{array}
\end{cases}
\]
and let $\rho\in C^\infty_c(\R)$ be a smooth mollifier, which is nonnegative, even, supported in the interval $(-1,1)$, and such that $\int_{\R}\rho(x)\,dx=1$. 
Then, for every $\eps,\delta>0$, and $h\in L^\infty(D)$, with $|h(x)|\le 1$, we define the rescaled functions 
\[
    \rho_\delta:\RR\to\RR\ ,\qquad \rho_\delta(z):= \frac{1}{\delta}\rho\left(\frac{z}{\delta}\right);
\]
%
\begin{equation}\label{e:p-eps-delta}
    p_{{\eps,\delta}}:\RR\times D\longrightarrow\RR\ ,\qquad  p_{{\eps,\delta}}(z,x) := (\rho_{\delta} \ast p)\big(\eps^{-1}(z-u_M(x))+h(x)\big).
\end{equation}
At the end of this section, we will impose $\delta =\eps$ and define $p_{\eps}$ as $p_{\eps, \eps}$. 
We nevertheless introduce $p_{\eps,\delta}$ for $\delta>0$, independent of $\eps>0$, to underline where and how the rescaling in the mollifier influences \eqref{e:efunctional}. 

\begin{center}
    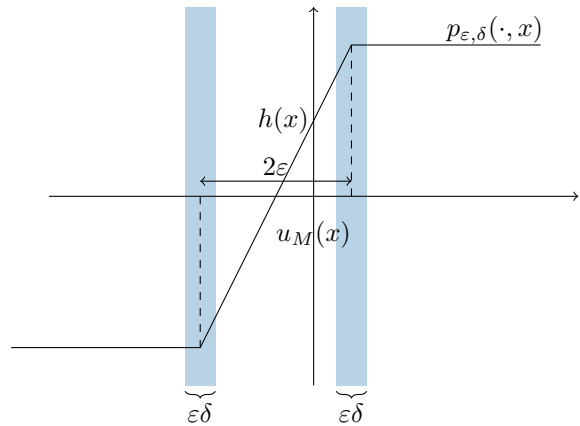
\begin{figure}[ht]
	\begin{tikzpicture}[use Hobby shortcut,scale=1]
    \begin{scope}[shift={(-0.5,0)}]
        \draw[decoration={brace,mirror,raise=5pt},decorate,]
        (-1.2,-2.4) -- node[below=6pt] {$\eps\delta$} (-0.8,-2.4);
    \draw[decoration={brace,mirror,raise=5pt},decorate,]
        (0.8,-2.4) -- node[below=6pt] {$\eps\delta$} (1.2,-2.4);

        \clip (-3.5,-2.5) rectangle (3.5,2.5);

    \draw[fill = MidnightBlue!20, draw opacity = 0] (-1.2,-5) rectangle (-0.8,5);

    \draw[fill = MidnightBlue!20, draw opacity = 0] (0.8,-5) rectangle (1.2,5);

    \draw (-3.5,-2) -- (-1,-2) -- (1,2) -- (3.5,2);

    \draw[dashed] (-1,0) -- (-1,-2);
    \draw[dashed] (1,0) -- (1,2);

    \node at (0.5,-0.5) {$u_M(x)$};
    \draw[<->] (-1,0.2) -- (1,0.2);
    \node at (0,0.4) {$2\eps$};
    \node at (0.1,1) {$h(x)$};
    \node at (2.9,2.2) {$p_{\eps,\delta}(\cdot,x)$};

\end{scope}
    \draw[->] (-3.5,0) -- (3.5,0);
    \draw[->] (0,-2.5) -- (0,2.5);

    \end{tikzpicture}
         \caption{A graphical representation of the two-sided approximated Heavyside function at scale $\eps,\delta$.}
        \label{fig:ex:p-eps}
    \end{figure}
\end{center}

\begin{remark}
If $|h(x)|\le 1$ on $D$, we have 
    \begin{equation}\label{e:distance-h-p-eps-delta}
        |p_{\eps,\delta}(u_M(x),x) -h(x)| \le \delta.
    \end{equation}
Indeed, we have $p_{\eps,\delta}(u_M(x),x) = (\rho_\delta\ast p)(h(x)),$
    and, since $p(z) = z$ and $|\de_zp|\le1$ for $|z|\le1$ , we get precisely \eqref{e:distance-h-p-eps-delta}.    
    Moreover, if $|h(x)| \le 1-\delta$ on $D$, then by construction
        $$p_{\eps,\delta}(u_M(x),x) = h(x).$$
\end{remark}

In order to study the approximation scheme, we need an exact control on the derivatives of $p_{\eps,\delta}$. 
Let us define the re-normalized derivatives of $p_\eps$ as
\[
f_{\eps,\delta}(z) := \eps \de_zp_{\eps,\delta} (z),\qquad\text{and}\qquad
k_{\eps,\delta}(z) := (\eps^2\delta)\,\de_{zz} p_{{\eps,\delta}}(z).
\]
\begin{lemma}[Estimates on $f_{{\eps,\delta}}$ and $k_{\eps,\delta}$]\label{lemma:estimates-f-k-epsilon}
Let $p$, $\rho$, $\rho_\delta$, $p_{\eps,\delta}$, $f_{\eps,\delta}$, and $k_{\eps,\delta}$ be as above. Then,

\[
    \|f_{\eps,\delta}\|_{L^\infty(\R)}\le 1\qquad\text{and}\qquad \|k_{\eps,\delta}\|_{L^\infty(\R)}\le \|\rho'\|_{L^1(\R)}.
\]

\end{lemma}

\begin{proof} Without loss of generality, we assume $u_M(x)=0$.
By the definition of $p_{\eps,\delta}$, for all $x\in D$ it holds
\[
    p_{\eps,\delta}(z,x) :=\int_{\R} \rho_{\delta}\big(\big(\eps^{-1}z+h(x)\big)-y\big) \,p(y)\dif y
        =\int_{\R} \rho_\delta(y)\,p\big(\big(\eps^{-1}z+h(x)\big)-y\big)\dif y.
\]
Taking the derivative in $z$, we get 
\[
    \de_z p_{\eps,\delta}(z,x)
        =\frac{1}{\eps}\int_{\R} \rho_\delta(y)\, \de_z p \big(\big(\eps^{-1}z+h(x)\big)-y\big) \dif y
        =\frac{1}{\eps}(\rho_{\delta} \ast \de_z p) \big(\eps^{-1}z+h(x)\big),
\]
and so, since $|\de_z p|\le 1$, 
we get that $\|f_{\eps,\delta}\|_{L^\infty(\R)}\le 1$.
In order to prove the bound on $k_{\eps,\delta}$, we write  
\begin{align*}
    \de_zp_{\eps,\delta}(z,x)
        &=\frac{1}{\eps}\int_{\R} \frac1{\delta}\rho\left(\frac{\big(\eps^{-1}z+h(x)\big)-y}{\delta}\right) \de_zp(y) \dif y,
\end{align*}
and we use this formula to compute the second derivatives of $p_{\eps,\delta}$ in $z$ as follows:
\begin{align*}
\de_{zz} p_{\eps,\delta}(z,x)
    =\frac{1}{\eps^2\delta}\int_{\R} \frac{1}{\delta}\rho'\left(\frac{\big(\eps^{-1}z+h(x)\big)-y}{\delta}\right) \de_zp(y)\dif y
    =\frac{1}{\eps^{2}\delta}((\rho')_{\delta} \ast \de_z p)\big(\eps^{-1}z+h(x)\big),
\end{align*}
where $(\rho')_{\delta}(z):=\delta^{-1}\rho'(\delta^{-1}z)$.
Finally, since $\|(\rho')_{\delta}\|_{L^1} = \|\rho' \|_{L^1}$ we get
\[
    |\de_{zz}p_{\eps,\delta}(z,x)| 
        \le \frac{1}{\eps^2\delta}\|(\rho')_\delta\|_{L^1(\R)}
        = \frac{1}{\eps^2\delta}\|\rho'\|_{L^1(\R)}
    \qquad\text{for all}\quad (z,x)\in \R\times D,
\]
and this conclude the proof.
\end{proof}
In the proof of \Cref{t:main}, a key role is played by the function
\begin{equation}\label{e:definition-q-epsilon}
    q_{\eps,\delta} : \R\times D \to \R, \qquad     
    q_{\eps,\delta}(z,x) := p_{\eps,\delta}(u_M(x),x) + \int_{u_M(x)}^z \frac{1}{\eps} \big( f_{\eps,\delta} \big)^2 (\zeta) \dif \zeta,
\end{equation}
as well as its relationship with $p_{\eps,\delta}$, which we investigate in the following Lemma.
\begin{lemma}[Estimates on $q_{\eps,\delta}$]\label{lemma:estimate-q-epsilon}
Let $p$, $\rho$, $\rho_\delta$, $h$, $u_M$, $p_{\eps,\delta}$, $q_{\eps,\delta}$, $f_{\eps,\delta}$, and $k_{\eps,\delta}$ be as above. Then for all $\eps,\delta>0$ it holds the following:
\begin{equation}\label{e:first-estimate-q-epsilon}
    \|p_{\eps,\delta}-q_{\eps,\delta}\|_{L^\infty(\R\times D)}\le {\delta}.
\end{equation}
\end{lemma}

\begin{proof} Without loss of generality we assume $u_M(x) =0$.
By \Cref{lemma:estimates-f-k-epsilon} we know that 
\begin{align*}
f_{\eps,\delta}(z,x)=\eps \de_z p_{\eps,\delta}(z,x)&=\int_{\R}\rho_{\delta}\big(\big(\eps^{-1}z+h(x)\big)-y\big)\de_z p(y)\dif y.
\end{align*}
Since $\de_zp(z)=\ind_{[-1,1]}(z),$
and $0\le f_{\eps,\delta}\le 1$ in $\R$, we have that
\begin{equation*}
    f_{\eps,\delta}(z,x) = 1 \quad \text{if}\quad |z+\eps h(x)|\le(1-\delta)\eps \qquad\text{and}\qquad f_{\eps,\delta}(z,x) = 0 \quad\text{if}\quad |z+\eps h(x)|> (1+\delta)\eps.
\end{equation*}
Therefore, for any $z+\eps h(x) \ge0$ and $x\in D$ it holds the following estimate
\begin{align*}
|p_{\eps,\delta}(z,x) - q_{\eps,\delta}(z,x)| &= 
    \frac{1}{\eps}\left|\int_0^z f_{\eps,\delta}(\zeta,x) - (f_{\eps,\delta})(\zeta,x)^2 \dif\zeta\right|\\
    & \le 
        \frac{1}{\eps}\left|\int_{-\eps h(x)+(1-\delta)\eps}^{-\eps h(x)+(1+\delta)\eps}f_{\eps,\delta}(\zeta,x)(1-f_{\eps,\delta}(\zeta,x))\dif \zeta \right|
    \le {\delta},
\end{align*}
since $f_{\eps,\delta}(\zeta,x)(1-f_{\eps,\delta}(\zeta,x))\le \frac12$. Similarly, for $z+\eps h(x) \le0$, we have  
\begin{align*}
|p_{\eps,\delta}(z,x) - q_{\eps,\delta}(z,x)| 
    & \le 
        \frac{1}{\eps}\left|\int^{-\eps h(x)-(1-\delta)\eps}_{-\eps h(x)-(1+\delta)\eps}f_{\eps,\delta}(\zeta,x)(1-f_{\eps,\delta}(\zeta,x))\dif \zeta \right|
    \le {\delta},
\end{align*}
which concludes the proof.
\end{proof}
We now set $\delta = \eps$ and define the functions $p_\eps, f_\eps, k_\eps,$ and $q_\eps$ as
\begin{equation}\label{e:definition-p-f-k-q-epsilon}
    p_\eps(z,x) := p_{\eps,\eps}(z,x), \quad f_\eps(z,x) := f_{\eps,\eps}(z,x), \quad k_\eps(z,x) := k_{\eps,\eps}(z,x), \quad 
    q_\eps(z,x) := q_{\eps,\eps}(z,x).
\end{equation}
In what follows, we shall often omit the explicit dependence on $x$ for the sake of brevity.

\subsection{The rescaled functional \texorpdfstring{$\Jc_\eps$}{Je}}
In the elliptic regularization scheme, it is convenient to introduce the following rescaled version of the functional \eqref{e:efunctional}.
Let $h\in L^\infty(D)$ with $|h|\le 1$, $u_M : D \to \R$ measurable, and $F \in L^2(D;\R^d)$.
For all $\eps>0$ we define $\Jc_\eps : \mathcal{U} \to \R \cup \{+\infty\}$ as
\begin{equation}\label{e:efunctional-rescaled}\tag{$\Jc_\eps$}
     \Jc_\eps(v) := \iint e^{-t} \Biggl\{ |\de_t v |^2 +  \big| \sqrt{\eps}\de_t(p_\eps(v,x)) \big|^2 + \eps^4 |\nabla v |^2 + 2\eps^4 F\cdot  \nabla v \Biggr\} \dif x \dif t,
\end{equation}
setting $\Jc_\eps(v) = +\infty$ if the integral diverges.
Through the same computation of the case $\Fc_\eps$, it holds 
\begin{equation}\label{e:eJ-lower-bound}
    \Jc_\eps(v) \ge -\eps^4\, \| F\|_{L^2(D)}^2  \qquad\text{for all}\quad v \in \Uc.
\end{equation}
By applying the change of variables $s = \eps^{-4}\,t$ and defining the function $v$ as
\begin{equation}\label{e:change-variable-rescale}
    v(x,s) := u(x, \eps^4 s),
\end{equation}
one obtains the relation
\begin{equation}
    \de_s v(x,s) = \eps^4 \de_t u(x, \eps^4 s).
\end{equation}
Consequently, the following relationship between \eqref{e:efunctional} and \eqref{e:efunctional-rescaled} holds
\begin{align}\label{e:equivalence_F_J}
    \Fc_\eps(u) &= \iint e^{-s} \Biggl\{ \eps^4 \left[ 1 + \frac{1}{\eps}f_\eps(u(x,\eps^4 s))^2 \right] |\de_t u(x,\eps^4 s)|^2 + |\nabla u (x,\eps^4 s)|^2 + 2F \cdot\nabla u(x,\eps^4 s)\Biggr\} \dif x \dif s \nonumber \\
    &= \frac{1}{\eps^4} \iint e^{-s} \Biggl\{ \frac{1}{\eps^4} \left[ 1 + \frac{1}{\eps}f_\eps(v)^2 \right] |\de_s v|^2 + \eps^4 |\nabla v|^2 +2 \eps^4 F \cdot\nabla v \Biggr\} \dif x \dif s = \frac{1}{\eps^4} \Jc_\eps(v).
\end{align}
Moreover, since we rescale only in time, we get that if $u \in \Uc_{\mathcal D}(D,g)$, then also $u \in \Uc_{\mathcal D}(D,g)$. Thus, solving the minimization problem for \eqref{e:efunctional} is equivalent to solving it for \eqref{e:efunctional-rescaled}. Specifically, $u_\eps$ is a minimizer for \eqref{e:efunctional} in $\Uc_{\mathcal D}(D,g)$, if and only if $v_\eps$, defined in \eqref{e:change-variable-rescale}, is a minimizer for \eqref{e:efunctional-rescaled} in $\Uc_{\mathcal D}(D,g)$. 
The same applies when we consider the Neumann problem in $\Uc_{\mathcal N}(D,g)$.

\subsection{The minimization problem for \texorpdfstring{$\Jc_\eps$}{Je}}
In this section, we show that there are non-trivial minimizers of the functional $\Jc_{\eps}$ in the Dirichlet and Neumann classes $\Uc_{\mathcal D}(D,g)$ and $\Uc_{\mathcal N}(D,g)$. 
As mentioned above, the proof holds also for unbounded domains $D\subset\R^d$.

\begin{proposition}[Existence of minimizer]\label{prop:existence-J-epsilon}
Let $ D$ be an open set in $\R^d$, $g\in H^1(D)$, $h\in L^\infty(D)$ with $|h(x)|\le1$, $u_M : D \to \R$ measurable, and $F \in L^2(D;\R^d)$. 
Then, for all $\eps>0$ there are minimizers of $\Jc_{\eps}$ in the classes $\Uc_{\mathcal D}(D,g)$ and $\Uc_{\mathcal N}(D,g)$. 
Moreover, if $v_\eps$ is a minimizers of $\Jc_{\eps}$ in $\Uc_{\mathcal D}(D,g)$ (or in $\Uc_{\mathcal N}(D,g)$), then the following estimate holds: 
\begin{equation}\label{e:uniform-energy-estimate}
   \left\vert \Jc_\eps(v_\eps) \right\vert  \le 2 \left( \|\nabla g \|_{L^2(D)}^2 + \|F\|_{L^2(D)}^2 \right)\, \eps^4.
\end{equation}
\end{proposition}
\begin{proof} We carry out the proof in the class $\Uc_{\mathcal D}(D,g)$, the Neumann case $\Uc_{\mathcal N}(D,g)$ being analogous. We proceed in four steps.

\paragraph{Step 1. Uniform energy estimates and well-posedness}
$\mathcal J_\eps$ is bounded by below by \eqref{e:eJ-lower-bound}. On the other hand, the function $w(x,t) \equiv g(x)$ belongs to $\mathcal U_{\mathcal D}(D,g)$, and it holds
\begin{equation}\label{e:uniformbound_0}
\Jc_\eps(w) =\eps^4 \int_0^{+\infty} e^{-t}\int_{D} \big(|\nabla g|^2 + 2F\cdot \nabla g \big)  \le 2 \big(\|\nabla g \|_{L^2( D)}^2 + \| F\|_{L^2(D)}^2 \big)\, \eps^4,
\end{equation}
and hence 
 \begin{equation*}
\left\vert \inf\Big\{\mathcal J_\eps(v)\ :\ v\in \mathcal U_{\mathcal D}(D,g)\Big\} \right\vert \le 2 \big(\|\nabla g \|_{L^2(D)}^2 + \| F\|_{L^2(D)}^2\big)\, \eps^4.
\end{equation*}
This estimate, together with \eqref{e:eJ-lower-bound}, proves \eqref{e:uniform-energy-estimate} for any minimizer $v_\eps$ in $\Uc_{\mathcal D}(D,g)$.\medskip

\paragraph{Step 2. Coercivity} 
We claim that, for all $T>0$ and all $v\in\mathcal U_{\mathcal D}(D,g)$, it holds 
\begin{equation}\label{e:coercivity}
\| v \|_{H^1(D_T)} \le C(T,\eps)\left(\|g\|_{L^2(D)} + \Jc_\eps(v) + \|F\|_{L^2(D)}^2\right);
\end{equation}
in particular, the minimizing sequences are bounded in $H^1(D_T)$.\medskip

For all $T>0$ and $v\in \Uc$ it holds the following energy estimate
\begin{equation}\label{e:energy_estimates_minseq}
    \begin{split}
    \iint_{D_T} |\de_tv |^2 + |\nabla v|^2\dif x \dif t  
        &\le \frac{e^T}{\eps^4}\iint_{D_T} e^{-t} \left\{ \left[1+ \frac{1}{\eps}f_\eps(v)^2\right]|\de_tv|^2 + \eps^4 |\nabla v|^2\right\} \dif x \dif t\\
        &\le \frac{e^T}{\eps^4}\Jc_\eps(v) + 2e^{T}\| F \|_{L^2(D)} \| \nabla v\|_{L^2(D_T)}\\
        &\le \frac{e^T}{\eps^4}\Jc_\eps(v) + e^T\| F \|_{L^2(D)}^2 +  e^T\| \nabla v\|_{L^2(D_T)}^2,
    \end{split}
\end{equation}
and thus we conclude the bound on the energy part $\| \nabla_{x,t}v \|_{L^2(D_T)}$.
The $L^2(D_T)$-estimate of $v$ follows by applying the Poincaré inequality in time. 
Precisely, since for every $\varphi:\R\to\R$, $\varphi\in H^1_{loc}(\R)$ with $\varphi(0)=0$, it holds
\[
\int_0^T\varphi^2(t) \dif t \le T^2\int_0^T|\varphi'(t)|^2 \dif t,
\]
we get that for almost-every fixed $x\in D$ we have 
\[
    \int_0^Tv^2(x,t)\dif t
        \le 2Tg^2(x)+2\int_0^T(v(x,t)-g(x))^2 \dif t
        \le 2Tg^2(x)+2T^2\int_0^T|\de_t v(x,t)|^2 \dif t.
\]
Therefore, integrating in $x\in D$, we obtain
\begin{equation}\label{e:poincare_time}
    \iint_{D_T} |v|^2 \dif x \dif t \le 2T\int_{D} g^2 \dif x+ 2T^2 \iint_{D_T} |\de_t v|^2 \dif x \dif t.
\end{equation}
Finally, \eqref{e:energy_estimates_minseq} and \eqref{e:poincare_time}, imply \eqref{e:coercivity}.
\medskip

\paragraph{Step 3. Compactness}
Thanks to Coercivity, a diagonal argument implies that, for all minimizing sequences that satisfy the following
\begin{equation}\label{e:min_seq1}
\{v_n\}_{n\in\NN}\subset\Uc_{\mathcal D}(D,g), \qquad \Jc_{\eps}(v_n) \xrightarrow[n\to +\infty]{} \inf_{\Uc_{\mathcal D}(D,g)} \Jc_{\eps}, \qquad \text{and}\qquad \sup_{n\in\NN}\Jc_{\eps}(v_n) \le C,
\end{equation}
there exists a subsequence $\{v_{n_j}\}_{j\in\NN}$ and $v\in \Uc_{\mathcal D}(D,g)$ such that
\begin{equation}\label{e:min_seq2}
v_{n_j}\xrightharpoonup[j\to+\infty]{} v \quad \text{in}\quad \Uc,
    \qquad 
    v_{n_j} \xrightarrow[j\to+\infty]{} v \quad \text{in}\quad L^2(D_T\cap C_R) \quad \text{for all}\quad R>0,
\end{equation}
where $C_R$ is defined in \Cref{sub:notation}, and
\begin{equation}\label{e:min_seq3}
v_{n_j}(x,t) \xrightarrow[j\to+\infty]{} v(x,t) \quad \text{for almost-every}\quad (x,t)\in D\times(0,+\infty).
\end{equation}
To prove that $v$ is a minimizer of $\Jc_\eps$, it is suffices to prove that the functional is lower-semicontinuous with respect to minimizing sequences.

\paragraph{Step 4. Lower semicontinuity} 
We claim that, for all minimizing sequences $\{v_n\}_{n\in\NN}$ satisfying \eqref{e:min_seq1}, \eqref{e:min_seq2}, and \eqref{e:min_seq3}, the following inequality holds:
\begin{equation}\label{e:semicontinuity_f}
   \Jc_\eps(v) \le \liminf_{n\to+\infty} \Jc_\eps(v_n).
\end{equation}

Since the $L^2$-norm of the gradient is lower-semicontinuous with respect to $H^1$-weak convergence,  we only need to check that for all $T>0$, the following quantity
\[
    \iint_{D_T} e^{-t}f_\eps(v)^2|\de_tv|^2\dif x \dif t.
\]
is lower-semicontinuous.
We first notice that from \eqref{e:min_seq1} it follows
\[
    \sup_{n\in\NN} \iint_{D_T} e^{-t} f_\eps(v_n)^2 |\de_t v_n|^2 \dif x \dif t 
        \le 
    \frac{1}{\eps^3} \sup_{n\in\NN} \Jc_\eps(v_n) \le \frac{C}{\eps^3}\,,
\]
and therefore the sequence $\left\{e^{-\sfrac t2}f_\eps(v_n) \de_t v_n\right\}_{n\in\NN}$ is uniformly bounded in $L^2( D_T)$, for every $T>0$.
Thus, there exists a subsequence $\{v_{n_j}\}_{j\in\NN}$ and a measurable function $w$ such that
\[
    e^{-\sfrac t2} f_\eps(v_{n_j})v_{n_j}\xrightharpoonup[j\to+\infty]{} w \qquad \text{weakly in}\quad L^2(D_T)\quad\text{for all}\quad T>0,
\]
and such that $w$ satisfies the following inequality
\[
    \iint_{D_T} |w|^2 \dif x \dif t 
    \le \liminf_{j\to+\infty} \iint_{D_T}e^{-t} f_\eps(v)^2 |\de_t v|^2 \dif x \dif t.
\]
\noindent
We now show that
$w = e^{-\sfrac t2}f_\eps(v)\,\de_tv.$\medskip

This follows since, for all $(x,t)\in D\times(0,+\infty)$, $\eps>0$, and $n\in\NN$, it holds
\[
    |f_\eps(v_n(x,t),x) - f_\eps(v(x,t),x)| = \left|\frac{1}{\eps^2}k_\eps\big(\phi_n(x,t),x\big)\right|\cdot|v_n(x,t)- v(x,t)| ,
\]
where $\phi_n(x,t)\in [v_n(x,t),v(x,t)]$. 
Thanks to \Cref{lemma:estimates-f-k-epsilon}, the function $k_\eps$ is bounded and therefore
\[
    |f_\eps(v_n)- f_\eps(v)| \le \frac{C}{\eps^2}|v_n - v|.
\]
Now, since for all $T,R>0$ we have
\[
v_{n} \xrightarrow[n\to+\infty]{} v \qquad \text{strongly in}\quad L^2(D_T \cap C_R), 
\]
the previous estimate implies that
\[
f_\eps(v_n) \xrightarrow[n\to+\infty]{} f_\eps(v) \quad \text{strongly in}\quad L^2(D_T\cap C_R). 
\]
Finally, by the {Weak-Strong} Convergence Criterion we conclude that 
\[
e^{-\sfrac t2}f_\eps(v_n)\de_t v_n 
    \xrightharpoonup[n\to+\infty]{} 
    e^{-\sfrac t2} f_\eps(v) \de_t v\quad\text{in}\quad L^2(D_T \cap C_R), 
\]
for all $R,T>0$, and since the weak limit is unique, this concludes the proof.
\end{proof}

\section{Energy estimates for the minimizers' sequence}\label{sec:energy-estimates}

For every minimizer $v_\eps\in\mathcal U_{\mathcal D}(D,g)$ of \eqref{e:efunctional-rescaled}, we define the following energies
\begin{align}
    I_\eps(t)& = \int_{D(t)}\left(1+\frac1\eps f_\eps(v_\eps,x)^2\right)|\de_tv_\eps|^2 \dif x,\label{eq:energy_time}\\ 
    L_\eps(t)& = \eps^4\int_{D(t)}|\nabla v_\eps|^2 + 2F \cdot\nabla v_\eps \dif x,\label{eq:energy_space}
\end{align}
which, thanks to \eqref{e:uniform-energy-estimate}, are well-defined for almost every time $t\ge0$.
We also define the tail energy
\begin{equation}\label{eq:energy_at_infinity}
    E_\eps(t) = e^{t}\int_t^{+\infty} e^{-\tau}\Bigl[I_\eps(\tau) + L_\eps(\tau)\Bigr]\dif \tau,
\end{equation}
which measures the weighted energy remainder as $t \to +\infty$. In particular, for all $t \ge 0$, we have 
\[
    E_\eps(t) = \Jc_\eps(v_\eps(\cdot, t+\cdot)),
\] and at the initial time $t=0$, it is precisely $E_\eps(0) = \Jc_\eps(v_\eps)$.

\begin{proposition}[Inner variation and monotonicity formula]\label{prop:Inner_monotonicity}
Let $D$ be an open set in $\R^d$, $g\in H^1(D)$, $h\in L^\infty(D)$ with $|h(x)|\le 1$, $u_M : D \to \R$ measurable, and $F \in L^2(D;\R^d)$.
For all $\eps>0$, if $v_\eps$ is a minimizer of the functional \eqref{e:efunctional-rescaled} in either $\Uc_{\mathcal D}(D,g)$ or  $\Uc_{\mathcal N}(D,g)$ and  $I_\eps$, $L_\eps$ ,and $E_\eps$ are respectively defined as in \eqref{eq:energy_time}, \eqref{eq:energy_space} and \eqref{eq:energy_at_infinity}, then the following properties hold:
\begin{enumerate}[{\itshape (i)}]
\item Inner variation identity. For almost-every $t>0$, we have 
\begin{equation}\label{e:inner-variation-identity}
L_\eps(t)-E_\eps(t)= I_\eps(t);
\end{equation}
\item\label{item:monotonicity-formula-for-epsilon} Energy decaying. The energy $E_\eps:[0,+\infty)\to [0,+\infty)$ is non-increasing in time and 
 \begin{equation}\label{e:monotonicity-formula-for-epsilon}
E_\eps'(t)=- 2\,I_\eps(t)\quad\text{for almost-every}\quad t> 0.
\end{equation}
\item\label{item:uniform-energy-bound-for-every-t} Uniform energy bound. For every $t\ge 0$, we have the bound
\begin{equation}\label{e:uniform-energy-bound-for-every-t}
- \eps^4 \| F \|_{L^2(D)}^2\le
E_\eps(t)\le 2 \left( \|\nabla g \|_{L^2(D)}^2 + \|F\|_{L^2(D)}^2 \right)\, \eps^4.
\end{equation}

\end{enumerate}
\end{proposition}

\begin{proof} 
For simplicity, we drop the index $\eps$ in $L_\eps$, $E_\eps$, $I_\eps$, $v_\eps$.

We first prove the energy bounds in \ref{item:uniform-energy-bound-for-every-t}. Since $E(0)=\Jc_\eps(v_\eps)$, the upper bound in \eqref{e:uniform-energy-bound-for-every-t} follows from the monotonicity property \ref{item:monotonicity-formula-for-epsilon} and the estimate \eqref{e:uniform-energy-estimate}.
Regarding the lower one, as in \Cref{prop:existence-J-epsilon}, we observe that
\begin{align*}
    E(t) &= \int_{t}^{+\infty}e^{-(\tau-t)} \big[I(\tau) + L(\tau)\big] \dif \tau\\
    &\ge \eps^4\int_{t}^{+\infty}e^{-(\tau-t)} \| \nabla v(\tau)\|_{L^2(D)}(\|\nabla v(\tau)\|_{L^2(D)} - 2\| F\|_{L^2(D)}) \dif \tau.
\end{align*}
Since for all $\tau\ge0$ it holds that
\[
    \| \nabla v(\tau)\|_{L^2(D)}(\|\nabla v(\tau)\|_{L^2(D)} - 2\| F\|_{L^2(D)}) \ge -\|F\|_{L^2(D)}^2,
\]
we get the estimate from below in \eqref{e:uniform-energy-bound-for-every-t}.

To derive the inner variation identity \eqref{e:inner-variation-identity}, we consider the time-reparametrization $\Phi_s(t) := t + s\phi(t)$, where $\phi \in C^\infty((0, +\infty))$.
Let 
\[
    v_s(x,t) := v(x, \Phi_s(t)),
\]
and note that if $v \in \Uc_{\mathcal D}(D,g)$ or $v \in \Uc_{\mathcal N}(D,g)$, the same holds for $v_s$ since $\Phi_s(0)=0$ and $v$ is time-independent on $\de_L D$.
By the chain-rule it follows that 
\[
    \partial_tv_s(x,t)=\big(1+s\phi'(t)\big)\,\partial_tv(x,\Phi_s(t)),
\]
and so the energy $\Jc_\eps(v_s)$ reads as
\[
    \Jc_\eps(v_s) 
        = \iint e^{-t} \biggl\{ 
            \left[1 + \frac{1}{\eps} f_\eps(v(x,\Phi_s(t))^2\right]|\de_t v(x,\Phi_s(t))|^2(1 + s\phi'(t) )^2 
            + \eps^4 |\nabla v (x,\Phi_s(t)|^2  + 2\eps^4 F \cdot\nabla v
        \biggr\} \dif x \dif t.
\]
We set $\Psi_s=\Phi^{-1}_s:\mathbb R\to\mathbb R$ to be the inverse of $\Phi_s$ and we consider the change of variables 
$$\tau = \Phi_s(t),\qquad\text{and}\qquad  t = \Psi_s(\tau) = \Phi_s^{-1}(\tau).$$ 
In particular $\dif t = \Psi_s'(\tau)\dif\tau$ where $\Psi_s(\tau) = \tau - s\phi(\tau) + o(s)$. Therefore
\begin{equation*}
\begin{split}
    \Jc_{\eps}(v_s) &= 
    \iint e^{-{\Psi_s(\tau)}} \Biggl\{ \biggl[1 + \frac{1}{\eps}f_\eps(v)^2 \biggr]|\de_t v |^2 \Bigl(1+s\phi'(\Psi_s(\tau))\Bigr)^2 
    + \eps^4\left(|\nabla v|^2 + 2F\cdot \nabla v\right) \Biggr\} \Psi_s'(\tau) \dif \tau\\
    &=\iint e^{-\tau}(1+s\phi ) \biggl\{ \biggl[1 + \frac{1}{\eps} f_\eps(v)^2\biggr]|\de_t v |^2 (1 + 2s\phi' )
    + \eps^4\left(|\nabla v|^2 + 2F\cdot \nabla v\right) \biggr\}(1-s\phi')\dif\tau + o(s)\\
    &=\iint e^{-\tau}\Biggl\{ \biggl[1 +\frac{1}{\eps}f_\eps(v)^2\biggr]|\de_t v|^2 + \eps^4\left(|\nabla v|^2 + 2F\cdot \nabla v\right) \Biggr\} \dif\tau \\
        & \hspace{2cm}+s\iint e^{-\tau}\biggl[1 +\frac{1}{\eps}f_\eps(v)^2\biggr]|\de_t v|^2 \phi'(\tau) \dif\tau \\
        &\hspace{3cm}- s\iint e^{-\tau}\eps^4\left(|\nabla v|^2 + 2F\cdot \nabla v\right)\phi'(\tau)\dif \tau\\
        &\hspace{4cm}+s\iint e^{-\tau}\phi \Biggl\{\biggl[1 +\frac{1}{\eps}f_\eps(v)^2\biggr]|\de_t v|^2 + \eps^4\left(|\nabla v|^2 + 2F\cdot \nabla v\right)\Biggr\}\dif \tau + o(s)\\
    &=\Jc_\eps(v) +s \Biggl\{\int_0^{+\infty} e^{-\tau}\bigl[I(\tau) - L(\tau)\bigr]\phi'(\tau) 
    + \int_0^{+\infty} e^{-\tau}\bigl[I(\tau) + L(\tau)\bigr]\phi(\tau)\Biggr\} + o(s).
\end{split}
\end{equation*}
By the minimality of $v$, we get
\begin{equation}\label{e:fisrt-variation-J-eps}
 \begin{split}  0=\frac{\dif}{\dif s}\Biggr|_{s=0} \Jc_\eps(v_s)& = \int_0^{+\infty} e^{-\tau}\bigl[I(\tau) - L(\tau)\bigr]\phi'(\tau)\dif\tau
 + \int_0^{+\infty} e^{-\tau}\bigl[I(\tau) + L(\tau)\bigr]\phi(\tau)\dif\tau.
 \end{split}
\end{equation}

We now choose a specific test function $\phi$. Precisely, for any $t>0$ and (small) $\lambda>0$, we define the non-decreasing continuous function $\phi_{t,\lambda}$ as follows:
\begin{equation}
    \phi_{t,\lambda}(\tau) := \begin{cases}
        0 & \mbox{if } \tau \le t, \\
        e^{t} &\mbox{if }\tau \ge t+ \lambda,\\
    \end{cases}
\end{equation}
and 
\[
    \phi_{t,\lambda}'(\tau) \le \frac{2 e^{t}}{\lambda} \qquad\text{for all}\quad\tau\in[t,t+\lambda].
\]
Testing \eqref{e:fisrt-variation-J-eps} with $\phi_{t,\lambda}$, for almost every $t>0$ we can take the limit as $\lambda\to0^+$ we find the identity
\begin{equation}
    0 = I(t)  - L(t) + E(t),
\end{equation}
which is precisely \eqref{e:inner-variation-identity}.

{We next prove the time-monotonicity of the energy $E(t)$.} Since $\Jc_\eps(u)<+\infty$, $E\in W^{1,1}_{loc}(\RR_+)$, and therefore the weak derivative $E'(t)$ is well defined for almost every $t\ge0$ and it reads
\begin{equation}\label{e:derivative-of-E}
    E'(t) 
        = E(t) -  
                \Bigl[I(t) + L(t)  \Bigr].
\end{equation}
Now, by the energy identity \eqref{e:inner-variation-identity}, we get 
\begin{equation}\label{e:derivative-of-E-parte-2}
    E'(t) = E(t) - L(t) - I(t) = - 2I(t),
\end{equation}
which is precisely \eqref{e:monotonicity-formula-for-epsilon} and gives  $E'(t) \le 0$. 
\end{proof}

\subsection{Energy estimates for \texorpdfstring{$v_\eps$}{rescaled minimizers}}
The energy decay from \Cref{prop:Inner_monotonicity} directly implies the following energy estimates for the minimizers of \eqref{e:efunctional-rescaled}.

\begin{lemma}\label{lemma:bound_grandient-rescaled}
Let $D$ be an open set in $\R^d$, $g\in H^1(D)$, $h\in L^\infty(D)$ with $|h(x)|\le 1$, $u_M : D \to \R$ measurable, and $F \in L^2(D;\R^d)$.
Let $\eps>0$ and let $v_\eps$ be a minimizer of the functional \eqref{e:efunctional-rescaled} in $\Uc_{\mathcal D}(D,g)$ (or in $\Uc_{\mathcal N}(D,g)$.
Then it holds the following estimate
\begin{equation}\label{e:bound_time}
    \iint\left(1+\frac{1}{\eps}f_\eps(v_\eps)^2\right) |\de_tv_\eps |^2\dif x\dif t
        \le 2\eps^4\Big(\|\nabla g\|_{L^2(D)}^2+ \| F \|_{L^2(D)}^2\Big).
\end{equation}
Moreover, for every $\tau\ge0$ and $T>0$ it holds
\begin{equation}\label{e:bound_space}
    \int_{\tau}^{\tau+T} \int_{D}|\nabla v_\eps|^2(x,t) \dif x \dif t  \le 8(T+1)\Big(\|\nabla g\|_{L^2(D)}^2+ \| F \|_{L^2(D)}^2\Big).
\end{equation}
\end{lemma}
\begin{proof}
We firstly show the estimate \eqref{e:bound_time}. By definition, for almost-every $t>0$, we have 
\[
    I_\eps(t) =
   \int_{D(t)}\left(1+\frac{1}{\eps}f_\eps(v_\eps)^2\right) |\de_tv_\eps |^2\dif x.
\]
If we integrate \eqref{e:monotonicity-formula-for-epsilon} in $[0,T]$, we get
\begin{align*}
    \iint_{D_T} \left(1+\frac{1}{\eps}f_\eps(v_\eps)^2\right) |\de_tv_\eps |^2\dif x\dif t 
    &= \int_{0}^T I_\eps(t)\dif t = -\frac{1}{2}\int_{0}^TE_\eps'(t)\dif t \\
    &= \frac{E_\eps(0)-E_\eps(T)}{2} 
    \le \frac32\eps^4\Big(\|\nabla g\|_{L^2(D)}^2 + \| F \|_{L^2(D)}^2\Big),
\end{align*}
where we used the upper bound for $E_\eps(0)$ from \eqref{e:uniform-energy-bound-for-every-t} and the lower bound $E_\eps(T)\ge -\eps^4 \|F\|_{L^2(D)}^2$. We then conclude by taking the limit for $T\to+\infty$.\medskip

We next prove \eqref{e:bound_space}. Thanks to \eqref{eq:energy_space},  
for almost every $t\ge0$ it holds
\begin{equation*}
    \| \nabla v_\eps(t)\|_{L^2(D)}^2 \le 4 \| F\|^2_{L^2(D)} + 2\eps^{-4} L_\eps(t),
\end{equation*}
and therefore by \eqref{e:inner-variation-identity}, the monotonicity of the energy \eqref{e:monotonicity-formula-for-epsilon}, \eqref{e:uniform-energy-bound-for-every-t}, and \eqref{e:bound_time} it follows that
\begin{align*}
   \int_{\tau}^{\tau+T} \int_{D}|\nabla v_\eps|^2 \dif x\dif t &\le 4T \|F\|_{L^2(D)}^2 
        + 2\eps^{-4}\int_\tau^{\tau+T}{L_\eps(t)} \dif t\\
   &= 4T \|F\|_{L^2(D)}^2 + 2\eps^{-4}\int_\tau^{\tau+T}  E_\eps(t) + I_\eps(t) \dif t \\
   &\le 4 T \|F\|_{L^2(D)}^2 + 2\eps^{-4}\left(T\, E_\eps(\tau) + \int_\tau^{\tau+T} I_\eps(t) \dif t\right)  \\
   &= 4 T \|F\|_{L^2(D)}^2 + 2\eps^{-4}\left(T\, E_\eps(\tau) + \frac12\big(E_\eps(\tau)-E_\eps(\tau+T)\big) \right)  \\
   &\le 8 (T+1) \big(\|\nabla g\|_{L^2(D)}^2 + \| F \|_{L^2(D)}^2 \big).
\end{align*}
which concludes the proof.
\end{proof}

\subsection{The energy estimates for \texorpdfstring{$u_\eps$}{original mininizers}}

The energy estimates provided by \Cref{lemma:bound_grandient-rescaled} directly imply similar estimates for minimizers $u_\eps$ of the original functional \eqref{e:efunctional}.
Indeed, for all minimizers $v_\eps$ of \eqref{e:efunctional-rescaled}, by \eqref{e:change-variable-rescale} and \eqref{e:equivalence_F_J}, $u_\eps\in \Uc_{\mathcal D}(D,g)$ defined as
\[
u_\eps(x,t):= v_\eps(x,\eps^{-4}t),
\]
is a minimizer for \eqref{e:efunctional}. 

\begin{lemma}[Uniform energy bounds]\label{lemma:energy-bound}
Let $D$ be an open set in $\R^d$, $g\in H^1(D)$, $h\in L^\infty(D)$ with $|h(x)|\le 1$, $u_M : D \to \R$ measurable, and $F \in L^2(D;\R^d)$.
For all $\eps>0$, if $u_\eps$ is minimizer of the functional \eqref{e:efunctional} in either $\Uc_{\mathcal D}(D,g)$ or $\Uc_{\mathcal N}(D,g)$, then
    they hold the following estimates:
    \begin{enumerate}[{\itshape(i)}]
        \item \textbf{Global integral time-derivative bound.}
    \begin{equation}\label{e:bound-time-energy}
        \frac12\iint \left(1 + \frac1\eps f_\eps(u_\eps)^2\right) |\de_t u_\eps|^2 \dif x \dif t 
        \le \|\nabla g\|_{L^2(D)}^2 + \|F\|_{L^2(D)}^2.
    \end{equation}

        \item \textbf{Integral bound.} For all $T>0$ it holds
        \begin{equation}\label{e:L2-bound}
            \| u_\eps \|_{L^2(D_T)}^2 \le 2T \| g\|_{L^2(D)}^2 + 2T^2 \left( \|\nabla g\|_{L^2(D)}^2 + \| F \|_{L^2(D)}^2\right).
        \end{equation}
    
        \item \textbf{Almost uniform energy bound.} For all $\tau\ge0$ and $T>0$ there hold
        \begin{equation}\label{e:energy-bound}
            \int_{\tau}^{\tau + T} \int_{D(t)} |\nabla u_\eps|^2 \le 8 \left( T + \eps^4 \right)\left(\|\nabla g\|_{L^2(D)}^2 + \|F\|_{L^2(D)}^2\right).
        \end{equation}
    \end{enumerate}
\end{lemma}
\begin{proof}
    By the identity \eqref{e:equivalence_F_J}, if $u_\eps$ is a minimizer for \eqref{e:efunctional} in $\Uc(D,g)$, then $v_\eps(x,t) = u_\eps(x,\eps^4t)$ is a minimizer for \eqref{e:efunctional-rescaled}.
    Thus, relying on the results of \Cref{lemma:bound_grandient-rescaled} we get that
    \begin{equation*}
       \frac12 \iint \left(1 + \frac1\eps f_\eps(u_\eps)^2\right) |\de_t u_\eps|^2 \dif x \dif t 
        = \frac{\eps^{-4}}{2} \iint \left(1 + \frac1\eps f_\eps(v_\eps)^2\right) |\de_t v_\eps|^2 \dif x \dif t
        \le \|\nabla g\|_{L^2(D)}^2 + \| F \|_{L^2(D)}^2,
    \end{equation*}
    that is precisely \eqref{e:bound-time-energy}. In particular it implies the following uniform estimate on the time derivative 
    \[
       \frac12 \|\de_t u_\eps \|_{L^2(D_\infty)}^2\le \|\nabla g\|_{L^2(D)}^2 + \| F \|_{L^2(D)}^2.
    \]
    Through a time-Poincaré inequality, as in \Cref{prop:existence-J-epsilon} the previous estimate implies $L^2(D_T)$-estimates:
    \[
        \| u_\eps \|_{L^2(D_T)}^2 \le 2T \| g\|_{L^2(D)}^2 + 2T^2 \left( \|\nabla g\|_{L^2(D)}^2 + \| F \|_{L^2(D)}^2\right) \qquad\text{for all}\quad T>0.
    \]
    Finally, the estimate \eqref{e:energy-bound} comes from \eqref{e:bound_space}. Indeed, for all $\tau \ge0$ and $T>0$, it holds
    \begin{align*}
        \int_{\tau}^{\tau+T}\int_{D} |\nabla u_\eps(x,t)|^2 \dif x \dif t 
        &= \int_{\tau}^{\tau+T}\int_D \left\vert\nabla v_\eps\left(x,\eps^{-4}t \right)\right\vert^2 \dif x \dif t\\
        &= \eps^{4}\int_{\eps^{-4}\tau}^{\eps^{-4}(\tau+T)}\int_D  |\nabla v_\eps|^2 \dif x \dif t \\
        &\le 8\,\eps^{4}(\eps^{-4}T + 1)\left( \| \nabla g\|_{L^2(D)}^2 + \|F\|_{L^2(D)}^2 \right),
    \end{align*}
   which concludes the proof. 
\end{proof}

\section{Convergence to the Stefan Problem}\label{sec:convergence_Stefan}

In this section, we employ the energy bounds established in the previous section to pass to the limit in the sequence of minimizers as $\eps \to 0$, thereby establishing that the elliptic regularization scheme converges to an enthalpy solution of the Stefan problem (\Cref{t:main}). 
At the end of the section, we discuss the uniqueness of the limit of the scheme.

\subsection{The outer variation of \texorpdfstring{$\Fc_\eps$}{Fe}}
The limit equation is determined by the limit of outer variations; thus we investigate them in the following lemma.

\begin{lemma}[Outer variation]\label{lemma:outer-variation-epsilon} 
Let $D$ be an open set in $\R^d$, $g\in H^1(D)$, $h\in L^\infty(D)$ with $|h(x)|\le 1$, $u_M : D \to \R$ measurable, and $F \in L^2(D;\R^d)$.
Let $\eps>0$, if $u_\eps$ is a minimizer of the functional \eqref{e:efunctional} in either $\Uc_{\mathcal D}(D,g)$ or $\Uc_{\mathcal{N}}(D,g)$.
Then it holds
\begin{equation}\label{e:outer-variation-epsilon}
    \iint  \Bigg\{\left[1+\frac{1}{\eps}f_\eps(u_\eps)^2\right]\de_t u_\eps\, (\eta + \eps^4 \de_t \eta)  +\nabla u_\eps\cdot \nabla \eta + F\cdot \nabla \eta +\eps\, k_\eps(u_\eps)f_\eps(u_\eps) (\de_tu_\eps)^2\, \eta\Bigg\} \dif x \dif t=0,
\end{equation}
for all $\eta\in H^1_0(D\times(0,+\infty))$.
\end{lemma}

\begin{proof}
It is sufficient to prove the statement for functions $\eta\in C^{\infty}_c(D\times(0,+\infty))$. 

Let $\phi\in C^{\infty}_c(D\times(0,+\infty))$ and $\eps>0$, then for all $s>0$, we define the competitor 
$$u_{\eps,s}(x,t) := u_\eps(x,t) + s\phi(x,t),$$
and we compute the energy $\mathcal F_\eps(u_{\eps,s})$. We have:
\begin{align}
    |\de_t u_{\eps,s} |^2 &= |\de_t u_\eps + s\de_t \phi|^2 = |\de_t u_\eps|^2 + 2s\de_tu_\eps \,\de_t\phi + o(s);\label{e:der-time-eps-outer-variation}\\
    |\nabla u_{\eps,s}|^2 &=|\nabla u_\eps + s\nabla \phi|^2 = |\nabla u_\eps|^2 + 2s\nabla u_\eps\cdot \nabla \phi + o(s);\label{e:gradient-eps-outer-variation}
\end{align}
and
\begin{align}
       \eps |\de_t (p_\eps(u_{\eps,s}))|^2 
            &= \eps \left|  \frac{1}{\eps}f_\eps(u_{\eps,s})(\de_t u_\eps+s\de_t\phi) \right|^2\nonumber\\
            &= \frac{1}{\eps}\left|f_\eps(u_\eps) + \frac{s}{\eps^2}k_\eps(u_\eps)\phi + o(s) \right|^2 \left| \de_t u_\eps + s\de_t\phi \right|^2\label{e:der_p-eps-outer-variation}\\
            &=\frac{1}{\eps}f_\eps(u_\eps)^2(\de_tu_\eps)^2\nonumber \\
                &\hspace{2cm} +\frac{2s}{\eps}\left(f_\eps(u_\eps)^2 \de_tu_\eps \de_t\phi + \frac{1}{\eps^2}k_\eps(u_\eps)f_\eps(u_\eps)(\de_tu_\eps)^2\phi\right) + {o}(s).\nonumber
\end{align}
The minimality condition of $u_\eps$ implies that the outer variation vanishes, i.e.,
\begin{equation*}
    \frac{\dif}{\dif s}\Bigg|_{s=0}\Fc_\eps(u_{\eps,s}) = 0.
\end{equation*}
Finally substituting in this expression the identities found in \eqref{e:der-time-eps-outer-variation}, \eqref{e:gradient-eps-outer-variation} and \eqref{e:der_p-eps-outer-variation}, we get
\begin{align*}
    \iint \frac{e^{-t/\eps^4}}{\eps^4}\Biggl\{ \eps^4\left[1+\frac{1}{\eps}f_\eps(u_\eps)^2\right]\de_t u_\eps\, \de_t \phi +\nabla u_\eps\cdot \nabla \phi + F\cdot \nabla\phi  + \eps k_\eps(u_\eps)f_\eps(u_\eps) (\de_tu_\eps)^2\, \phi \Biggr\}=0.
\end{align*}
We conclude using as a test the function
\[
    \varphi(x,t)=\eps^4e^{t/\eps^4}\eta(x,t)
        \quad\text{where}\quad 
        \eta\in C^{\infty}_c(D\times(0,+\infty)),
\]
and then, by direct computation the above identity precisely becomes \eqref{e:outer-variation-epsilon}.
\end{proof}

\subsection{The convergence argument}

Before stating the main result of the section, we introduce a couple of function and a notation that it will be used many times in the following.

\begin{itemize}\setlength{\itemsep}{6pt}

    \item Let $\eta \in C^\infty_c(D \times \R)$. We denote by $R_\eta > 0$ a radius such that, for all $t \ge0$, the support of $\eta(\cdot, t)$ is contained in the ball $B_{R_\eta}\subset \R^d$.
	\item Let $ t_0\ge0$,  we define $\ind^{\eps}_{t_0}(t)$ as a smooth approximation of $\ind_{[t_0,+\infty)}$ of size $\eps^6$, but this time with the approximation that is not centered in $t_0$, i.e.,
	\begin{equation}\label{e:approx_caratteristica_begin}
		\ind^\eps_{t_0} :=
		\begin{cases}
			0 &\text{if}\quad t\le t_0,\\
            1 &\text{if} \quad t \ge t_0+(1+2\eps)\eps^6,
        \end{cases}
	\end{equation}
	again with $\de_t \ind_{t_0}^\eps \equiv \eps^{-6}$ in $(t_0 + \eps^{7}, t_0 + (1-\eps)\eps^6)$ and $|\de_t \ind_{t_0}^{\eps}(t)|\le \eps^{-6}$ for all $t\ge0$.
\end{itemize}

\begin{remark}
    It is straightforward to observe that \Cref{t:main} holds even if we consider the test functions $\eta$ to have bounded support and we require that $\eta \in H^1_{0,L}(D_\infty)$. 
    However, we cannot remove the boundedness assumption on its support without some further compensation, since $\mu$ has no natural integrability assumption and thus
    \[
        \int_{D(t)} \mu\, \eta \dif x \qquad\text{and}\qquad \iint \mu\, \de_t \eta \dif x \dif t,
    \]
    are not well-defined if $\eta(\cdot,t) \not\in L^1(D)$ or $\de_t\eta \not\in L^1(D_\infty)$.
    We will treat again this topic in a subsequent work \cite{Paiano_Velichkov2026:monotonicity}, where we improve \ref{item:as:continuity} showing that the mushy coefficient $\mu(\cdot,t)$ is $L^1(D)$-strong continuous in time, even if $|D|=+\infty$.
\end{remark}

\begin{proof}[Proof of \Cref{t:main}]
   We proceed in several steps. First; in Steps 1-11, we give the detailed proof  in the Dirichlet case $u\in \mathcal U_{\mathcal D}(D,g)$; Step 12 is dedicated to the Neumann problem $u\in \mathcal U_{\mathcal N}(D,g)$. The strategy of the proof is as follows: we aim to use a test function of the form $\eta \ind_{[t_1,t_2]}$ for some $\eta \in C^\infty_c(D\times\R)$, and then pass to the limit as $\eps \to 0$ in the outer variation identity \eqref{e:outer-variation-epsilon} to recover all the desired properties. 
    Unfortunately, 
    both $u_\eps$ and $p_\eps(u_\eps)$ lack the necessary regularity to immediately justify such limits.
    To overcome this issue, we use the family of test functions $\eta \ind_{[t_1,t_2]}^\eps$, where $\ind_{[t_1,t_2]}^\eps := \ind_{t_1}^\eps(1-\ind_{t_2}^{\eps})$ is defined via \eqref{e:approx_caratteristica_begin}. This requires to  track and maintain all energy estimates throughout the limiting process. \medskip

    \paragraph{Step 1} {\it Convergence of $u_\eps$ and $p_\eps(u_\eps)$: proof of \ref{item:main_thm:convergence:a} and \ref{item:main_thm:convergence:b}.}\medskip

    Thanks to \Cref{lemma:energy-bound}, we already know that any family of minimizers $\{u_\eps\}_{\eps>0}$ is uniformly bounded in $H^1(D_T)$, for all $T>0$.
    Thus we can take a sequence $\eps_n\to0$ such that:
    \begin{itemize}
        \item $u_{\eps_n} \weak u$ {weakly in} $H^1(D_T)$ for all $T>0$; 
        \item $u_{\eps_n} \rightarrow u$ {strongly in} $L^2(D_T \cap C_R)$ for all $R,T>0$;
        \item $u_{\eps_n}(x,t) \to u(x,t)$ for almost-every $(x,t)\in D\times(0,+\infty)$,
    \end{itemize}
    where the limit function $u$ is in $\mathcal U(D,g)$. Moreover, since $|p_\eps(u_\eps)|\le 1$, we can also suppose that
    \begin{itemize}
        \item $p_{\eps_n}(u_{\eps_n})\weak\mu$ weakly in $L^2(D_T \cap C_R)$, for every $R,T>0$,
    \end{itemize}
    where $\mu\in L^\infty(D\times(0,+\infty))$ with $|\mu|\le 1$. \qed
        \medskip

    \begin{note}
        For the sake of readability, from now on we write $\eps=\eps_n$ and $\eps\to0$, meaning everywhere that we are working up to subsequence and we are taking the limit as $n\to+\infty$.
    \end{note}

    \paragraph{Step 2}{\itshape Proof of the energy bound \eqref{e:main_thm:temperature}.}\medskip

    Since $u_\eps \weak u$ in $H^1(D_T)$ for all $T>0$, then $\nabla u(t) \in L^2(D;\R^d)$ for almost every time $t\ge0$.
    Moreover, the semicontinuity of the norm under the $L^2$-weak convergence implies that for all $t_0 \ge0$ and $\tau>0$
    \[
        \int_{t_0 - \tau}^{t_0 + \tau} \int_{D} |\nabla u |^2 \dif x \dif t 
        \le \liminf_{\eps\to0}\int_{t_0 - \tau}^{t_0 + \tau} \int_{D} |\nabla u_\eps|^2 \dif x \dif t
        \le 16\,\tau \big(\|\nabla g \|_{L^2(D)}^2 + \|F\|_{L^2(D)}^2\big),
    \]
    with the convention that the integrals are extended to zero for negative times.
    Therefore, we have the universal bound 
    \[
    \frac{1}{2\tau}\int_{t_0-\tau}^{t_0+\tau} \| \nabla u(t)\|_{L^2(D)}^2 \dif t \le C,
    \]
    for all $t_0 \ge0$ and $\tau>0$, so $\nabla u \in L^\infty((0,+\infty);L^2(D;\R^d))$.
    Moreover, $\{u(t)\}_{t\ge0}$ are locally (in time) bounded in $H^1(D)$; thus for all $t_0\ge0$ there exists $w\in H^1(D)$, $w = g$ on $\de D$, such that (up to subsequences)
    \[
        u(\cdot,t) \xrightharpoonup[t\to t_0]{} w \qquad \text{weakly in}\quad H^1(D) 
        \qquad\text{and}\qquad
        u(x,t) \xrightarrow[t\to t_0]{} w(x) \qquad\text{for almost every}\quad x\in D.
    \]
    However, since $u\in \Uc$, up to subsequences we have $u(\cdot,t) \to u(\cdot,t_0)$ almost everywhere in $D$, and thus $w = u(\cdot,t_0)$ and the norm weak lower-semicontinuity concludes the proof of \ref{item:main_thm:temperature}.\qed\medskip  
        
    \paragraph{Step 3}
    {\it The limit functions $(u,\mu)$ solves the interior \eqref{e:weak_solution}, that is,
    \begin{equation}\label{e:main_thm:interior_stefan}
        \iint (u+\mu) \de_t\eta - \nabla u \cdot \nabla \eta  - F \cdot \nabla\eta \dif x \dif t, \qquad \text{for all}\quad \eta \in C^\infty_c(D\times(0,+\infty)).
    \end{equation}  }
    \medskip
    Let $\eta\in C^{\infty}_c(D\times(0,+\infty))$ and let $q_\eps$ be the function from \eqref{e:definition-p-f-k-q-epsilon} and \eqref{e:definition-q-epsilon}:
    \begin{equation*}
        q_\eps:\R\longrightarrow\R\ ,
        \qquad     
        q_\eps(z) = \int_{0}^z \frac{1}{\eps}f_\eps(\zeta)^2 \dif \zeta.
    \end{equation*}
    Then, we can write \eqref{e:outer-variation-epsilon} as
    \begin{equation}
        \iint \Biggl\{ \de_t\Big(u_\eps+q_\eps(u_\eps)\Big) (\eta + \eps^4\de_t \eta) +\nabla u_\eps\cdot \nabla \eta + F \cdot \nabla \eta  + \eps k_\eps(u_\eps) f_\eps(u_\eps) (\de_tu_\eps)^2\, \eta \Biggr\}=0.
    \end{equation}
    Integrating by parts $(u_\eps + q_\eps(u_\eps))$ in the time variable, we get
    \begin{equation*}
        \iint \Biggl\{ -\Big(u_\eps+q_\eps(u_\eps)\Big) \de_t\eta + \eps^4\de_t \Big(u_\eps+q_\eps(u_\eps)\Big) \de_{t} \eta +\nabla u_\eps\cdot \nabla \eta + F \cdot\nabla\eta  + \eps k_\eps(u_\eps) f_\eps(u_\eps) (\de_tu_\eps)^2\, \eta \Biggr\}=0.
    \end{equation*}
    The two $\eps$-terms vanish as $\eps\to0$, since, by \Cref{lemma:estimates-f-k-epsilon} and \eqref{e:bound-time-energy}, we have the bounds
    \[
         \iint k_\eps(u_\eps) f_\eps(u_\eps) (\de_tu_\eps)^2\, \eta \dif x \dif t
        \le C \| \eta\|_{L^\infty} \iint (\de_tu_\eps)^2 \dif x \dif t \le C \| \eta\|_{L^\infty}
    \]
    and
    \begin{align*}
       \iint \de_t \Big(u_\eps+q_\eps(u_\eps)\Big) \de_{t} \eta \dif x\dif t 
       &\le  C\, \left(\iint  \left(1+\frac{1}{\eps}f_\eps(u_\eps)^2\right) |\de_t u_\eps|^2 \dif x \dif t\right)^{\sfrac12}\\
       &\hspace{3cm}\cdot
       \left(\iint  \left(1+\frac{1}{\eps}f_\eps(u_\eps)^2\right) |\de_t \eta|^2 \dif x \dif t\right)^{\sfrac12}
       \\
       &\le C \, \|\de_t \eta\|_{L^2(D_\infty)}\,\eps^{-\sfrac12},
    \end{align*}
    and so they vanish being multiplied by $\eps$ and $\eps^{4}$, respectively.
    Finally, by \Cref{lemma:estimate-q-epsilon} it holds that
    \begin{equation*}
        \left| p_\eps(z,x) - q_\eps(z,x)\right| \le \eps \qquad \text{for all}\quad (z,x)\in\R\times D,
    \end{equation*}
    and therefore, up to an error $\mathcal O(\eps)$, we can replace $q_\eps(u_\eps)$ with $p_\eps(u_\eps)$.
    Thus, the previous estimates imply that the Outer Variation is of the following form
    \begin{equation}\label{e:step3-last-estimate}
        \iint  \Big(u_\eps+p_\eps(u_\eps)\Big)\de_t\eta - \nabla u_\eps\cdot \nabla \eta - F \cdot \nabla \eta  \dif x \dif t=\Oc(\eps),
    \end{equation}
    and so, as $\eps\to0$, we get \eqref{e:main_thm:interior_stefan}.\qed\medskip

\paragraph{Step 4}\label{par:Step3_main_thm}
{\it For all $0\le t_1 < t_2 < +\infty$ and $\eta\in C^\infty_c(D\times\R)$ there exists the limit 
\begin{equation}\label{e:convergence_step4}
	\lim_{\eps\to0} \iint \Big(u_\eps + p_\eps(u_\eps)\Big) \eta(x,t) \de_t(\ind_{t_1}^\eps(1-\ind_{t_2}^\eps) ) \dif x \dif t = - \int_{t_1}^{t_2} \int_D (u + \mu) \de_t \eta - \nabla u \cdot \nabla \eta - F \cdot \nabla \eta \dif x \dif t,
\end{equation}
where $\ind_{t_1}^\eps$ is the one defined in \eqref{e:approx_caratteristica_begin}.
}\medskip

We observe that, since for $\eps < t_2-t_1$, $\de_t \ind_{t_1}^\eps$ and $\de_t \ind_{t_2}^{\eps}$ have disjoint supports, we have that
$$\de_t(\ind_{t_1}^\eps\,(1-\ind_{t_2}^\eps))=\de_t \ind_{t_1}^\eps - \de_t \ind_{t_2}^{\eps},$$
so we only need to prove that, for all $t_0\ge0$, the following limit exists
\begin{equation}\label{e:step4-begin-time}
	\lim_{\eps\to0} \iint \Big(u_\eps + p_\eps(u_\eps)\Big) \eta(x,t) \de_t(\ind_{t_1}^\eps ) \dif x \dif t = - \int_{t_1}^{+\infty} \int_D (u + \mu) \de_t \eta - \nabla u \cdot \nabla \eta - F \cdot \nabla \eta \dif x \dif t.
\end{equation}

Let $\eta\in C^\infty_c(D\times\R)$. We notice that $\eta\ind_{t_0}^\eps \in C^\infty_c(D\times(0,+\infty))$, but it depends on $\eps$, so we cannot apply directly \eqref{e:step3-last-estimate}, in which the test function was fixed. 
Instead, we start again from the outer variation \eqref{e:outer-variation-epsilon} applied to $\ind_{t_0}^\eps\eta$ and, arguing as in \emph{Step 3}, we get
\begin{align}
	\iint \Big(u_\eps+p_\eps(u_\eps)\Big)\eta \de_t(\ind_{t_0}^\eps) \dif x \dif t
	&= -\iint \Big(\big(u_\eps+p_\eps(u_\eps)\big)\de_t\eta - \big(\nabla u_\eps\cdot \nabla \eta\big) - F \cdot \nabla \eta\Big) \ind_{t_0}^\eps \dif x \dif t \label{e:conv-full-line1}\tag{$A_{1,\eps}$}\\
    &\hspace{2cm}+\eps^4\iint\de_t \Big(u_\eps+q_\eps(u_\eps)\Big)\eta\de_t\ind_{t_0}^\eps \dif x \dif t + \mathcal{O}(\eps).  \label{e:conv-full-line3}\tag{$A_{2,\eps}$}
\end{align}
We need to show that the expression on the Left-Hand side admits limit as $\eps\to0$, thus that all \eqref{e:conv-full-line1} and \eqref{e:conv-full-line3} separately converge.
The first term \eqref{e:conv-full-line1} converges because, by \ref{item:main_thm:convergence:a} and \ref{item:main_thm:convergence:b}, $u_\eps$ and $p_\eps(u_\eps)$ are (locally) weakly convergent, $\eta$ is compactly supported and $\ind_{t_0}^\eps \to \ind_{[t_0,+\infty)}$ pointwise. Specifically, it converges to the Right-Hand side of \eqref{e:convergence_step4}, and thus we need to show that  \eqref{e:conv-full-line3} vanishes in the limit.
To deal with \eqref{e:conv-full-line3}, we observe that, thanks to \eqref{e:approx_caratteristica_begin} we have
\[
    0\le \de_t \ind_{t_0}^\eps(t) \le \frac{1}{\eps^6} \ind_{[t_0,t_0+2\eps^{6}]}(t),
\]
so, for all $R>0$, it holds
\[
\| \de_t \ind_{t_0}^\eps \|_{L^2(D\cap C_R)} \le |D \cap B_R|^{\sfrac12}\; \eps^{-3}.
\]
In particular, together with the energy bound \eqref{e:bound-time-energy} and the H\"older inequality, the previous estimate implies that
\[
\begin{split}
	\left|\iint \de_t \Big(u_\eps+p_\eps(u_\eps)\Big)\eta \de_t(\ind_{t_0}^\eps) \dif x \dif t \right|
		&= \left| \iint \left(1 + \frac1\eps f_\eps(u_\eps)^2\right)\de_t u_\eps \eta \de_t \ind_{t_0}^\eps \dif x \dif t \right|\\
		&\le
				\left| \iint \left(1 + \frac1\eps f_\eps(u_\eps)^2\right)(\de_t u_\eps)^2  \dif x \dif t \right|^{\sfrac12} \\
		&\hspace{2cm}			\cdot
				\left| \iint \left(1 + \frac1\eps f_\eps(u_\eps)^2\right) \eta^2 (\de_t \ind_{t_0}^\eps)^2 \dif x \dif t \right|^{\sfrac12}\le C\, \eps^{-\sfrac72},
\end{split}
\]
where $C = C(d,F,g) |D \cap B_{R_\eta}|^{\sfrac12}\|\eta\|_{L^\infty(D_\infty)}$.

Hence, for $\eps \to 0$, it holds that $\text{\eqref{e:conv-full-line3}} = \mathcal O(\eps^{\sfrac12})$, and thus
\begin{align*}
	\iint \Big(u_\eps+p_\eps(u_\eps)\Big)\eta \de_t(\ind_{t_0}^\eps) \dif x \dif t
	&= -\iint \Big(\Big(u_\eps+p_\eps(u_\eps)\Big)\de_t\eta- \nabla u_\eps\cdot\nabla\eta - F \cdot\nabla \eta \Big)\ind_{t_0}^\eps \dif x \dif t + \mathcal{O}(\eps^{\sfrac12}),
\end{align*}
and so we conclude the proof of \eqref{e:convergence_step4}.\qed
\medskip

\paragraph{Step 6}
{\it Definition of $\mu(\cdot,t)$ for almost every time $t\ge0$.
}\medskip

Since $\mu\in  L^\infty(D\times(0,+\infty))$, its definition of for almost-every time follows from Fubini's theorem. 
Let us briefly recall the construction by duality. Let $\mu\in L^\infty(D\times(0,+\infty))$ and let $\mathcal N$ be a countable set of functions in $C^\infty_c(D)$, which is dense in $L^1(D)$.
Since $\mathcal N$ is countable, we can find a set of times $\mathcal T\subset(0,+\infty)$ such that $\mathcal L^1((0,+\infty)\setminus\mathcal T)=0$ and such that every $t\in \mathcal T$ is a Lebesgue point for every function 
$$T_{\mu,\eta}(t):=\int_{D}\mu(x,t)\eta(x)\dif x,$$
with $\eta\in\mathcal N$, which means that for all $t\in \mathcal T$ and all $\eta\in\mathcal N$ it holds
$$
    T_{\mu,\eta}(t)
        =\lim_{\tau\to 0}\frac{1}{2\tau}\int_{t-\tau}^{t+\tau}T_{\mu,\eta}(s)\dif s
            \qquad\text{and}\qquad 
        \lim_{\tau\to 0}\frac{1}{2\tau}\int_{t-\tau}^{t+\tau}|T_{\mu,\eta}(s)-T_{\mu,\eta}(t)|\dif s=0.
$$
Since, we have the inequality
$$|T_{\mu,\nu}(t)|\le \|\mu\|_{L^\infty(D_\infty)}\|\eta\|_{L^1(D)},$$
we get that there is a function $\mu(\cdot,t)\in L^\infty(D)$ such that 
$$\|\mu(\cdot,t)\|_{L^\infty(D)}\le \|\mu\|_{L^\infty(D_\infty)}\qquad\text{and}\qquad T_{\mu,\eta}(t)=\int_D\mu(x,t)\eta(x)\,dx\quad\text{for all}\quad \eta\in L^1(D).$$
In particular, the density of $\mathcal N$ implies that, for every $t\in\mathcal T$, we have
\begin{align*}
\int_D\mu(x,t)\eta(x)\dif x&=\lim_{s\to0}\frac1{s}\int_{t-s}^{t}\int_D\mu(x,\tau)\eta(x)\dif x \dif \tau
=\lim_{s\to0}\frac1{s}\int_{t}^{t+s}\int_D\mu(x,\tau)\eta(x)\dif x \dif \tau,
\end{align*}
for all $\eta \in C^\infty_c(D)$ and
\begin{align*}
\int_D\mu(x,t)\eta(x,t)\dif x
    &=\lim_{s\to0}\frac1{s}\int_{t-s}^{t}\int_D\mu(x,\tau)\eta(x,\tau)\dif x \dif \tau
    =\lim_{s\to0}\frac1{s}\int_{t}^{t+s}\int_D\mu(x,\tau)\eta(x,\tau)\dif x \dif \tau,
\end{align*}
for all $\eta\in C^\infty_c(D\times \R)$.
\qed\medskip

\paragraph{Step 7}
{\it For all $\eta\in C^\infty_c(D\times (0,+\infty))$, and for all $t_0 \in \mathcal T$ there exist the following limits:
\begin{equation}\label{e:formula_for_Lebesgue_time}
	\lim_{\eps\to0} \iint (u_\eps + p_\eps(u_\eps)) \eta\,\de_t(\ind_{t_0}^\eps) \dif x \dif t =
    \int_{D(t_0)} (u + \mu ) \eta \dif x.
\end{equation}}
\medskip

By construction, $\de_t\ind_{t_0}^{\eps}$ and $\eps^{-6}\ind_{[t_0,t_0+\eps^6]}$ differ only in two intervals of size $\eps^7$.
Therefore it holds that
\[
    \left|
        \frac{1}{\eps^6}\int_{t_0}^{t_0+ \eps^6}\int_D (u+\mu)\,\eta \dif x \dif t 
        - \iint (u+\mu)\,\eta\, \de_t(\ind_{t_0}^{\eps}) \dif x \dif t 
    \right| 
        \le C( D,\eta,g,\bar t)\, \eps,
\]
where $\eta\in C^\infty_c(D\times(0,+\infty))$. Thus if one of the two admits limit as $\eps\to0$, the same holds for the second one and the two limits coincide. Thanks to the previous Step 6 and to the fact that $u$ is Sobolev in space-time, $u\in \Uc(D,g)$, we have that
\[
\int_{D(t_0)} (u+\mu) \,\eta \dif x
\]
is well-defined and it is the limit of the mean values around that time, i.e.,
\begin{align*}
\int_{D(t_0)} (u + \mu ) \,\eta \dif x 
	&= \lim_{\eps\to0} \frac{1}{\eps^6}\int_{t_0}^{t_0+ \eps^6}\int_D (u+\mu)\,\eta \dif x \dif t=\lim_{\eps\to0} \iint (u+\mu)\eta \de_t(\ind_{t_0}^{\eps}) \dif x \dif t.
\end{align*}
Now, by \textit{Step 3}, $(u,\mu)$ solves \eqref{e:main_thm:interior_stefan} for all smooth functions with compact support. Thus, since for all $t_0>0$ $\eps < t_0$, $\eta {\ind}^\eps_{t_0} \in C^\infty_c(D\times(0,+\infty))$, we have
\begin{align*}
    \iint (u+\mu)\,\eta \,\de_t(\ind_{t_0}^{\eps}) \dif x \dif t
        &=  - \iint (u+\mu)\,\Big( \de_t\eta\,\ind_{t_0}^{\eps} - \de_t(\eta\,\ind_{t_0}^{\eps}) \Big) \dif x \dif t \\
        &= - \iint \left((u+\mu) \de_t \eta - \nabla u \cdot \nabla \eta - F \cdot \nabla \eta \right) {\ind}_{t_0}^\eps(t) \dif x \dif t.
\end{align*}
Therefore it holds
\begin{align*}
\int_{D(t_0)} (u + \mu ) \,\eta \dif x 
	&= -	\lim_{\eps\to0} \iint (u+\mu)\,\eta \,\de_t(\ind_{t_0}^{\eps}) \dif x \dif t \\
	&= - \lim_{\eps\to0} \iint\left( (u+\mu) \de_t \eta - \nabla u \cdot \nabla \eta - F \cdot \nabla \eta \right) \ind_{t_0}^\eps \dif x \dif t\\
	&= - \int_{t_0}^{+\infty}\int_D \left((u+\mu) \de_t \eta - \nabla u \cdot \nabla \eta - F \cdot \nabla \eta \right) \dif x \dif t\\
    &=\lim_{\eps\to0} \iint \Big(u_\eps + p_\eps(u_\eps)\Big) \eta(x,t) \de_t(\ind_{t_0}^\eps ) \dif x \dif t,
\end{align*}
where the last inequality is due to \eqref{e:step4-begin-time}. This concludes the proof of \textit{Step 7.}
\qed
\medskip

With the identity from Step 7, we are now in position to define $\mu(\cdot,t)$ for every time $t\ge 0$.

\paragraph{Step 8}
{\it
For all $t\ge0$, $\mu(\cdot,t) \in L^\infty(D)$ is well defined as
$$\int_D\mu(x,t)\eta(x)\dif x:=\lim_{\substack{\tau\to t \\ \tau\in \mathcal T}}\int_D\mu(x,\tau)\eta(x)\dif x\quad\text{for all}\quad \eta\in C^\infty_c(D).$$ 
In particular, $\mu(\cdot,t)$ is continuous with respect to the weak-$\ast$ topology and,
for all $0 \le t_1 < t_2$, all $\eta\in C^\infty_c(D\times\R)$ the integral identity \eqref{e:weak-equation-F} holds.
}\medskip

By {\it Step 4} and {\it 7}, \eqref{e:weak-equation-F} holds for all $0<t_1 < t_2$ with $t_1,t_2 \in \mathcal T$, and all $\eta\in C^\infty_c( D\times\R)$. Thus, for every $\eta\in C^\infty_c(D)$ we have 
\[
        \int_{D(t_2)}(u+\mu)\eta\dif x - \int_{D(t_1)}(u+\mu)\eta\dif x
        = -\int_{t_2}^{t_1}\int_{D} \nabla u \cdot \nabla \eta + F \cdot \nabla \eta \dif x \dif t,
\]
which implies 
\begin{align*}
\left|\int_{D}\mu(x,t_2)\eta(x)\dif x-\int_{D}\mu(x,t_1)\eta(x)\,dx\right|&\le \|\eta\|_{L^2(D)}\|u(\cdot,t_2)-u(\cdot,t_1)\|_{L^2(D)}\\
&\qquad +\|\nabla \eta\|_{L^2(D)}\int_{t_2}^{t_1}\left(\|\nabla u(\cdot,t)\|_{L^2(D)}+\|F\|_{L^2(D)}\right)\dif t,
\end{align*}
for all $t_1,t_2\in \mathcal T$. Thus, the limit 
$$T_{\mu,\eta}(t):=\lim_{\substack{\tau\to t \\ \tau\in \mathcal T}}\int_D\mu(x,\tau)\eta(x)\dif x,$$
exists for every $t\ge 0$. Moreover, since we have the bound
$$|T_{\mu,\eta}(t)|\le \|\eta\|_{L^1(D)}\|\mu\|_{L^\infty(D_\infty)},$$
we get the existence of a function $\mu(\cdot,t)\in L^\infty(D)$ such that 
$$T_{\mu,\eta}(t)=\int_D\mu(x,t)\eta(x)\dif x.$$
Finally, the validity of \eqref{e:weak-equation-F} for all times follows by passing to the limit \eqref{e:weak-equation-F} for times in $\mathcal T$. This concludes the proof of Step 8. \qed\medskip

We already showed that $\mu(\cdot,t)$ is well defined, is weak-$\ast$ continuous in time and that the couple $(u,\mu)$ solves the integral identity \eqref{e:weak-equation-F}. 
It is left to show the compatibility condition as well as \ref{item:main_thm:convergence:c}.
Also here the scale of $\de_t \ind_{t_0}^\eps$ plays a fundamental role.\medskip

 In the final steps of the proof, we will use the following well-known properties of the traces of Sobolev functions (see for instance Evans \cite{Evans2010:bookPDE} or Maz'ya \cite{Mazya2011:BookSobolev}):
\begin{itemize}
   
    \item If $v\in H^1(D\times(0,T))$ for some $T>0$, then the trace $v(\cdot,t)\in L^2(D)$ exists for every $t\in[0,T]$, and we have the following estimate
\begin{equation}\label{e:holder-continuity-traces}
            \|v(\cdot,t) - v(\cdot,s) \|_{L^2(D)} \le \|\de_t v\|_{L^2(D_T)} |t-s|^{\sfrac12}\quad\text{for all}\quad 0\le s<t\le  T.
\end{equation}
      \item Suppose that $v_n\in H^1(D\times(0,T))$ converges to $v\in H^1(D\times(0,T))$ weakly in $H^1(D\times(0,T))$. Then, all the traces converge of $v_n$ converge to the traces of $v$ strongly in $L^2_{loc}(D)$, that is 
      \begin{equation}\label{e:strong-convergence-traces}
       \lim_{n\to+\infty} \|v_n(\cdot,t)- v(\cdot,t)\|_{L^2(D\cap B_R)}=0\quad\text{for all}\quad t\in[0,T] \quad \text{and} \quad R>0.
    \end{equation} 
\end{itemize}

\paragraph{Step 9}
{\it
For all $t_0\ge0$ it holds the following limit
\[
    \int_\R p_\eps(u_\eps) \,\de_t(\ind_{t_0}^\eps) \dif t
        \xrightharpoonup[\eps\to0]{\ast} \mu(\cdot,t_0) \qquad \text{weakly-$\ast$ in}\quad L^\infty(D). 
\]
}\medskip

First, we observe that, fixed $\eta\in C^\infty_c(D\times \R)$, for all $t_0\ge0$, it holds
\begin{equation*}
	\lim_{\eps\to0} \iint (u_\eps + p_\eps(u_\eps)) \eta\,\de_t(\ind_{t_0}^\eps) \dif x \dif t =
    \int_{D(t_0)} (u + \mu ) \eta \dif x.
\end{equation*}
Indeed, the identity holds for $t_0\in\mathcal T$ and the right-hand side is continuous in time by \eqref{e:convergence_step4}, while the left-hand side by {\it Step 8}.
To complete the proof, thanks to the linearity of the weak-$\ast$ limit, we reduce ourselves to prove that
\[
    \int_\R u_\eps\,\de_t(\ind_{t_0}^\eps) \dif t
        \xrightharpoonup[\eps\to0]{\ast} u_\eps(\cdot,t_0) \qquad \text{weakly-$\ast$ in}\quad L^\infty(D). 
\]

To do so, we observe that, for all $\eta\in C^\infty_c(D\times\R)$ by \eqref{e:strong-convergence-traces}
$$\lim_{\eps\to0}\int_{D(t_0)} u_\eps\, \eta \dif x =\int_{D(t_0)} u\, \eta \dif x.$$
On the other hand, by the continuity of $\eta$ in time we have 
\begin{align*}
    \iint u_\eps\, \eta \, \de_t \ind_{t_0}^\eps \dif x\dif t 
        &= \iint u_\eps(x,t) \eta(x,t_0) \de_t\ind_{t_0}^\eps(t)  \dif x \dif t+o(1)\\
        &= \int_{D} u_\eps(x,t_0) \eta(x,t_0) \dif x \\
        &\hspace{1cm}+ \iint \left( u_\eps(x,t) - u_\eps(x,t_0)\right) \eta(x,t_0) \, \de_t \ind_{t_0}^\eps \dif x \dif t 
        + o(1).
\end{align*}
Finally, by the energy bound \eqref{e:bound-time-energy}, $\| \de_tu_\eps\|_{L^2(D_\infty)}\le C$ and the estimate \eqref{e:holder-continuity-traces}, we get 
\begin{align*}
\left|\int_D\left( u_\eps(x,t) - u_\eps(x,t_0)\right) \eta(x,t_0)\, \dif x\right|&\le  \|\eta(\cdot,t_0)\|_{L^2(D)}\|u_\eps(\cdot,t) - u_\eps(\cdot,t_0) \|_{L^2(D)} \le C |t-t_0|^{\sfrac12}.
\end{align*}
Since $|t-t_0|\le \eps^6$ when $t$ is in the support of $\de_t \ind_{t_0}^\eps$, integrating in time, we get 
$$\iint \left( u_\eps(x,t) - u_\eps(x,t_0)\right) \eta(x,t_0) \, \de_t \ind_{t_0}^\eps \dif x \dif t \le \frac{C}{\eps^6}\, \int_{t_0}^{t_0+\eps^6} |t-t_0|^{\sfrac12} \dif t \le C \eps^{3},$$
where $C=C(d,g,D,\eta)>0$ is as in \textit{Step 4}. 
Then, by taking $\eps\to0$, we get 
\begin{align*}
  \lim_{\eps\to0}  \iint u_\eps\, \eta \, \de_t \ind_{t_0}^\eps \dif x\dif t = \lim_{\eps\to0}\int_{D(t_0)} u_\eps\, \eta \dif x =\int_{D(t_0)} u\, \eta \dif x,
\end{align*}
which concludes the proof of \textit{Step 9}.
\qed\medskip

\paragraph{Step 10}
{\it 
For all $t_0\ge0$, the limit \eqref{e:main_thm:convergence_mu} holds.
}\medskip

We need to prove that
\[
    p_\eps(u_\eps(\cdot,t_0)) \xrightharpoonup[\eps\to0]{\ast} \mu(\cdot,t_0) \qquad \text{weakly-$\ast$ in}\quad L^\infty(D).
\]
By \Cref{lemma:estimates-f-k-epsilon}, it holds that for all $t,s\ge0$ it holds
\[
    |p_\eps(u_\eps(x,t),x) - p_\eps(u_\eps(x,s),x)| \le \frac1\eps |u_\eps(x,t) - u_\eps(x,s)|.
\]
For the sake of brevity, let us adopt the notation $p_\eps(x,t) := p_\eps(u_\eps(x,t),x)$ in this step.
By {\it Step 9} it is enough to prove that for all $\eta \in C^\infty_c(D)$ it holds
\begin{equation}\label{e:step9-convergence-fixed-time}
    \lim_{\eps\to0}
    \left|
        \iint p_\eps(x,t) \,\eta\, \de_t \ind_{t_0}^\eps \dif x \dif t 
        - \int_D p_\eps(x,t_0)\, \eta \dif x
    \right|=0,
\end{equation}
the statement follows by density and by the triangular inequality.

For all $t_0\ge0$ and $\eps>0$ it holds
\[
    \int_{\R} \de_t \ind_{t_0}^{\eps}(t)\dif t =1 \qquad \text{and}\qquad \de_t \ind_{t_0}^\eps(t) \ge0,
\]
and thus, for all $\eta\in C^\infty_c(D)$, we get that
\begin{align*}
    \left|
        \iint p_\eps(x,t) \,\eta(x)\, \de_t \ind_{t_0}^\eps \dif x \dif t 
        - \int_D p_\eps(x,t_0)\, \eta(x) \dif x
    \right|
    &= 
        \left| 
            \iint\Big( p_\eps(x,t) - p_\eps(x,t_0)\Big)\eta(x)\,  \de_t \ind_{t_0}^\eps \dif x \dif t
        \right|\\
    &\le \|\eta\|_{L^2(D)} \iint\big\| p_\eps(x,t) - p_\eps(x,t_0)\big\|_{L^2(D)} \de_t \ind_{t_0}^\eps \dif x \dif t.
\end{align*}
Moreover, by \Cref{lemma:estimates-f-k-epsilon} and the trace regularity of $H^1$-functions, we get that
\[
    \big\| p_\eps(x,t) - p_\eps(x,t_0)\big\|_{L^2(D)} \le \frac{1}{\eps} \big\| u_\eps(\cdot,t) - u_\eps(\cdot,t_0)\big\|_{L^2(D)} \le \frac1\eps |t-t_0|^{\sfrac12}
\]
and then since $\de_t \ind_{t_0}^\eps\le \eps^{-6}\ind_{[t_0,t_0+2\eps^6]}$, we obtain
\[
    \iint\big\| p_\eps(\cdot,t) - p_\eps(\cdot,t_0)\big\|_{L^2(D)} \de_t \ind_{t_0}^\eps \dif x \dif t
        \le \frac{1}{\eps^{6}}\int_{t_0}^{t_0+2\eps^6}\frac1\eps |t-t_0|^{\sfrac12} \le 4 \eps^{2},
\]
that implies \eqref{e:step9-convergence-fixed-time} and so it concludes the proof of {\it Step 10}.\qed

\paragraph{Step 11}{\it The compatibility property \ref{item:main_thm:enthalpy:b} holds for all $t_0\ge0$.}\medskip

We will show that for all $t_0\ge0$ there exists a subsequence $\{\eps_n\}_{n\in\NN}$, $\eps_n\to0$, such that
\[   p_{\eps_n}(u_{\eps_n}(\cdot,t_0)) \xrightharpoonup[n\to+\infty]{\ast} \pm 1 \qquad\text{weakly-$\ast$ in}\quad L^\infty(\Omega_u^\pm(t_0));
\]
and thus we conclude, since the whole sequence $p_{\eps_n}(u_{\eps_n}(\cdot,t_0))$ converges weakly-$\ast$ to $\mu(x,t)$ as $\eps\to0$.
We will prove the statement for the positive part; 
the negative counterpart follows by an analogous argument.\medskip

Since $u_\eps\weak u$ in $\Uc$, then by \eqref{e:strong-convergence-traces} we have
\[
    u_\eps(\cdot,t_0) \xrightarrow[\eps\to0]{} u(\cdot,t_0) \qquad\text{strongly in}\quad L^2(D\cap B_R),
\]
for all $t_0\ge0$ and $R>0$. 
Thus there exists a (sub)sequence $\{\eps_n\}_{n\in\NN}$ such that
\[
    u_{\eps_n}(x,t_0) \xrightarrow[j\to+\infty]{} u(x,t_0) \qquad\text{for almost-every}\quad x\in D.
\]
In particular, for almost every $x\in \Omega_u^+(t_0)$, there exists $\bar n = \bar n(x,t_0)\gg1$ such that 
\[
    u_{\eps_n}(x,t_0) \ge \frac{1}{2} u(x,t_0), 
\]
and hence there exists $\hat n = \hat n(x,t_0)$ such that
\[
    p_{\eps_n}(u_{\eps_n}(x,t_0)) = 1 \qquad\text{for all}\quad n\ge\hat n;
\]
so $p_{\eps_n}(u_{\eps_n}(\cdot,t_0))$ converges almost everywhere to $1$ in $\Omega_u^+(t_0)$.
This concluded the proof of {\it Step 11}.

\paragraph{Step 12}{\it Proof of \ref{item:main_thm:initial-mushy-coefficient}.}
\medskip

    By construction, $u_\eps(x,0) = g(x)$ for all $\eps>0$, thus by {\it Step 10} we only need to prove that
    \begin{equation}\label{e:boh-time-zero}
		\lim_{\eps\to0} p_\eps(g,x) = \ind_{\Omega_g^+}(x) - \ind_{\Omega_g^-}(x) + h(x) \ind_{\{g=u_M\}}(x) \qquad \text{for almost every}\quad x\in D,
    \end{equation}
    then the Dominate convergence theorem applies and thus the limit holds in $L^p(D \cap B_R)$, for all $p\ge1$ and $R>0$, and also   weakly-$\ast$ in $L^\infty(D)$.
    However, \eqref{e:boh-time-zero} comes straightforwardly by the definition of $p_\eps$ as
    \[
		\lim_{\eps\to0} p_\eps(z,x) = \ind_{\{z>u_M\}}(x) - \ind_{\{z<u_M\}}(x) + h(x) \ind_{\{z=u_M\}}(x) \qquad\text{for all}\quad (z,x) \in \R\times D.
	\]
    This conclude the proof of {\it Step 12} and hence the proof of \Cref{t:main} in the Dirichlet case $u\in \mathcal U_{\mathcal D}(D,g)$.
\qed

\paragraph{Step 13}{\it The Neumann case.}\medskip

The Neumann case $u\in \mathcal U_{\mathcal N}(D,g)$ follows exactly the same steps. We notice that \emph{Steps 1-3} involve only the energy estimates from \Cref{sec:energy-estimates}, which also hold for minimizers in $\Uc_{\mathcal N}(D,g)$.
\emph{Steps 4-12} are local properties that involve perturbations only along the time direction and therefore, they are equally valid in the Neumann case and hence, we conclude the proof of \Cref{t:main}.
\end{proof}

\section{Uniqueness of the Cauchy problem}\label{sec:initial_enthalpy}

In this section we discuss the uniqueness of solutions of the Stefan problem and of the convergence of the above scheme.
Specifically, we prove the comparison principles (\Cref{t:comparison-dirichlet,t:comparison-neumann}) and yet we prove that the limit of the elliptic approximation scheme is unique and that every enthalpy solution (with time-independent Dirichlet boundary condition or Neumann boundary condition) is the limit of such scheme, and it is completely determined by the choice of $g$, $h$, $F$, and $u_M$.\medskip

The strategy of the proof of \Cref{t:comparison-dirichlet} and \Cref{t:comparison-neumann} follows some general ideas that can be traced back to the work of Kamin \cite{Kamenomostskaya} and it has seen a series of generalization (see, for instance, \cite{Friedman_1968,Meirmanov_1992,DingDuGuo2021:StefanProblemFisher-unbounded}).
The main technical novelty in our proof of \Cref{t:comparison-dirichlet} lies in the choice of the auxiliary problem \eqref{e:appendix-system-unbounded}, that allows to overcome the integrability issues in the Gr\"onwall like argument.

\begin{proof}[\bf Proof of \Cref{t:comparison-dirichlet}]
	Taking the difference of the inequalities for $u_2$ and $u_1$, we get
	\begin{equation}\label{e:comparison-H1}
		\int_{D(t)} V\phi \dif x \bigg|_{t=t_1}^{t_2} \ge \int_{t_1}^{t_2} \int_D V \,\de_t\phi - \nabla U \cdot\nabla \phi \dif x \dif t,
	\end{equation}
	for all $\phi \in C^\infty(D\times\R)$, where 
	\[
		U(x,t) := u_2(x,t) - u_1(x,t) 
			\qquad \text{and}\qquad
		V(x,t) := U(x,t) + (\mu_2(x,t)-\mu_1(x,t)).
	\]
	The underlying  idea of the proof does not change whether $D\subset\R^d$ is a bounded or unbounded set. 
	However, some technical adjustments are needed to deal with the unbounded case, and therefore, we treat them separately. 
	In particular, the bounded case will be essentially equivalent to the one presented in \cite{Meirmanov_1992}, but since we need some adjustments for the unbounded case, we decided to give the full proof in both cases.	Fixed  $T\in(0,T^\ast)$, we introduce two families that will be useful in both proofs.
	
	    \begin{enumerate}[{\itshape(i)}]
            \setlength\itemsep{0.4cm}

        \item We consider a sequence of smooth bounded open sets  $D^n\subset D$, $n\in\mathbb N$, such that $D^n\uparrow D$ as $n\to+\infty$. We also set $D_T^n=D^n\times[0,T)$. By the divergence theorem, for all nonnegative functions $\phi\in H^1_0(D_T)$  with $\{\phi>0\} = D^n_T$ and $\phi \in H^2(D^n_T)$ we can recast \eqref{e:comparison-H1} in $D_T$ as
         	\begin{equation}\label{e:comparison-H2}
				\int_{D^n(T)} V\phi \dif x  \ge \iint_{D^n_T} V\big(\de_t\phi + A\,\Delta \phi\big) \dif x \dif t + \int_{0}^{T}\int_{\de D^n} U |\nabla \phi| \dif \Hc^{d-1}x \dif t,
         	\end{equation}
            since $V(\cdot,0)\phi(\cdot,0)\ge 0$ in $D$ and $\de_\nu \phi = - |\nabla \phi|$ on the lateral boundary $\de_L D^n_T$, and $A\in L^\infty(D_T)$ is defined as
			\[
				A(x,t) := \begin{cases}\displaystyle\frac{U(x,t)}{V(x,t)}&\quad\text{if}\quad U(x,t)\neq0,\\
                \quad0&\quad\text{if}\quad U(x,t)=0.
                \end{cases}
			\]
			We notice that $A$ is wall defined and satisfies $0\le A \le 1$ almost everywhere in  $D_T$. This is due to the fact that if $u_2 > u_1$, then $\mu_2 \ge \mu_1$ (analogously, if $u_1 > u_2$, then $\mu_1 \ge \mu_2$). This also proves the inclusion
            \[
                 \{(x,t) \in D_T : U(x,t) >0\} \subset \{(x,t) \in D_T : V(x,t) >0 \}.
            \]

        \item 
        There exists a sequence $\{A_n\}_{n\in\NN}\subset C^\infty(\R^d\times \R)$, with the following properties:
        \begin{enumerate}[label={\itshape(ii.\alph*)}]
            \setlength\itemsep{0.2cm}
            \item $0\le A(x,t) \le A_n(x,t) \le 1$ for almost every $(x,t) \in D_T$;
            \item $A_n(x,t) \xrightarrow[n\to+\infty]{} A(x,t)$ for almost every $(x,t) \in D_T$;
            \item $(A_n-A)\xrightarrow[\phantom{n\to\infty}]{}0$ strongly in $L^1 (D_T)\cap L^2(D_T)$.
        \end{enumerate}
    \end{enumerate}

	\paragraph{Case 1: $D\subset\R^d$ bounded}
    Let us fix $\eps>0$ and $\phi_0\in C^\infty_c(D)$ with $0\le \phi_0 \le 1$ in $D$. For every $n\in\mathbb N$ large enough, such that $\phi_0\in C^\infty_c(D^n)$, we consider the solution $\phi_{\eps,n}$ of the reversed heat system
    \begin{equation}\label{e:appendix-system-bounded}
    \begin{cases}
        \de_t \phi + (\eps e^{-|x|} + A_n) \Delta \phi =  \phi & \text{in}\quad D^n\times(0,T)\\
        \phi(x,T) = \phi_{0}(x) &\text{on} \quad D^n\\
        \phi \equiv0 & \text{on}\quad \de D^n \times (0,T),
    \end{cases}
    \end{equation}
    where $A_n(x,t)$ is defined as above. 
	We are interested in passing the family $\{\phi_{\eps,n}\}_{\eps>0,n\in\NN}$ to the limit; to do so, we now investigate uniform energy bounds on them.
	Firstly, for all     $\tau<T$, multiplying the equation in \eqref{e:appendix-system-bounded} by $\Delta \phi_{\eps,n}$ and integrating it in $D^n\times(\tau,T)$, we get
    \begin{align*}
    \frac12\|\nabla\phi_{\eps,n}\|_{L^2(D^n(\tau))}^2 
        + \int_{\tau}^T\int_{D^n} (\eps e^{-|x|} &+ A_n)|\Delta \phi_{\eps,n}|^2 \dif x \dif t \\
    &=
        \frac12\|\nabla \phi_{0}\|_{L^2(D^n(T))}^2 -  \int_{\tau}^T\int_{D^n} |\nabla \phi_{\eps,n}|^2 \dif x \dif t,
    \end{align*}
    and so
    \[
    \|\nabla\phi_{\eps,n}\|_{L^2(D^n(\tau))}^2 
        \le  C \|\nabla \phi_0\|_{L^2(D)}^2 \qquad\text{for all} \quad \tau \le T.
    \]
    Using the above estimate and the Poincaré inequality in time, and by testing \eqref{e:appendix-system-bounded} with $\partial_t\phi_{\eps,n}$, we deduce that there exists a constant $C = C(d,\phi_0,T)>0$, that does not depend on $\eps$ and $n$, such that
    \begin{equation}\label{e:comparison-energy}
         \|\phi_{\eps,n}\|_{L^2(D_T)}^2+\|\de_t \phi_{\eps,n}\|_{L^2(D_T)}^2 + \sup_{0\le\tau\le T} \| \nabla \phi_{\eps,n} \|_{L^2(D(\tau))}^2 \le C,
    \end{equation}
    and 
    \begin{equation}\label{e:comparison-laplacian}
        \iint_{D^n_T} \big(\eps e^{-|x|} + A_{n}\big) |\Delta \phi_{\eps,n}|^2 \dif x \dif t \le C.
    \end{equation}
    For all fixed $\eps>0$, the sequence is then weakly compact in $H^1(D_T)$ and in $H^2(D^{n_0}_T)$, for all fixed $n_0\in \NN$, and so as $n\to+\infty$ they converge to a certain $\phi_\eps \in H^1(D_T)$ (up to subsequences), and, for all $\eps>0$, the limit $\phi_\eps$ is such that
    \[
        0\le \phi_\eps \le 1 \quad\text{in}\quad D_T, \qquad \phi_\eps(x,0) = \phi_0(x) \quad \text{in} \quad D, \qquad\text{and}\qquad \phi_\eps \equiv 0 \quad \text{on}\quad \de_LD_T,
    \]
    and it satisfies 
    \begin{equation}\label{e:comparison-energy-2}
         \|\phi_{\eps}\|_{L^2(D_T)}^2+\|\de_t \phi_{\eps}\|_{L^2(D_T)}^2 + \sup_{0\le\tau\le T} \| \nabla \phi_{\eps} \|_{L^2(D(\tau))}^2 \le C,
    \end{equation}
    and 
    \begin{equation}\label{e:comparison-laplacian-2}
        \iint_{D^n_T} \big(\eps e^{-|x|} + A_{n}\big) |\Delta \phi_{\eps}|^2 \dif x \dif t \le C\quad\text{for all}\quad n\in\mathbb N.
    \end{equation}

    Moreover, for all $\eps>0$ and $n\in\NN$, substituting $\phi_{\eps,n}$ into \eqref{e:comparison-H2} we get  
    \begin{equation}\label{e:comparison-stefan-proof-n-eps}
        \int_{D^n(T)} V \phi_{0,n} \dif x  \ge \int_{D_T^n} V \big( \phi_{\eps,n} - \big(\eps e^{-|x|}+ (A_n - A)\big) \Delta \phi_{\eps,n}\big) \dif x \dif t - \int_{0}^{T} \int_{\de D^n} U \de_\nu \phi_{\eps,n} \dif \Hc^{d-1}\dif t,
    \end{equation}
    and thus, as $n\to+\infty$, the integral of $V(A_n-A)\Delta \phi_{\eps,n}$ vanishes since by \eqref{e:comparison-laplacian}, it holds
	\[
		\left\vert\iint_{D^{n}_T} (\mu_2-\mu_1)(A_n - A)\Delta \phi_{\eps,n}\dif x \dif t \right\vert \le C \|A_n - A\|_{L^1(D_T)}^{\sfrac12},
	\]	
	and
	\[
        \left\vert \iint_{D^{n}_T} U (A_n - A)\Delta \phi_{\eps,n}\dif x \dif t\right\vert 
        \le \left(\iint_{D^{n}_T} (A_n - A) U^2\dif x \dif t\right)^{\sfrac12}\left(\iint_{D^{n}_T} (A_n - A)|\Delta \phi_{\eps,n}|^2\dif x \dif t\right)^{\sfrac12},
  	\]
  	converges by Dominate Convergence Theorem, since $A_n\to A$ almost everywhere and by
  	$0\le A\le A_n \le 1$ and $U\in L^2(D_T)$ the first integral is bounded while by \eqref{e:comparison-laplacian} we bound the second one.
  	
    We next analyze the convergence of the boundary term as $n\to+\infty$.
    First, we notice that the (backward) maximum principle implies that, for all $\eps>0$ and $n\in\NN$ it holds
    \[
        0 \le \phi_{\eps,n}(x,t) \le 1 \qquad \text{for all}\qquad (x,t) \in D_n \times[0,T],
    \]
    and thus $0\le\phi_\eps\le 1$. Moreover, since $\de D_n$ is smooth and bounded, it holds
    \[
        \de_\nu \phi_{\eps,n}(x,t) \le0 \qquad\text{for almost every} \quad (x,t) \in (\de D_{n})\times(0,T).
    \]
    Therefore, for all $\eps>0$ and $n\in\NN$, and for all nonnegative $w\in C^\infty_c(\R^d\times\R)$ it follows that
    \begin{align*}
        0 \ge \int_{0}^T \int_{\de D^n} w\, \de_\nu \phi_{\eps,n} \dif \Hc^{d-1} \dif t 
        &= \iint_{D_T^n} \nabla w \cdot \nabla \phi_{\eps,n} + w\, \Delta \phi_{\eps,n} \dif x \dif t \\
        &= \iint_{D_T} \nabla w \cdot \nabla \phi_{\eps,n} +  \frac{w}{\eps e^{-|x|} + A_n} \big( \phi_{\eps,n} -  \de_t \phi_{\eps,n}   \big) \dif x \dif t.
    \end{align*}
    Since $\phi_{\eps,n}$ converge to $\phi_\eps$ weakly in $H^1(D_T)$ as $n\to+\infty$, for all fixed $\eps>0$ and $w\ge0$ of compact support, the right-hand term above also converges. 
    Thus, we proved that there exists a distribution $\sigma \in \mathcal D^*$ such that (up to subsequences) it holds
    \begin{equation}\label{e:comparison-measure-convergence}
        \de_\nu \phi_{\eps,n} (\dif \Hc^{d-1}\otimes\dif t)\llcorner (\de_L D^n_T) \xrightharpoonup[n\to+\infty]{} \sigma_\eps \qquad\text{in}\quad \mathcal D^*\quad \text{and in}\quad H^{-1}(B_R \times(0,T)),
    \end{equation}
    for all $R>0$ and $T>0$. Since $\sigma_\eps$ is negative distribution (by construction), it is represented by a negative Radon (capacitary) measure (that we still denote by $\sigma_\eps$) uniquely characterized  through the identity 
    \[
        \lim_{n\to=\infty}\int_{0}^T \int_{\de D^n} v(x,t)\, \de_\nu \phi_{\eps,n} \dif \Hc^{d-1} \dif t =\int_{\R^d\times(0,T)} v(x,t) \dif \sigma_\eps(x,t) ,
    \]
    for all $v \in H^1(D_T)$.
    Let now $\eta \in C^\infty_c(D\times\R)$. Then, there exists $n\in\NN$ such that $\eta \equiv0$ in $D_T \setminus D_T^n$, so
    \[
        \int_{D_T} \eta(x,t) \dif \sigma_\eps(x,t) = \lim_{n\to+\infty} \int_{0}^T\int_{\de D_n} \eta(x,t)\, \de_\nu \phi_{\eps,n}(x,t) \dif \mathcal H^{d-1}x\dif t = 0.
    \]
    Hence, $\sigma_\eps$ is supported on $\de_L D_T$ and (by density)
    \[
        \int_{D_T} \eta(x,t) \dif \sigma_\eps(x,t)=0 \qquad \text{for all}\quad \eta \in H^1_{0,L}(D_T).
    \]
    Therefore, since $U^- \in H^1_{0,P}(D_T) \subset H^1_{0,L}(D_T)$ and $\sigma_\eps\le 0$, we get that
    \[
        \int_{D_T} U \dif \sigma_\eps = \int_{D_T} (U^+ + U^-) \dif \sigma_\eps = \int_{D_T} U^+ \dif \sigma_\eps \le 0.
    \]
     Finally, thanks to \eqref{e:comparison-laplacian} and \eqref{e:comparison-laplacian-2}, we have that 
    \[
    \lim_{n\to+\infty}\iint_{D_T^n} V   \eps e^{-|x|} \Delta \phi_{\eps,n} \dif x \dif t = \iint_{D_T} V \eps e^{-|x|} \Delta \phi_{\eps} \dif x \dif t,
    \]
    and thus, we proved that, as $n\to+\infty$, \eqref{e:comparison-stefan-proof-n-eps} leads to the following integral inequality
    \begin{equation}\label{e:comparison-stefan-proof-eps}
    \begin{split}
        \int_{D(T)} V \phi_{0} \dif x  
        	\ge \iint_{D_T} V \big( \phi_{\eps} - \eps e^{-|x|} \Delta \phi_{\eps}\big) \dif x \dif t.
    \end{split}
    \end{equation}\medskip

    We now investigate the limit as $\eps\to0$.
    By \eqref{e:comparison-energy-2}, the sequence $\{\phi_\eps\}_{\eps>0}$ is weakly compact in $H^1(D_T)$. Thus, there exists $\phi \in H^1(D_T)$ with
    \[
        0\le \phi \le 1 \quad\text{in}\quad D_T, \qquad \phi(x,0) = \phi_0(x) \quad \text{in} \quad D, \qquad\text{and}\qquad \phi \equiv 0 \quad \text{on}\quad \de_LD_T,
    \]
     and such that $\phi_\eps\weak \phi$ weakly in $H^1(D_T)$ (up to subsequences).
	Moreover, arguing similarly as for $V(A_n-A)\Delta \phi_{\eps,n}$, we deduce that
	\[
		\left\vert\iint_{D_T} V\eps e^{-|x|}\Delta\phi_\eps\dif x \dif t\right\vert
		\le \eps^{\sfrac12}\big( \| U \|_{L^2(D_T)} + 1\big)C,
	\] 
	and so it vanishes as $\eps\to0$.
	Thus, \eqref{e:comparison-stefan-proof-eps} converges to
    \[
        \int_{D(T)} V \phi_0 \dif x 
        \ge  \iint_{D_T} V  \phi \dif x \dif t 
        \ge  \iint_{D_T} V^- \dif x \dif t,
    \]
    where the last inequality holds since $0\le \phi\le 1$ in $D_T$.

    Finally, the inequality is now independent from $\nabla \phi_0$; thus we can approximate $\phi_0 \to \ind_{\{V(\cdot,T)<0\}}$ in $L^2(D)$ and we deduce that
    \begin{equation}\label{e:appendix-gronwall-bounded}
    \int_{D(T)}(-V^-)\dif x \le    \iint_{D_T} (-V^-)\dif x \dif t,
    \end{equation}
    and we conclude by Gr\"onwall's Lemma, since $-V^-$ is positive, $V(\cdot,0)\ge0$ in $D$, and the integral on the right-hand side is finite. This proves the inequality 
    $V(x,T)\ge 0$ for almost every $x \in D$. Finally, using again that $u_1(x,T) > u_2(x,T)$ implies $\mu_1(x,T) \ge \mu_2(x,T)$, we get 
    \[
    	u_2(x,T) \ge u_1(x,T) \qquad\text{and}\qquad \mu_2(x,T) \ge \mu_1(x,T),
    \] 
    for almost every $x\in D$, which completes the proof.

	\paragraph{General case: $D\subset\R^d$ possibly unbounded}
    The proof in the case of unbounded $D$ relies on the strategy of the bounded case. On the other hand, when $D$ is unbounded, the set $\{V<0\}$ might have infinite Lebesgue measure, which does not allow to approximate $\ind_{\{V(\cdot,T)<0\}}$ with smooth functions in $L^2$, so we cannot write \eqref{e:appendix-gronwall-bounded}. In order to overcome this issue, in the unbounded case we use a different family of approximating problems.
	
    Let us fix $\eps>0$ and $\psi_0\in C^\infty_c(D)$. For all $n\in\NN$ large enough such that $\psi_0\in C^\infty_c(D^n)$, we consider the solution $\psi_{n,\eps}$ of the following reversed heat system.
    \begin{equation}\label{e:appendix-system-unbounded}
    \begin{cases}
        \de_t \psi + (\eps e^{-|x|} + A_n) \Delta \psi = B_\eps\, \psi & \text{in}\quad D^{n}\times(0,T)\\
        \psi(x,T) = \psi_{0}(x) &\text{on} \quad D^n\\
        \psi \equiv0 & \text{on}\quad \de D^n \times (0,T),
    \end{cases}
    \end{equation}
    where $A_n=A_n(x,t)$ is defined as in \eqref{e:appendix-system-bounded}, and 
    \[  
        B_\eps := \left|U^{-}_{\eps}\vee (-1)\right|^2, \qquad \big|  \nabla B_\eps \big|=  2\big|U^-_\eps\vee(-1)\big| \big| \nabla \big(U^-_\eps \vee(-1)\big) \big| \le 2 \big| \nabla U_\eps \big|,
    \]
    where $U_\eps$ is a family of smooth functions such that $U_\eps \to U$ strongly in $H^1(D_T)$, which we choose in such a way that 
    \begin{equation}\label{e:2-3-bound}
    \big\|  \nabla B_\eps \big\|_{L^2(D_T)} \le 2 \big\| \nabla U_\eps \big\|_{L^2(D_T)}\le 3 \big\| \nabla U \big\|_{L^2(D_T)}.
    \end{equation}
    As for \eqref{e:appendix-system-bounded}, testing the equation \eqref{e:appendix-system-unbounded} with $\Delta \psi_{\eps,n}$, we get
    \begin{equation}
    \begin{split}\label{e:comparison:energy-identity-test}
    \frac12\|\nabla\psi_{\eps,n}\|_{L^2(D^n(\tau))}^2 
        + \int_{\tau}^T\int_{D^n} (\eps e^{-|x|} + A_n)|\Delta \psi_{\eps,n}|^2 \dif x \dif t 
    &=
        \frac12\|\nabla \psi_{0}\|_{L^2(D^n)}^2 \\
        &\hspace{1cm}-  \int_{\tau}^T\int_{D^n} B_\eps\,|\nabla \psi_{\eps,n}|^2 \dif x \dif t\\
        &\hspace{1.5cm}- \int_{\tau}^T\int_{D^n} \psi_{\eps,n} \nabla B_\eps\cdot \nabla \psi_{\eps,n} \dif x \dif t.
    \end{split}
    \end{equation}

    Since $0\le \psi_0\le 1$ and $B_\eps\ge0$, the maximum principle for parabolic equations implies that $0\le \psi_{n,\eps}\le 1$, and thus from the previous identity and the bound \eqref{e:2-3-bound}, we recover the following estimate
    \[
       \|\nabla\psi_{\eps,n}\|_{L^2(D^n(\tau))}^2 \le \|\nabla\psi_{0}\|_{L^2(D)}^2 + C \|\nabla U\|_{L^2(D_T)} \cdot \left(\int_\tau^T \|\nabla \psi_{\eps,n}\|_{L^2(D^n(t))}^2 \dif t\right)^{\sfrac12}.
    \]
    Let then
    \[
        \tau_{0,\eps,n} := \inf \left\{ \tau \in[0,T] : \int_{\tau}^T \int_{D}|\nabla \psi_{\eps,n}|^2 \dif x \dif t \le 1\right\},
    \]
    For all $\tau\in[ \tau_{0,\eps,n},T]$, it holds
    \[
        \|\nabla\psi_{\eps,n}\|_{L^2(D^n(\tau))}^2 \le \|\nabla\psi_{0}\|_{L^2(D)}^2 + C\|\nabla U\|_{L^2(D_T)},
    \]
    conversely, for $\tau\le \tau_{0,\eps,n}$,
    \[
       \|\nabla\psi_{\eps,n}\|_{L^2(D^n(\tau))}^2 
       \le \|\nabla\psi_{0}\|_{L^2(D)}^2 + \|\nabla U\|_{L^2(D_T)} \cdot \int_\tau^T \|\nabla \psi_{\eps,n}\|_{L^2(D^n(t))}^2 \dif t.
    \]
    Therefore, Gr\"onwall's Lemma applies, and with the same elementary manipulations as in the bounded case, there exists a constant ${C = C(d,T,U)>0}$ such that \eqref{e:comparison-energy} and \eqref{e:comparison-laplacian} hold for $\psi_{\eps,n}$.
    Therefore, as in the bounded case, $\{\psi_{\eps,n}\}$ is relatively compact in $H^1(D_T)$ and in $H^2(D^{n_0}_T)$, for all fixed $n_0\in \NN$.
    Thus, as $n\to+\infty$ they converge to a certain $\psi_\eps \in H^1(D_T)$ (up to subsequences), and, for all $\eps>0$, $\psi_\eps$ satisfies the same \eqref{e:comparison-energy-2} and \eqref{e:comparison-laplacian-2}.

    Let finally introduce the radial nonnegative family of functions $\{\rho_R\}_{R>0}$, where $\rho_R(x) := \tilde\rho(\sfrac {|x|}R)$ and function $\tilde\rho \in C^\infty(\R)$
    such that
    \[
    	\tilde\rho(\zeta) = 1 \qquad\text{if}\quad \zeta \le \frac12, \qquad
        \tilde\rho(\zeta) = 0 \qquad\text{if}\quad \zeta \ge 1,
    		\qquad\text{and}\qquad
    	0\le\tilde\rho'(\zeta)\le 4 \qquad\text{for all}\quad \zeta\in\R.
    \]
    Then, for all $\eps>0$, $n\in\NN$, and $R>0$, we test \eqref{e:comparison-H1} with $\rho_R \psi_{n,\eps}$ in $D_T$, so that, similarly to the bounded case we get
    \begin{align*}
		\int_{D^n(T)} V\rho_R\,\psi_{\eps,n} \dif x  
		&\ge \iint_{D^n_T}
		\Big(
			\rho_R\,V \,\de_t\psi_{\eps,n} 
				- \rho_{R}\,(\nabla U \cdot\nabla \psi_{\eps,n}) 
				- \psi_{\eps,n}(\nabla U \cdot \nabla \rho_R)
		\Big) \dif x \dif t\\
		&= \iint_{D^n_T}
			\rho_R\,V \Big(B_\eps\psi_{\eps,n} - \big(\eps e^{-|x|} +(A_n - A)\big)\Delta \psi_{\eps,n}\Big)  \dif x \dif t \\
				&\hspace{1.5cm}+ \iint_{D^n_T}\Big( U (\nabla \rho_R\cdot \nabla \psi_{\eps,n})   
				- \psi_{\eps,n}(\nabla U \cdot \nabla \rho_R) \Big)\dif x \dif t\\
				&\hspace{3cm}- \int_0^T \int_{\de D^n} U \rho_R |\nabla \psi_{\eps,n}|\dif \Hc^{d-1}\dif t.
    \end{align*}
        
    In particular, since they hold the same estimate of the bounded case, we again observe that
    \[
    \lim_{n\to+\infty}\iint_{D^n_T} V\, (A_n-A) \, \Delta \psi_{\eps,n} \dif x \dif t =0,
    \]
    and for all $\eps>0$ there exists a negative (capacitary) Radon measure $\sigma_\eps\le0$ supported on $\de_L D_T$ such that, for all $R_0>0$, it holds
    \[
    	|\nabla \psi_{\eps,n}| (\dif \Hc^{d-1}\otimes\dif t) \llcorner (\de_L D^n_T) \xrightharpoonup[n\to+\infty]{}\sigma_\eps \qquad\text{in}\quad \mathcal D^*\quad \text{and in}\quad H^{-1}(B_{R_0} \times(0,T)).
    \]
    As in the bounded case, since $U\rho_R$ is of compact support and $U \rho_R \in H^1_{0,L}(D_T)$, we get that
    \[
        \int_{D_T} U\rho_R \dif \sigma_\eps  = \int_{D_T} U^+\rho_R \dif \sigma_\eps \le 0,
    \]
    and therefore, for $n\to+\infty$, we get the following integral inequality
    \begin{equation}\begin{split}\label{e:comparison:inequality-eps-unbounded}
		\int_{D(T)} V\rho_R\,\psi_{\eps} \dif x  
		&\ge  \iint_{D_T}
			\rho_R\,V \Big(B_\eps\psi_\eps - \eps e^{-|x|} \,\Delta \psi_\eps\Big)  \dif x \dif t \\
				&\hspace{1.5cm}+ \iint_{D_T} U (\nabla \rho_R\cdot \nabla \psi_{\eps})  \\ 
				&\hspace{3cm}- \iint_{D_T}
				\psi_{\eps}(\nabla U \cdot \nabla \rho_R) \dif x \dif t.
    \end{split}\end{equation}  
    Again, since $\{\psi_\eps\}_{\eps>0}$ is uniformly bounded in $H^1(D_T)$, there exists $\psi\in H^1(D_T)$ with $0\le \psi\le 1$ and $\psi\equiv0$ on $\de_L D_T$ such that $\psi_\eps \weak \psi$ weakly in $H^1(D_T)$.
    Through the same argument as in the bounded case, we get that
    \[
        \left\vert\iint_{D_T} V \eps e^{-|x|} \Delta \psi_\eps \dif x \dif t \right\vert \le \eps^{\sfrac12}(\|U\|_{L^2(D_T)} + 1)C,
    \]
    and since $U_\eps \to U$ strongly in $H^1(D_T)$, from \eqref{e:comparison:inequality-eps-unbounded} we recover the following integral inequality
    \[
		\int_{D(T)} V\rho_R\,\psi_0 \dif x  
		\ge  \iint_{D_T}
			\rho_R\,V \big|U^-\vee(-1)\big|^2\psi  \dif x \dif t 
				+ \iint_{D_T} U (\nabla \rho_R\cdot \nabla \psi)  
				- \psi(\nabla U \cdot \nabla \rho_R) \dif x \dif t,
    \] 
    Finally, since
    \[
    	\rho_R \xrightarrow[R\to+\infty]{} 1 
    		\qquad\text{and}\qquad
    	\nabla \rho_R \xrightarrow[R\to+\infty]{} 0 \qquad\text{pointwise in}\quad \R^d,
    \]
    and
    \[
        \big |V \big| \big| U^- \vee(-1) \big|^2 
            \le \big( \big|U^-\big| +2) \big| U^- \vee(-1) \big|^2 
            \le 3 \big|U^-\big|^2
    \]
    Dominate Convergence Theorem applies and we recover the following integral inequality
    \begin{equation}
		\int_{D(T)} V \,\psi_0 \dif x  
		\ge  \iint_{D_T} V \big|U^-\vee(-1)\big|^2\psi  \dif x \dif t 
		\ge  \iint_{D_T} V^- \big|U^-\vee(-1)\big|^2  \dif x \dif t 
    \end{equation}
    We notice that the integrand on the right-hand side is non positive. In order to prove that the right-hand side is in $L^1(D_T)$ we first notice that when $U_-\neq 0$ we have that $u_2<u_1$ and $\mu_2\le \mu_1$, and so 
    \[
    V=V^-=(U+(\mu_2-\mu_1))^-=U^--(\mu_1-\mu_2)^+.
    \]
    Thus, we have 
    \begin{align*}
    	V\big|U^-\vee(-1)\big|^2  
    		&= V^-\big|U^-\vee(-1)\big|^2\\
    		&= \big(U+ (\mu_2-\mu_1)\big)^-\big(U^-\ind_{\{-1< U <0\}} -\ind_{\{U<-1\}}\big)^2\\
            &= \big(U+ (\mu_2-\mu_1)\big)^-\big((U^-)^2\ind_{\{-1< U <0\}} +\ind_{\{U<-1\}}\big)\\
    		&= -\big|U^-\big|^3\ind_{\{-1< U <0\}} 
    			- (\mu_1-\mu_2)\big|U^-\big|^2\ind_{\{-1< U <0\}}
    			-\big|U^-\big| \ind_{\{U<-1\}}
    			-(\mu_1-\mu_2) \ind_{\{U<-1\}}\\
    		&=-\big( W_1 + W_2 + W_3 + W_4\big),
    \end{align*}    
    where $W_1,\dots,W_4$ denote the four terms in the expression above.
    For all $i=1,\dots,4$, $W_i\ge0$ in $D_T$ and $W_i\in L^1(D_T)$ since:
    \begin{itemize}
    \setlength\itemsep{0.2cm}
    	\item for $W_1 = |U^-|^3 \le |U^-|^2$ (in $|U^-|<1$) and $U\in L^2(D_T)$;
    	\item for  $W_2 \le C |U|^2$, and $U\in L^2(D_T)$;
        \item  $\|W_3\|_{L^1(D_T)} \le \|U\|_{L^2(D_T)} |\{U<-1\} \cap D_T|^{\sfrac12}$, that is finite since $U \in L^2(D_T)$;
    	\item $\|W_4\|_{L^1(D_T)}\le 2|\{U<-1\}\cap D_T|<+\infty$, since $U\in L^2(D_T)$. 
    \end{itemize}
    
    Therefore, taking $\psi_0\to \big|U^-\vee(-1)\big|^2$ in $L^1(D)$, we get
    \[
    \int_{D(T)} (W_1 + W_2 + W_3 + W_4) \dif x 
    	\le \int_{D_T} (W_1 + W_2 + W_3 + W_4) \dif x \dif t,
    \]
    that implies, by Gr\"onwall's Lemma that
    \[
        (W_1 + \cdots W_4)(\cdot,T) \equiv 0 \qquad \text{almost everywhere in}\quad D.
    \]
    Specifically, since $W_i\ge0$ for all $i=1,\dots,4$, it implies that $W_i(\cdot,T)\equiv0$ almost everywhere in $D$, and therefore that
    \begin{equation}\label{e:comparison-temperatures-ordered-unbounded}
    	U(x,T) = u_2(x,T)-u_1(x,T) \ge 0 \qquad\text{for almost every}\quad x\in D.
    \end{equation}
    This also implies that $\mu_2\ge \mu_1$ on the set $\{u_2>u_1\}$; still, this does not conclude the proof, since, if $U=0$, we cannot directly conclude that 
    \[
    	\mu_2(x,T)-\mu_1(x,T) \ge0 \qquad\text{for almost every}\quad x\in D.
    \]
    However, for all $R>0$, $u_1$ and $u_2$ are respectively a subsolution and a supersolution of the Stefan problem in $D \cap B_R$.
    Moreover, by \eqref{e:comparison-temperatures-ordered-unbounded}, $u_2 \ge u_1$ almost everywhere in $D_T$, and the same holds for their traces on $\de_L (C_R \cap D_T)$.
    Hence 
    \[
    	u_1 \le u_2 \qquad\text{on}\quad \de_P (D_T\cap C_R) \qquad \text{and}\qquad \mu_1(x,0) \le \mu_2(x,0) \qquad\text{in}\quad D \cap B_R,
    \]
    and so the bounded case applies and we can conclude that 
    \[
    	\mu_1(x,T) \le \mu(x,T) \qquad \text{for almost every}\quad x\in D,
    \]	
    and so we conclude the proof also in the case of unbounded sets $D\subset\R^d$. This concludes the proof of \Cref{t:comparison-dirichlet}.
\end{proof}

\begin{proof}[\bf Proof of \Cref{t:comparison-neumann}]
The proof of \Cref{t:comparison-neumann} follows by the same steps of \Cref{t:comparison-dirichlet}. We define the family of approximating domains as $D^n:=D\cap B_n$, where $B_n$ is the ball of radius $n$ in $\R^d$, and the family of test functions $\{\psi_{\eps,n}\}$, defined as the solutions of following mixed  problem
    \[
        \begin{cases}
            \de_t \psi + (\eps e^{-|x|} + A_n) \Delta\psi = B_\eps \psi &\text{in}\quad D\times(0,T)\\
            \psi(x,T) = \psi_0 & \text{on} \quad D \\
            \de_\nu \psi = 0& \text{on}\quad (\de D \cap B_n) \times (0,T),\\
            \psi = 0 & \text{on} \quad (D\cap \partial B_n) \times(0,T),
        \end{cases}
    \]
    where $A_n$ and $B_\eps$ are defined as above. Since $\de D$ is sufficiently regular, these are $H^2(D \cap B_R)$ and the conclusion follows by the same argument as in the proof of \Cref{t:comparison-dirichlet}.
\end{proof}
    
\begin{remark}    
    In the Neumann case, we require $D$ to have smooth boundary since, in the Neumann framework we cannot approximate the set $D$ with a family of regular ones by maintaining the boundary conditions and so, since we need $H^2$-estimates through the proof, we cannot drop the regularity in the limit in this case.
\end{remark}


\subsection{Further consequences of the uniqueness of the weak solutions}\label{sec:further-consequences}

Once we have established the existence (\Cref{t:main}) and the uniqueness (\Cref{cor:every-weak-solution-is-the-limit-of-u-eps}) of the enthalpy solution and the convergence of the regularized sequences $u_\eps$ are assured, we obtain two immediate consequences, that we summarize in the next corollaries. 

\begin{corollary}\label{cor:uniqueness-h}
	Let $D\subset\R^d$ be an open set, $u_M : D \to \R$ be a measurable function, $F\in L^2(D;\R^d)$, $g\in H^1(D)$, and $h\in L^\infty(D)$ with $|h(x)|\le 1$.
    Let $\{u_\eps\}_{\eps>0}$ be a family of minimizers for \eqref{e:efunctional} in $\Uc_{\mathcal D}(D,g)$ or in $\Uc_{\mathcal N}(D,g)$.
	Then the limit $(u,\mu)$ does not depend on the value of $h$ in $\Omega_g^+ \cup \Omega_g^-$.
\end{corollary}
\begin{proof}
The proof follows immediately from \Cref{t:main}, claim \ref{item:main_thm:initial-mushy-coefficient}.	
\end{proof}

\begin{corollary}[Boundedness of the solutions in the homogeneous case $F\equiv 0$]\label{cor:boundeness-solutions}
    Let $D\subset\R^d$ be an open set, $u_M : D\to \R$ measurable, $g\in H^1(D)$, and $h\in L^\infty(D)$, with $|h(x)|\le1$. Let $(u,\mu)$ be an enthalpy solution of the Stefan problem with Dirichlet or Neumann boundary conditions, melting temperature $u_M$ and heat source $F \equiv0$, initial temperature $g$, and initial mushy coefficient $h$.
    If $g\in L^\infty(D)$, then $u \in L^\infty(D_\infty)$ and 
    $
        \| u\|_{L^\infty(D_\infty)} \le  \|g \|_{L^\infty(D)}.
    $
\end{corollary}

\begin{proof}
We prove the statement in the Dirichlet case $u\in\mathcal U_{\mathcal D}(D,g)$, the Neumann one being analogous.
Fix $\eps>0$. We prove that if $u_\eps$ is a minimizer of \eqref{e:efunctional} in $\Uc_{\mathcal D}(D,g)$, then  $\| u_\eps\|_{L^\infty(D_\infty)} \le \| g\|_{L^\infty(D)}$.
Assume by contradiction that there exists $K>0$ such that
    \[
        |\{|u_\eps|>K\}|>0 \qquad \text{and}\qquad \|g\|_{L^\infty(D)} < K.
    \]
    By construction, 
    \[
    \tilde{u}_\eps(x,t) = (u_\eps \wedge K) \vee (-K) \in \Uc_{\mathcal D}(D,g),
    \]
    but since $(|\de_t u_\eps|+ |\nabla u_\eps|) \not\equiv0$ in $\{|u_\eps|>K\}$, it follows that
    \[
    \Fc_\eps(\tilde{u}_\eps) = \iint_{\{u|\le K\}}\frac{e^{\sfrac{-t}{\eps^4}}}{\eps^4}\Bigg\{\eps^4\Bigg[1+\frac{1}{\eps}f_\eps(u_\eps)^2\Bigg]|\de_t u_\eps|^2 + |\nabla u_\eps|^2 \Bigg\}  <   \Fc_\eps({u}_\eps).
    \]
    This contradicts the minimality of $u_\eps$ in $\Uc(D,g)$, which concludes the proof.
\end{proof}

\appendix 
\crefalias{section}{appendix}
\crefalias{subsection}{appendix}

\section{Classical and weak solutions of the Stefan problem}\label{sec:history}

This section is dedicated to the classical counterpart of \Cref{def:enthalpy_solution}. 
We also briefly discuss the history of the Stefan problem in its classical and weak formulations.
For more details about the Stefan problem, its history and physical motivation we refer to \cite{Ladyzhenskaya_Solonnikov_Uraltseva68:Parabolic_equation_BOOK,rubinstein71:Stefan,Meirmanov_1992,Visintin:BookPhaseTransition,Andreucci2004:StefanNotes,CaffarelliSalsa:GeomApproachToFreeBoundary}.

\subsection{The Stefan problem - classical formulations}\label{sub:history-classical-formulation}
The Stefan problem is a free boundary problem that models the liquid-solid phase transition. 
The problem was originally introduced by Stefan \cite{Stefan1889:AcidBaseDiffusion,Stefan1889:HeatConduction,Stefan1890:Diffusion,Stefan1891:IceFormation} and by Lamé and Clapeyron \cite{LameClapeyron1831}. In its original formulation, the space is divided into two regions separated by a smooth interface; in the region occupied by the material in a solid state the temperature is strictly negative, while in the region where the material is in liquid state the temperature is strictly positive.
In both regions, the temperature satisfies the heat equation, while the interface evolves with normal velocity proportional to the heat flux.\medskip 

The space region is a fixed open set $D\subset \R^d$ and the temperature is represented by a continuous function $u:D\times[0,+\infty)\to\R$. As in the rest of the paper, we set $\Omega_u^\pm$ to be space-time domains $\{\pm u>0\}\subset D\times [0,+\infty)$. In the classical formulation(s) (see \eqref{e:classic-stefan-tp1} and \eqref{e:classical-Stefan-op} and the discussion below) the space-time free boundaries $\Gamma^\pm:=\partial\Omega_u^\pm\cap \big(D\times(0,+\infty)\big)$ are assumed to be $C^1$ smooth $d$-dimensional surfaces that cross all the horizontal planes $D\times\{t\}$ transversally. For all $t\ge 0$ we set 
$$\Gamma^\pm(t):=\Gamma^\pm\cap(D\times\{t\})\quad\text{and}\quad\Sigma^\pm(t):=\{x\in D\,:\,(x,t)\in\Gamma^\pm(t)\}.$$
The space-time normal vector to $\Gamma^\pm$ (pointing outwards $\Omega_u^+$ and inwards $\Omega_u^-$) at $(x,t)\in \Gamma^\pm$ is given by 
$$\widetilde\nu^\pm=\widetilde\nu^\pm(x,t)=\big(\widetilde\nu_x^\pm(x,t),\widetilde\nu_t^\pm(x,t)\big)\in\R^d\times\R,$$
while by $\nu^\pm\in\R^d$ we denote the normal vector (in space) to $\Sigma^\pm(t)$ at $x\in \Sigma^\pm(t)$, that is, 
$$\nu^\pm=\nu^\pm(x)=\frac{\widetilde\nu_x^\pm(x,t)}{|\widetilde\nu_x^\pm(x,t)|}.$$
The normal velocity of $\Sigma^\pm(t)$ at $x$ is then given by
\begin{equation}\label{e:normal-velocity}
    V_\nu^\pm(x,t):=-\frac{\widetilde\nu_t^\pm(x,t)}{|\widetilde\nu_x^\pm(x,t)|}, 
\end{equation}
where the last equality follows since $\Gamma^\pm$ are the level sets of $u_\pm$. Furthermore, the temperature $u : D\times[0,+\infty) \longrightarrow \R$ is locally Lipschitz in $D\times[0,+\infty)$ and they hold
    \begin{equation}\label{e:regularity-assumptions-u-classical-solutions}
        u^+ \in C^{1,1}_{x,t}\big(\overline{\Omega_u^+}\big) \cap C^2_x \big( \Omega_u^+\big) 
            \qquad\text{and}\qquad
        u^- \in C^{1,1}_{x,t}\big(\overline{\Omega_u^-}\big) \cap C^2_x \big( \Omega_u^-\big).
    \end{equation}
\paragraph{Classical two-phase solutions} In the classical formulation of Stefan, that the two free boundaries are fully collapsed ($\Gamma^+(t)=\Gamma^-(t)$ for all $t\ge0$), while the evolution of the temperature 
 ${u:D\times[0,+\infty)\to\R}$ 
 is governed by the system 
\begin{equation}\label{e:classic-stefan-tp1}\tag{ST-tp}
\begin{cases}
    \begin{array}{ll}
        \de_t u_+ = \Delta u_+ & \text{in}\quad \Omega^+_u,\\
        \de_t u_- = \Delta u_- & \text{in}\quad \Omega^-_u,\\
        \hspace{0,37cm}V_\nu^\pm =\frac{1}{2L} \big( |\nabla u_+|- |\nabla u_-| \big) &\text{on}\quad \Gamma^\pm(t)\quad\text{for all}\quad t>0,
    \end{array}
\end{cases}
\end{equation}
where the constant $L>0$ represents the {latent heat} of the material.

\paragraph{Classical one-phase solutions.}
In the limit case, in which the ice is always at thermal equilibrium  $u_-\equiv 0$, we obtain the \emph{one-phase Stefan problem}
in which the heat diffusion takes place only in the water region $\Omega_u^+$.
Thus, under the smoothness assumptions above, \eqref{e:classic-stefan-tp1} reads as
\begin{equation}\label{e:classical-Stefan-op}\tag{ST-op}
    \begin{cases}
        \de_t u_+ = \Delta u_+ & \text{in}\quad \Omega_u^+,\\
        \hspace{0,37cm}V_\nu^+ = \frac{1}{2L}|\nabla u_+| & \text{on}\quad \Gamma^+(t)\quad\text{for all}\quad t>0.
    \end{cases}
\end{equation}
In this case, the positivity set is expanding:
\[
    \Omega_{u}^+(t_1) \subset \Omega_u^+(t_2) \quad\text{for all times}\quad 0 \le t_1 \le t_2.
\]

\begin{remark}
    Let $u$ be a solution of \eqref{e:classic-stefan-tp1} or \eqref{e:classical-Stefan-op} and let  $v = \frac1L u$. Then $v$ is a solution of the Stefan problem with $L=1$, so we can set $L=1$. 
\end{remark}

    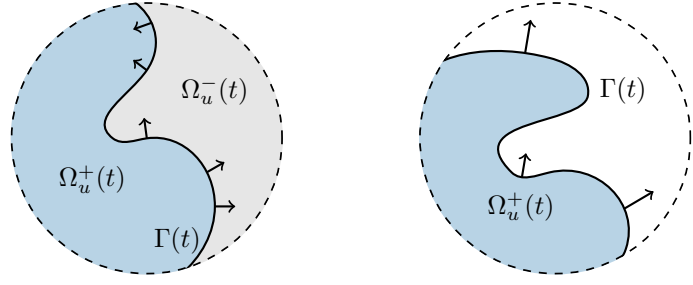
\begin{figure}[!ht]
    	\centering
	\begin{tikzpicture}[scale=0.9, use Hobby shortcut]
		\begin{scope}[shift={(-3,0)}]
			\clip (0,0) circle (2);
			\draw[fill = black!10] (0,0) circle (2.5);
			\draw[thick, fill = MidnightBlue!20] (-1,2) .. (0,1) .. (-0.5,0) .. (0,0) .. (1,-1) .. (0.5,-2) .. (-2,-2) .. (-2,2);
			\draw[dashed, thick] (0,0) circle (1.99);
			
			\draw[->,thick] (0,0) -- (-0.04,0.3);
			\draw[->,thick] (1,-1) -- (1.3,-1);
			\draw[->,thick] (0,1) -- (-0.2,1.15);
			\draw[->,thick] (0.08,1.7) -- (-0.2,1.6);
			\draw[->,thick] (0.88,-0.5) -- (1.14,-0.35);
			
			\node at (-0.8,-0.6) {$\Omega_u^+(t)$};
			\node at (1,0.7) {$\Omega_u^-(t)$};
            \node[shift={(0,0)}] at (0.45,-1.5) {$\Gamma(t)$};
		\end{scope}
		
        \begin{scope}[shift={(3,0)}]
			\clip (0,0) circle (2);
            \begin{scope}[shift={(-0.5,-0.5)},scale=1.4]
            \draw[->,thick] (0,-0.1) -- (0.05,0.2);
			\draw[->,thick] (1,-1) -- (1.2,-1.2);
			\draw[->,thick] (0,1) -- (0.1,1.6);
			\draw[->,thick] (-1,-1) -- (-1.3,-1.4);
			\draw[->,thick] (0.88,-0.5) -- (1.39,-0.2);
            
			\draw[thick, fill = MidnightBlue!20,closed,shift={(0.3,0)},scale=0.8] (0,1.5) .. (0.5,1) .. (-0.5,0) .. (0,0) .. (1,-1) .. (0.5,-1.5) .. (-2,1);
            \end{scope}
			\draw[dashed, thick] (0,0) circle (1.99);
						\node at (1,0.7) {$\Gamma(t)$};
			\node at (-0.5,-1) {$\Omega_u^+(t)$};
		\end{scope}
		
	\end{tikzpicture}
         \caption{
         A two-phase (on the left) and one-phase (on the right) classical solutions.}
    \end{figure}
%
%
%
%
\subsection{Short history of the theory of weak solutions}\label{sub:weak-solutions-history} 
The theory of weak solutions has been originally developed by Kamin \cite{Kamenomostskaya} (in dimension $d\le3$) and Oleinik \cite{Oleinik60} (for $d\ge 3$), in the framework of distributional solutions. 
The corresponding theory for $H^1$ solutions was later developed by Ladyzhenskaya, Solonnikov, and Uraltseva \cite{Ladyzhenskaya_Solonnikov_Uraltseva68:Parabolic_equation_BOOK}, Friedman \cite{Friedman_1968}, Cannon and DiBenedetto \cite{CannonDiBenedetto1980:ExistenceStefan}, and G\"otz and Zaltzman \cite{GotzZaltzman:NonincreaseMushyInhomogeneousPb}. 
Specifically, in \cite{Ladyzhenskaya_Solonnikov_Uraltseva68:Parabolic_equation_BOOK,Friedman_1968,CannonDiBenedetto1980:ExistenceStefan} the authors recast the Stefan problem as a differential inclusion, namely
\begin{equation}\label{e:stefan-energetic-appendix-inclusion}
	\de_t H(u) \ni \Delta u \qquad\text{in}\quad D\times(0,+\infty),
\end{equation}
where the enthalpy $H(u):D\times [0,+\infty)\to\R$ is a multivalued function of $u$, such that,
\begin{equation}\label{e:definition-of-the-enthalpy-H}
    H(u):= 
    \begin{cases}
        u+1 &\text{if}\quad u>0\\
        [-1,1] &\text{if}\quad u=0\\
        u-1 &\text{if}\quad u<0.
    \end{cases}
\end{equation}
In this framework, $u$ is an enthalpy solution in $D\times(0,+\infty)$ if there exists $v\in H(u)$ such that for almost every $0\le t_1 < t_2 <+\infty$ 
\begin{equation*}
	\int_{D(t)}v\,\eta \dif x \bigg|_{t=t_1}^{t_2} 
	= \int_{t_1}^{t_2}\int_D v\,\de_t \eta 
	-\nabla u \cdot \nabla \eta \dif x \dif t,
\end{equation*}
where $\eta$ is a proper test function (Neumann or Dirichlet). 
We remark that this formulation is slightly different point of view of the one introduced in \Cref{def:enthalpy_solution,def:enthalpy_solution-F}, where the solutions are couples $(u,\mu)$ such that $v = u +\mu$ (and so if $(u,\mu)$ is a solution in the sense of \Cref{def:enthalpy_solution} (or \Cref{def:enthalpy_solution-F}) then $u$ is a solution in the sense of \eqref{e:stefan-energetic-appendix-inclusion}).

    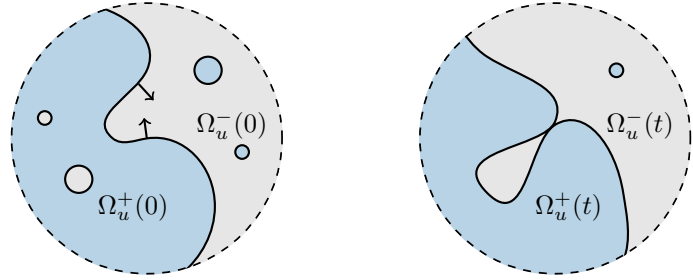
\begin{figure}[ht]
   \centering 	
    \begin{tikzpicture}[scale=0.9,use Hobby shortcut]
		\begin{scope}[shift={(-3,0)}]
			\clip (0,0) circle (2);
			\draw[fill = black!10] (0,0) circle (2.5);
			\draw[thick, fill = MidnightBlue!20,closed] (-1,2) .. (0,1) .. (-0.5,0) .. (0,0) .. (1,-1) .. (0.5,-2) .. (-2,-2) .. (-2,2);
			\draw[dashed, thick] (0,0) circle (1.99);

            \draw[thick, fill = black!10] (-1.5,0.3) circle (0.1);
            \draw[thick, fill = black!10] (-1,-0.6) circle (0.2);
            \draw[thick, fill = MidnightBlue!20] (0.9,1) circle (0.2);
            \draw[thick, fill = MidnightBlue!20] (1.4,-0.2) circle (0.1);

            \draw[->,thick] (0,0) -- (-0.04,0.3);
			\draw[->,thick] (-0.12,0.8) -- (0.1,0.56);

			\node at (-0.2,-1) {$\Omega_u^+(0)$};
			\node at (1.25,0.2) {$\Omega_u^-(0)$};

		\end{scope}
    
		\begin{scope}[shift={(3,0)}]
			\clip (0,0) circle (2);
			\draw[fill = black!10] (0,0) circle (2.5);
			
			\draw[thick, fill = MidnightBlue!20,closed,shift={(0,0.2)}] (-1,1) .. (0,0.05) .. (-1,-1) .. (0,0) .. (1,-1) .. (1,-2) .. (-2,-2) .. (-2,2);

            \draw[thick, fill = MidnightBlue!20] (0.9,1) circle (0.1);

			\draw[dashed, thick] (0,0) circle (1.99);

			\node at (0.2,-1) {$\Omega_u^+(t)$};
			\node at (1.25,0.2) {$\Omega_u^-(t)$};

		\end{scope}
	\end{tikzpicture}
         \caption{An example of a configuration that develops topological changes in finite time.}
        \label{fig:shrinking}
    \end{figure}

The enthalpy formulation requires no a priori regularity assumptions on the free interface $\Gamma$, so it allows to treat topology changes such as vanishing or fragmenting regions (see \Cref{fig:shrinking}). The loss of control over the geometry of the interface is compensated by the stable numerical properties of the enthalpy formulation. This has been originally pointed out by Atthey in \cite{Atthey74}, where he proved the numerical convergence of a finite differences scheme for 1D enthalpy solutions and  gave an example of solution with non-empty zero set, which he named \emph{mushy region}.

\subsection{Remarks about the physical interpretation of the enthalpy and the mushy region}\label{sub:mushy-physical}  From a physical point of view, the enthalpy of the material is defined as
\[
H = U + pV,
\]
where $U$ denotes the internal energy, $p$ is the pressure, and $V$ is the volume of the system.
In the Stefan problem, both pressure and volume remain constant throughout the evolution.
Therefore the enthalpy completely describes the energy state of the system.
Namely, we can write 
$$H(x,t) = u(x,t) + \mu(x,t),$$
where $u$ is the temperature of the system, while $\mu$ is the mushy coefficient and encodes the information about the state of the material, that is, 
\[
\mu(x,t) :=\lim_{r\to0} \frac{|B_r(x)\cap\{\text{liquid at time }t\}| - |B_r(x)\cap\{\text{solid at time }t\}|}{|B_r|}.
\]
Therefore, the matter state is uniquely determined by the value of the enthalpy, identifying the liquid state with $\{\mu=1\}$, the solid one with $\{\mu = -1\}$, and one can interpret the \emph{mushy region} as the points where the energy is too high to allow the material to stay in the solid state, but is not sufficient to completely melt it. Specifically, for all $t\ge0$, one can define $\mathcal M(t)\subset \R^d$, the \emph{mushy region at time $t$}, as 
\begin{equation}\label{e:mushy-intro0}
	\mathcal M(t) := \left\{x \in D: |\mu(x,t)| <1\right\}.
\end{equation}

	\begin{figure}[htbp]
	\centering
	\begin{tikzpicture}[use Hobby shortcut]
		\begin{scope}
			\clip (0,0) circle (2);
			\draw[thick, fill = MidnightBlue!20] (-1,2) .. (0,1) .. (-0.5,0) .. (0,-0.1) .. (0,-0.7) .. (-1.8,-2) .. (-3,-2) .. (-2,2);
			\draw[thick, fill = black!10] (-1,2) .. (0,1) .. (-0.5,0) .. (0,0) .. (1,-1) .. (2.5,-2) .. (3,-2) .. (2,2);
			
			\draw[dashed, thick] (0,0) circle (1.99);
			\draw[thick,->] (0,1) -- (0.3,0.8);
			\draw[thick,->] (0,-0.7) -- (0.3,-0.9);
			\draw[thick,->] (1,-1) -- (0.7,-1.2);
		\end{scope}
		\node at (-1,-0.7) {$\Omega_u^+(t)$};
		\node at (1,0.5) {$\Omega_u^-(t)$};
		\node at (0.3,-1.6) {$\{ u=0\}$};
		
		\node at (0,2.3) {$\Gamma_{tp}(t)$};
		\node[rotate=-135, shift={(0,2.5)}, rotate=135] at (0,0) {$\Gamma^-_{op}(t)$};
		\node[rotate=145, shift={(0,2.5)}, rotate=-145] at (0,0) {$\Gamma^+_{op}(t)$};
		
		
	\end{tikzpicture}
	\caption{An example of a solution of the Stefan problem in $\R^2$ with non-empty zero set. }
	\label{fig:non-empty-mushy}
\end{figure}
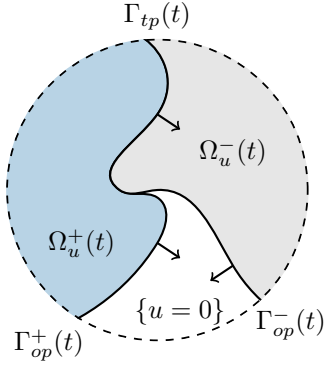

Solutions with non-empty mushy region have been widely investigated starting from the 80s, as Atthey in \cite{Atthey74} showed the first example of a numerically stable (enthalpy) solution with non empty mushy region.
In particular it has been proved that in the presence of certain  heat sources \cite{Meirmanov81:ExampleMushy,Primicerio83:Mushy} or with space-dependent melting temperature \cite{FasanoPrimicerio85:Mushy-variablemelting}, the problem naturally develops a mushy region, while if the problem is homogeneous $\mathcal M(t)$ is non increasing in time \cite{RogersBerger1984:MonotonicityMushyRegionTwoPhaseStefan,GotzZaltzman:NonincreaseMushyInhomogeneousPb}. In the latter cases, a wide research field investigated the dynamic of such region (for a more complete discussion on this last topic, see \cite{Paiano_Velichkov2026:monotonicity} and the references therein).

On the other hand regular solutions of the Stefan problem with possible non empty mushy region $\mathcal M(t)$ have been widely studied in dimension one; for a detailed account on the results in the 1D case we refer to the monographs \cite{rubinstein71:Stefan,Meirmanov_1992,Visintin:BookPhaseTransition}, and the references therein. 
We notice that in 1D the positivity set $\{u>0\}$ and the negativity set $\{u<0\}$ are unions of disjoint intervals; in particular, when these intervals are finitely many, the two-phase Stefan problem is locally equivalent to either the one-phase problem or to the two-phase problem with empty mushy region (see \Cref{fig:1D-profiles}). 
This reduction is no longer possible in dimension $d\ge 2$ (see \Cref{fig:non-empty-mushy}), where the description of the free boundaries remains an open problem.
So, in the next Section we introduce a {\it classical} formulation for the Stefan problem with non-empty mushy region in higher dimension.

\begin{center}
    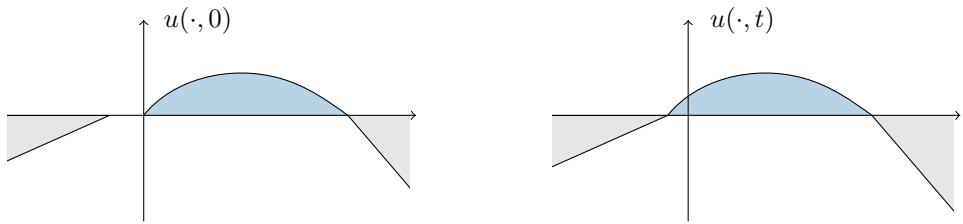
\begin{figure}[ht]
	\begin{tikzpicture}[use Hobby shortcut,yscale=0.7,xscale=0.9]
		\begin{scope}[shift={(-4,0)}]
			\clip (-2,-2) rectangle (4,2.5);

            \begin{scope}
                \clip (-3,0) rectangle (3.9,2);
                    \draw[fill = MidnightBlue!20] (0,-4) -- (0,0) .. (2,0.7) .. (3,0) .. (5,-2) -- (5,-4);
            \end{scope}

            \begin{scope}
                 \clip(-2,-2) rectangle (3.9,0);
                     \draw[fill = black!10]  
                         (3,0) -- (5,-3) -- (5,0);

                    \draw[fill = black!10]
                        (-4,-2) -- (-0.5,0) -- (-4,0);
            \end{scope}

            \draw[->] (-3,0) -- (4,0);
            \draw[->] (0,-2) -- (0,1.8);
            \node at (0.8,1.8) {$u(\cdot,0)$};			
		\end{scope}

		\begin{scope}[shift={(4,0)}]
			\clip (-2,-2) rectangle (4,2.5);

            \begin{scope}
                \clip (-3,0) rectangle (3.9,4);
                    \draw[fill = MidnightBlue!20,shift={(-0.3,0)}] 
                        (0,-4) -- (0,0) .. (2,0.7) .. (3,0) .. (5,-2) -- (5,-4);
            \end{scope}

            \begin{scope}
                 \clip(-2,-4) rectangle (3.9,0);
                     \draw[fill = black!10,shift={(-0.3,0)}]  
                         (3,0) -- (5,-3) -- (5,0);

                    \draw[fill = black!10,shift={(0.2,0)}]
                        (-4,-2) -- (-0.5,0) -- (-4,0);
            \end{scope}

            \draw[->] (-3,0) -- (4,0);
            \draw[->] (0,-2) -- (0,1.8);
            \node at (0.8,1.8) {$u(\cdot,t)$};
		\end{scope}
	\end{tikzpicture}
    \caption{
        The example shows the temperature graph for a one-dimensional two-phase solution with a mushy region. 
        When the phases meet, the mushy region disappears, and so the problem is always locally equivalent to a one-phase problem or a collapsed two-phase one, while no intermediate configurations appear.
        }
        \label{fig:1D-profiles}
    \end{figure}
\end{center}

\subsection{The two-phase Stefan problem revisited}\label{sub:new-classical}
As we already observed above, 
there might be weak solutions that do not enter in the framework of the classical formulations  discussed in \Cref{sub:history-classical-formulation} due, for instance,  to a loss of regularity of the free interface or the presence of a mushy region. In this section, we discuss the notion of classical solution naturally associated to the weak formulations from \Cref{def:enthalpy_solution} and \Cref{def:enthalpy_solution-F}.

\begin{proposition}[Classical solutions are weak solutions]\label{prop:classical-enthalpy}
If $D\subset\R^d$ be a bounded open set and $h:D\to(-1,1)$ be a continuous function. Let $u:D\times[0,+\infty)\to\R$ be a Lipschitz continuous function in space-time and satisfying \eqref{e:regularity-assumptions-u-classical-solutions}.  
    Suppose that the surfaces $\Gamma^\pm$ are Lipschitz in space-time. Furthermore, we suppose that there are relatively open subsets of $\Gamma^\pm$ 
    \begin{equation}\label{e:appendix-definition-of-Gamma-tp-op}
    \Gamma_{tp}\subset\Gamma^+\cap\Gamma^-\ ,\quad \Gamma^+_{op}\subset\Gamma^+\setminus\Gamma^-\quad\text{and}\quad \Gamma^-_{op}\subset\Gamma^-\setminus\Gamma^+,
    \end{equation}
    which are $C^1$ manifolds transversal to each $D\times\{t\}$
    and are such that 
    \begin{equation}\label{e:appendix-Hausdorff-measure-of-Gamma}
    \mathcal H^{d}\Big(\Gamma^\pm\setminus(\Gamma_{tp}\cup\Gamma^\pm_{op})\Big)=0.
    \end{equation}
  Suppose that $u$ satisfies \begin{equation}\label{e:classical-solution-full}\tag{ST}
	\begin{cases}
		\de_t u^+ = \Delta u^+ & \text{in}\quad \Omega^+_u,\\
		\de_t u^- = \Delta u^- & \text{in}\quad \Omega^-_u,\\
		V_\nu = \frac{1}{1-h(x)}|\nabla u^+| &\text{on}\quad \Gamma_{op}^+,\\
		V_\nu = \frac{1}{h(x)-1} |\nabla u^-| &\text{on}\quad \Gamma_{op}^-,\\
		V_\nu =\frac{1}{2}\left( |\nabla u^+|- |\nabla u^-| \right)&\text{on}\quad \Gamma_{tp}.
	\end{cases}
\end{equation}
Let $\mu : D \times[0,+\infty) \to \R$ be defined as
\[
    \mu(x,t) = h(x) \qquad\text{in} \quad \big\{u(x,t) = 0\big\}.
\]
Then, $(u,\mu)$ is an enthalpy solution in the sense of \Cref{def:enthalpy_solution}, with initial mushy coefficient $h(x)$.
\end{proposition}

\begin{remark}
The formulation \eqref{e:classical-solution-full} covers both the classical {\it collapsed} two-phase problem \eqref{e:classic-stefan-tp1}, in which $\Interior\{u=0\}=\emptyset$ and $\Gamma= \Gamma_{tp}$, and one-phase problem \eqref{e:classical-Stefan-op}, in which $u \ge0$ and $\Gamma = \Gamma^+_{op}$.
\end{remark}

\begin{remark}
In the setting of \Cref{sub:weak-solutions-history} and \Cref{sub:weak-solutions-definitions}, the two phases $\Omega^\pm_u$ might collapse with positive speed, and these phenomena appear also in the 1D problem. Indeed, let us consider the following example: let $u : \R \times \R \to \R$ be defined as
\[
    u(x,t) := 
    \begin{cases}
        2\left(e^{x - 2t}-1\right)^+ + \left(e^{x+t} -1\right)^- &\text{for all}\quad t\le 0,\\
        2\left(e^{t-x}-1\right)^+ + \left(e^{t-x}-1 \right)^- &\text{for all}\quad t > 0.
    \end{cases}
\]
Then $u$ is a two-phase solution with empty mushy region (for $t<0$) and the free boundaries $\Gamma^+$ and $\Gamma^-$ are only Lipschitz continuous in time.
On the other hand, it is natural to expect that for {\it flat} smooth solutions $\Gamma^\pm(t)$ are $C^1$ in space, for all $t\ge 0$. 
\end{remark}

\begin{proof}[Proof of \Cref{prop:classical-enthalpy}]
 We will show that 
    \begin{equation}\label{e:weak_solution-interior}
        \iint (u+\mu) \de_t \eta 
            -\nabla u \cdot \nabla \eta \dif x \dif t =0 \quad\text{for all}\quad \eta\in C^\infty_c(D\times(0,+\infty)).
    \end{equation}
    The validity of \eqref{e:weak_solution} in the general case $0\le t_1<t_2$ and $\eta \in C^\infty_c(D \times\R)$ follows by testing the above equation with $\ind_{t_1}^\eps(1-\ind_{t_2}^\eps)\, \eta(x,t)$, where $\ind_{t_i}^\eps$ is defined as in \eqref{e:approx_caratteristica_begin}, and taking the limit as $\eps\to0$.


%
    Let us  divide the integration domains into $\Omega_u^+$, $\Omega_u^-$, and $\{u=0\}$, and treat each domain separately.

    In $\Omega_u^+$ we have that $\mu\equiv1$, so applying the divergence theorem to the vector field $(0,\eta)\in\R^d\times\R$ we get the following identity
    \[ 
    \iint_{\Omega_u^+} \mu \de_t\eta \dif x\dif t = \iint_{\Omega_u^+} \de_t\eta \dif x\dif t = \iint_{\Gamma^+} \eta\, \widetilde \nu_t^+ \dif \mathcal H^{d}.
    \]
    Still in $\Omega_u^+$, by the divergence theorem applied to the field $(\eta\,\nabla u,0)\in \R^d\times\R$, the gradient term becomes
        \begin{align*}
        \iint_{\Omega_u^+} \nabla u \cdot \nabla \eta \dif x \dif t 
           & = - \iint_{\Omega_u^+} \Delta u \,  \eta \dif x \dif t 
                + \iint_{\Gamma^+} \nabla u \cdot \widetilde\nu_x^+ \, \eta \dif\mathcal H^{d}\\
                & = - \iint_{\Omega_u^+} \partial_tu \,  \eta \dif x \dif t 
                + \iint_{\Gamma^+} \nabla u \cdot \widetilde\nu_x^+ \, \eta \dif\mathcal H^{d}\\
                 & = \iint_{\Omega_u^+}  u \,  \partial_t\eta \dif x \dif t 
                + \iint_{\Gamma^+} \nabla u \cdot \widetilde\nu_x^+ \, \eta \dif\mathcal H^{d}
        \end{align*}
where we used that $u$ is caloric in $\Omega^+_u$ and vanishes on $\Gamma^+$, while $\eta$ vanishes for $t=0$.

Analogously, in $\Omega_u^-$ we have 
    \begin{align*}
    \iint_{\Omega_u^-} \mu \,\de_t\eta \dif x\dif t  &= \iint_{\Gamma^-} \eta\, \widetilde\nu_t^- \dif x\dif t,
    \end{align*}
    and
    \begin{align*}
        \iint_{\Omega_u^-} \nabla u \cdot \nabla \eta \dif x \dif t 
          &  = \iint_{\Omega_u^-}  u \,  \partial_t\eta \dif x \dif t 
                - \iint_{\Gamma^-} \nabla u \cdot \widetilde\nu_x^- \, \eta \dif \mathcal H^d ,
    \end{align*}
    as $\widetilde\nu^-$ is the inward normal vector to $\Omega^-_u$.
    By construction $\de \{u=0\}\cap(D\times(0,+\infty)) = \Gamma^+\Delta \Gamma^-$ and therefore, the divergence theorem applied to the $(d+1)$-dimensional vector field $(0,h\,\eta)$ gives
    \begin{align*}
        \iint_{\{u=0\}} \mu \,\de_t\eta \dif x \dif t 
            = - \iint_{\Gamma^+\setminus\Gamma_-} h \, \eta\, \widetilde\nu_t^+ \dif x \dif t 
            + \iint_{\Gamma^-\setminus\Gamma_+} h \, \eta \, \widetilde\nu_t^- \dif x.
    \end{align*}
Putting together these identities, we get 
    \begin{align*}
\iint (u+\mu) \de_t \eta 
           & -\nabla u \cdot \nabla \eta \dif x \dif t\\
           &=\iint_{\Gamma^+} \eta\, \widetilde \nu_t^+ \dif \mathcal H^{d}+\iint_{\Gamma^-} \eta\, \widetilde \nu_t^- \dif \mathcal H^{d}\\
           &\qquad- \iint_{\Gamma^+_{op}} h \, \eta\, \widetilde\nu_t^+ \dif x \dif t 
            + \iint_{\Gamma^-_{op}} h \, \eta \, \widetilde\nu_t^- \dif x\\
            &\qquad\qquad-\iint_{\Gamma^+} \nabla u_+ \cdot \widetilde\nu_x^+ \, \eta \dif\mathcal H^{d}+\iint_{\Gamma^-} \nabla u_- \cdot \widetilde\nu_x^- \, \eta \dif\mathcal H^{d}.
    \end{align*}
    By using \eqref{e:appendix-Hausdorff-measure-of-Gamma} we can decompose these integrals on $\Gamma_{tp}$ and $\Gamma_{op}^\pm$ as follows:
   \begin{equation}\label{e:integral-classical}
   \begin{split}
\iint (u+\mu) \de_t \eta 
           & -\nabla u \cdot \nabla \eta \dif x \dif t\\
           &=\iint_{\Gamma_{tp}} 2\eta\, \widetilde \nu_t^+  - (\nabla u_+ - \nabla u^-) \cdot \widetilde \nu_x^+ \dif \mathcal H^{d}\\
           &\hspace{1cm}+\iint_{\Gamma^+_{op}} (1-h) \, \eta \, \widetilde\nu_t^+ - \nabla u^+\cdot \widetilde\nu_x^+  \dif \mathcal H^d \\
            &\hspace{2cm}
            +\iint_{\Gamma^-_{op}} (h+1)\,\eta\, \widetilde \nu_t^- + \nabla u^- \cdot \widetilde\nu_x^-  \dif \mathcal H^d,
    \end{split}
    \end{equation}
where we used that on $\Gamma_{tp}$, $\widetilde\nu^+ = \widetilde\nu^-$.
We last to check that the latter integrals on the right hand vanish.

Without loss of generality, we can assume that $(|\nabla u^+| + |\nabla u^-|)\ne0$ for $\mathcal H^d$-almost every point on $\Gamma^+ \cup \Gamma^-$.
Otherwise, if both the gradients vanish at $(x,t)\in \Gamma^+\cup \Gamma^-$, then $V_\nu^\pm(x,t) =0$, as $u$ is a solution of \eqref{e:classical-solution-full}. Moroever, the transversality condition \eqref{e:normal-velocity} holds on $\Gamma_{tp}\cup\Gamma^{+}_{op}\cup\Gamma^{-}_{op}$, and therefore $\widetilde\nu_t^\pm=0$.
Thus, for $\mathcal H^d$-almost every point in $\Gamma^+ \cup \Gamma^-$ they hold
\[
    \widetilde\nu_x^+ = -\frac{|\nu_x|}{|\nabla u^+|} \nabla u^+ \qquad \text{on} \quad \Gamma^+ \qquad\text{and}\qquad \widetilde\nu_x^- = \frac{|\nu_x|}{|\nabla u^+|} \nabla u^- \qquad \text{on} \quad \Gamma^-.
\]
So, we can rewrite the right hand side of \eqref{e:integral-classical}, by \eqref{e:normal-velocity}, we get
\[
   \begin{split}
\iint (u+\mu) \de_t \eta 
           & -\nabla u \cdot \nabla \eta \dif x \dif t\\
           &=\iint_{\Gamma_{tp}} |\widetilde\nu_x^+|\big( -2\, V_\nu^+  + |\nabla u_+| - |\nabla u^-| \big) \eta \dif \mathcal H^{d}\\
           &\hspace{1cm}+\iint_{\Gamma^+_{op}} |\widetilde\nu_x^+| \big( -(1-h) \, V_\nu^+ + |\nabla u^+|\big)\eta  \dif \mathcal H^d \\
            &\hspace{2cm}
            +\iint_{\Gamma^-_{op}} |\widetilde\nu_x^-|\big(-(h+1)\, V_\nu^- + |\nabla u^-| \big)\eta  \dif \mathcal H^d,
    \end{split}
\]
that are precisely the conditions of \eqref{e:classical-solution-full}, and thus they all are zero.
\end{proof}

\begin{remark}
In the presence of a force term $F \in C^\infty(D;\mathbb{R}^d)$ the classical counterpart of the weak formulation \eqref{def:enthalpy_solution-F} is given by the following system 
\begin{equation}\label{e:intro:regular-no-assumptions}
	\begin{cases}
		\de_t u = \Delta u + \dive F & \text{in}\quad \Omega_u^+\\
		\de_t u = \Delta u + \dive F & \text{in}\quad \Omega_u^-\\
		\de_t \mu = \dive F & \text{in}\quad \{u=0\}\\
		\nu_t^+ = \frac12(\nabla u^-- \nabla u^+) \cdot\nu_x^+ & \text{on}\quad \Gamma_{tp}\\
		\big(\mu-1\big)\nu_t^+ = \nabla u^+ \cdot \nu_{x}^+ &\text{on}\quad \Gamma^{+}_{op}\\
		\big(1+\mu \big)\nu_t^- = -\nabla u^- \cdot \nu_{x}^- &\text{on}\quad \Gamma^{-}_{op},
	\end{cases}
\end{equation}
We notice that in this case the mushy coefficient $\mu$ in the zero set $\{u=0\}$ depends on the time variable. 
\end{remark}

\Cref{prop:classical-enthalpy} suggests that we can define the classical solutions of the Stefan problem as follows. 
\begin{definition}[Classical Solution - revised]\label{def:appendix-classical-sol}
    Let $D\subset\R^d$ be an open set with smooth boundary and $h\in C(D)$ with $|h(x)|<1$. 
    We say that $u : D \times[0,+\infty)\to \R$ is a classical solution of the Stefan problem with respect to the initial mushy coefficient $h(x)$ if:
    \begin{itemize}
     \item[(a)] $u:D\times[0,+\infty)\to\R$ is a Lipschitz continuous function in space-time satisfying \eqref{e:regularity-assumptions-u-classical-solutions};
    \item[(b)] the surfaces $\Gamma^\pm:=\partial\Omega_u^\pm\cap\big(D\times(0,+\infty)\big)$ are Lipschitz in space-time;
    \item[(c)] there are $
    \Gamma_{tp}\subset\Gamma^+\cap\Gamma^-$, $\Gamma^+_{op}\subset\Gamma^+\setminus\Gamma^-$ and $\Gamma^-_{op}\subset\Gamma^-\setminus\Gamma^+$ such that:
    \begin{itemize}
    \item $\Gamma_{tp}$ and $\Gamma_{op}^\pm$ are relatively open subsets of $\Gamma^\pm$,
    \item $\Gamma_{tp}$ and $\Gamma_{op}^\pm$ are $C^1$-regular $d$-dimensional manifolds transversal to each $D\times\{t\}$,
    \item \eqref{e:appendix-Hausdorff-measure-of-Gamma} holds;  
    \end{itemize}
    \end{itemize}
    and if $u$ solves \eqref{e:classical-solution-full}.
\end{definition}

The same arguments imply that also the classical theory for sub/super-solutions is coherent with the enthalpy one, namely, we give the following definition.

\begin{definition}[Classical sub/super-solution - revised]\label{def:appendix-classical-subsupersol}
    Let $D\subset\R^d$ be an open set with smooth boundary and $h\in C(D)$ with $|h(x)|<1$. 
    We say that $u : D\times[0,+\infty) \to \R$ is a classical subsolution (resp. supersolution) of the Stefan problem with respect to the initial mushy coefficient $h(x)$ if $u$ satisfies the regularity assumptions (a), (b) and (c) from \Cref{def:appendix-classical-sol} and if 
\begin{equation}\label{e:classical-sub-super-solution-full}
	\begin{cases}
		\de_t u^+ \le \Delta u^+ & \text{in}\quad \Omega^+_u,\\
		\de_t u^- \le  \Delta u^- & \text{in}\quad \Omega^-_u,\\
		V_\nu \le  \frac{1}{1-h(x)}|\nabla u^+| &\text{on}\quad \Gamma_{op}^+,\\
		V_\nu \le \frac{1}{h(x)-1} |\nabla u^-| &\text{on}\quad \Gamma_{op}^-,\\
		V_\nu \le \frac{1}{2}\left( |\nabla u^+|- |\nabla u^-| \right)&\text{on}\quad \Gamma_{tp}.
	\end{cases} \qquad \text{(resp. $\ge$)}
\end{equation}
\end{definition}

\begin{remark}
    We observe that, even if may not appear clear by the convention used, the negative phase of supersolutions expands slower than the one of subsolutions.
\end{remark}

\begin{proposition} Let $u : D\times[0,+\infty) \to \R$ be a classical subsolution (resp. supersolution) in the sense of \Cref{def:appendix-classical-subsupersol}. Then, defining the mushy coefficient $\mu : D \times[0,+\infty) \to \R$ as
\[
    \mu(x,t) = h(x) \qquad\text{in} \quad \big\{u(x,t) = 0\big\},
\]
the couple $(u,\mu)$ is an enthalpy subsolution (resp. enthalpy supersolution) in the sense of \Cref{def:enthalpy-sub-supersol}, with initial mushy coefficient $h$.
\end{proposition}

\begin{proof}
    The statement follows as for \Cref{prop:classical-enthalpy} once one observes that integrating by parts, the inequalities  \eqref{e:classical-sub-super-solution-full} imply the ones in \Cref{def:enthalpy-sub-supersol}.
\end{proof}

\begin{remark}
	In terms of the Baiocchi-Duvaut transform (see, for instance, \cite{Duvaut:ResolutionStefan,Baiocchi1974,Figalli2018:Survey}) under the further assumptions that $\de_t w^+ >0$, the dynamic of the free boundary $\Gamma^+(t)$ is locally equivalent to the one described by the following parabolic obstacle problem
	\begin{equation}\label{e:obstacle-space-dep}
	\de_t w = \Delta w - (1-h(x)) \ind_{\Omega_w^+}. 
	\end{equation}
    Indeed, if $u(\cdot,0)\equiv0$ (locally), then one can show with the usual computations that if $w(x,t)$ is the time-integral of the temperature $u(x,t)$, that is,
    \[
        w(x,t) := \int_{t(x)}^t u(x,\tau)\dif \tau,
    \]
    where $t(x) := \inf\{t>0 : u(x,t) >0 \}$, is a solution of \eqref{e:obstacle-space-dep} and it holds $\de \Omega_u^+(t) = \de \Omega_w^+(t)$.
\end{remark}

\subsection*{{Acknowledgments}}
The authors acknowledge the MIUR Excellence Department Project awarded to the Department of Mathematics, University of Pisa, CUP I57G22000700001. 
F.P. is also partially supported by the INdAM - GNAMPA Project "Struttura fine e regolarita' in problemi variazionali non-lineari" (CUP E53C25002010001).
B.V. acknowledges the support of the European Research Council (ERC) via the project ERC
FiRM - Fine structure and regularity of stationary and moving free boundaries (grant agreement No. 101230705).

\bibliographystyle{amsalpha}
\bibliography{Stefan}

\providecommand{\bysame}{\leavevmode\hbox to3em{\hrulefill}\thinspace}
\providecommand{\MR}{\relax\ifhmode\unskip\space\fi MR }
\providecommand{\MRhref}[2]{%
  \href{http://www.ams.org/mathscinet-getitem?mr=#1}{#2}
}
\providecommand{\href}[2]{#2}
\begin{thebibliography}{ACS96b}

\bibitem[ACS96a]{Athanasopoulos_Caffarelli_Salsa_1996a}
Ioannis Athanasopoulos, Luis~A. Caffarelli, and Sandro Salsa, \emph{Caloric
  functions in lipschitz domains and the regularity of solutions to phase
  transition problems}, Annals of Mathematics \textbf{143} (1996), no.~3,
  413--434.

\bibitem[ACS96b]{Athanasopoulos_Caffarelli_Salsa_1996b}
\bysame, \emph{Regularity of the free boundary in parabolic phase-transition
  problems}, Acta Mathematica \textbf{176} (1996), no.~2, 245--282.

\bibitem[And04]{Andreucci2004:StefanNotes}
Daniele Andreucci, \emph{Lecture notes on the stefan problem}, 2004, online
  notes on the personal webpage:
  \url{https://www.sbai.uniroma1.it/pubblicazioni/doc/phd_quaderni/02-1-and.pdf}.

\bibitem[AS26]{AudritoSanzPerela2025:singular-elliptic-regularization}
Alessandro Audrito and Tom{\'a}s {Sanz-Perela}, \emph{On the existence of
  solutions to some singular parabolic free boundary problems}, Nonlinear
  Analysis \textbf{272} (2026), 114170.

\bibitem[AST21]{AudritoSerraTilli:SegregatedSolutions}
Alessandro Audrito, Enrico Serra, and Paolo Tilli, \emph{A minimization
  procedure to the existence of segregated solutions to parabolic
  reaction-diffusion systems}, Comm. Partial Differ. Equations \textbf{46}
  (2021), no.~12, 2268--2287.

\bibitem[Att74]{Atthey74}
David~R. Atthey, \emph{A {{Finite Difference Scheme}} for {{Melting
  Problems}}}, IMA Journal of Applied Mathematics \textbf{13} (1974), no.~3,
  353--366.

\bibitem[Bai75]{Baiocchi1974}
Claudio Baiocchi, \emph{Free boundary problems in the theory of fluid flow
  through porous media}, Proceedings of the {I}nternational {C}ongress of
  {M}athematicians ({V}ancouver, {B}.{C}., 1974), {V}ol. 2, Canad. Math.
  Congr., Montreal, QC, 1975, pp.~237--243.

\bibitem[CD80]{CannonDiBenedetto1980:ExistenceStefan}
John~R. Cannon and Emmanuele DiBenedetto, \emph{On the {{Existence}} of
  {{Weak-Solutions}} to an {\emph{n}} -{{Dimensional Stefan Problem}} with
  {{Nonlinear Boundary Conditions}}}, SIAM J. Math. Anal. \textbf{11} (1980),
  no.~4, 632--645.

\bibitem[CK06]{ChoiKim2006:WaitingTime}
Sunhi Choi and Inwon~C. Kim, \emph{Waiting time phenomena of the {{Hele-Shaw}}
  and the {{Stefan}} problem}, Indiana University Mathematics Journal
  \textbf{55} (2006), no.~2, 525--552.

\bibitem[CS05]{CaffarelliSalsa:GeomApproachToFreeBoundary}
Luis Caffarelli and Sandro Salsa, \emph{A geometric approach to free boundary
  problems}, Graduate Studies in Mathematics, vol.~68, American Mathematical
  Society, Providence, RI, 2005.

\bibitem[DDG21]{DingDuGuo2021:StefanProblemFisher-unbounded}
Weiwei Ding, Yihong Du, and Zongming Guo, \emph{The {{Stefan}} problem for the
  {{Fisher}}--{{KPP}} equation with unbounded initial range}, Calculus of
  Variations and Partial Differential Equations \textbf{60} (2021), no.~2, 69.

\bibitem[DG96]{DeGiorgi1996:ConjecturesEvolution}
Ennio De~Giorgi, \emph{Conjectures on some evolution problems}, Duke Math. J.
  \textbf{81} (1996), no.~2, 255--268 (Italian).

\bibitem[Duv73]{Duvaut:ResolutionStefan}
Georges Duvaut, \emph{R{\'e}solution d'un probl{\`e}me de {Stefan} (fusion d'un
  bloc de glace {\`a} z{\'e}ro degr{\'e})}, C. R. Acad. Sci., Paris, S{\'e}r. A
  \textbf{276} (1973), 1461--1463.

\bibitem[Eva10]{Evans2010:bookPDE}
Lawrence~C. Evans, \emph{Partial differential equations}, 2nd ed. ed., Grad.
  Stud. Math., vol.~19, Providence, RI: American Mathematical Society (AMS),
  2010.

\bibitem[Fig18]{Figalli2018:Survey}
Alessio Figalli, \emph{Free boundary regularity in obstacle problems},
  Journ\'ees \'equations aux d\'eriv\'ees partielles, Groupement de recherche
  2434 du CNRS, 2018, talk no. 2, pp.~1--24.

\bibitem[FP85]{FasanoPrimicerio85:Mushy-variablemelting}
Antonio Fasano and Mario Primicerio, \emph{Mushy regions with variable
  temperature in melting processes}, Boll. Unione Mat. Ital., VI. Ser., B
  \textbf{4} (1985), 601--626 (English).

\bibitem[Fri68]{Friedman_1968}
Avner Friedman, \emph{The stefan problem in several space variables},
  Transactions of the American Mathematical Society \textbf{133} (1968),
  51–87.

\bibitem[GZ91]{GotzZaltzman:NonincreaseMushyInhomogeneousPb}
Ivan~G. G{\"o}tz and Boris~B. Zaltzman, \emph{Nonincrease of mushy region in a
  nonhomogeneous {Stefan} problem}, Q. Appl. Math. \textbf{49} (1991), no.~4,
  741--746.

\bibitem[HS17]{HadzicShkoller2017:Gibbs-ThomsonStefan}
Mahir Had{\v z}i{\'c} and Steve Shkoller, \emph{Well-posedness for the
  {{Classical Stefan Problem}} and the {{Zero Surface Tension Limit}}}, Archive
  for Rational Mechanics and Analysis \textbf{223} (2017), no.~1, 213--264
  (en).

\bibitem[Ilm94]{Ilmanen1994:EllipticRegularization}
Tom Ilmanen, \emph{Elliptic regularization and partial regularity for motion by
  mean curvature}, Mem. Am. Math. Soc., vol. 520, Providence, RI: American
  Mathematical Society (AMS), 1994.

\bibitem[Kam61]{Kamenomostskaya}
Susanna~L. Kamenomostskaya, \emph{On the {Stefan} problem}, Mat. Sb., Nov. Ser.
  \textbf{53} (1961), 489--514 (Russian).

\bibitem[KLV95]{KingLaceyVasquez1995:AnglesHeleShaw}
John~R. King, Andrew~A. Lacey, and Juan~L. Vazquez, \emph{Persistence of
  corners in free boundaries in {Hele}-{Shaw} flow}, Eur. J. Appl. Math.
  \textbf{6} (1995), no.~5, 455--490.

\bibitem[LC31]{LameClapeyron1831}
Gabriel Lamé and Benoit Paul~\'Emile Clapeyron, \emph{M\'emoire sur la
  solidification par refroidissement d'un globe liquide}, Ann. Chimie Physique
  \textbf{47} (1831), 250--256.

\bibitem[Lio65]{LionsJL65:elliptic-regularization}
Jacques-Louis Lions, \emph{Sur certaines \'equations paraboliques non
  lin\'eaires}, Bull. Soc. Math. France \textbf{93} (1965), 155--175 (French).

\bibitem[LM68]{LionsMagenes68:BookVol1}
Jacques-Louis Lions and Enrico Magenes, \emph{Probl\`emes aux limites non
  homog\`enes et applications. {V}ol. 1}, Travaux et Recherches
  Math\'ematiques, vol. No. 17, Dunod, Paris, 1968 (French).

\bibitem[LSU68]{Ladyzhenskaya_Solonnikov_Uraltseva68:Parabolic_equation_BOOK}
Olga~A. Ladyzhenskaya, Vsevolod~A. Solonnikov, and Nina~N. Uraltseva,
  \emph{Linear and quasi-linear equations of parabolic type. {Translated} from
  the {Russian} by {S}. {Smith}}, Transl. Math. Monogr., vol.~23, American
  Mathematical Society (AMS), Providence, RI, 1968.

\bibitem[Maz11]{Mazya2011:BookSobolev}
Vladimir~G. Maz'ya, \emph{Sobolev spaces. {With} applications to elliptic
  partial differential equations. {Transl}. from the {Russian} by {T}. {O}.
  {Shaposhnikova}}, 2nd revised and augmented ed. ed., Grundlehren Math. Wiss.,
  vol. 342, Berlin: Springer, 2011.

\bibitem[Mei81]{Meirmanov81:ExampleMushy}
Anvarbek~M. Meirmanov, \emph{An example of nonexistence of a classical solution
  of the {Stefan} problem}, Sov. Math., Dokl. \textbf{23} (1981), 564--566.

\bibitem[Mei92]{Meirmanov_1992}
\bysame, \emph{The {Stefan} problem. {Translated} from the {Russian} by {Marek}
  {Niezg{\'o}dka} and {Anna} {Crowley}}, De Gruyter Expo. Math., vol.~3, Berlin
  etc.: Walter de Gruyter, 1992.

\bibitem[Ole60]{Oleinik60}
Olga~A. Ole{i}nik, \emph{A method of solution of the general stefan problem.},
  Soviet Math. Dokl. (1960), 1054–1057 (Russian).

\bibitem[Ole64]{Oleinik64:elliptic-regularization}
\bysame, \emph{On a problem of {G}. {F}ichera}, Dokl. Akad. Nauk SSSR
  \textbf{157} (1964), 1297--1300 (Russian).

\bibitem[Par26]{Park2026:radial3Dsharp}
Jeongheon Park, \emph{Melting and freezing rates of the radial interior
  {S}tefan problem in two dimension}, J. Funct. Anal. \textbf{291} (2026),
  no.~7, Paper No. 111559.

\bibitem[Pri83]{Primicerio83:Mushy}
Mario Primicerio, \emph{Mushy region in phase change problems}, Applied
  Nonlinear Functional Analysis: Variational Methods and Ill-Posed Problems
  (Karl-Heinz Hoffmann and Rudolf Gorenflo, eds.), Verlag Peter Lang,
  Frankfurt, 1983, pp.~251--269.

\bibitem[PV26]{Paiano_Velichkov2026:monotonicity}
Filippo Paiano and Bozhidar Velichkov, \emph{Solutions of the stefan problem
  with finite total enthalpy}, 2026, In preparation.

\bibitem[RB84]{RogersBerger1984:MonotonicityMushyRegionTwoPhaseStefan}
Joel~C.W Rogers and Alan~E Berger, \emph{Some properties of the nonlinear
  semigroup for the problem {$u_t - \Delta f(u) = 0$}}, Nonlinear Anal., Theory
  Methods Appl. \textbf{8} (1984), no.~8, 909--939.

\bibitem[Rub71]{rubinstein71:Stefan}
Lev~I. Rubinstein, \emph{The {{Stefan}} problem}, Translation of {{Mathematical
  Monograph}}, vol.~27, American Mathematical Society, Providence, RI, 1971.

\bibitem[ST12]{Serra_Tilli_2012:wave}
Enrico Serra and Paolo Tilli, \emph{Nonlinear wave equations as limits of
  convex minimization problems: proof of a conjecture by de giorgi}, Ann. Math.
  (2) \textbf{175} (2012), no.~3, 1551–1574.

\bibitem[Ste89a]{Stefan1889:AcidBaseDiffusion}
Josef Stefan, \emph{Über die diffusion von säuren und basen gegen einander},
  Monatsh. Chem. \textbf{10} (1889), 201--219 (german), Original publication:
  \emph{S.-B. Wien. Akad. Math. Natur.} \textbf{98} (1889), 616--634.

\bibitem[Ste89b]{Stefan1889:HeatConduction}
\bysame, \emph{Über einige probleme der theorie der wärmeleitung}, Wien.
  Akad. Sitzungsber. Math.-Naturwiss. Kl. \textbf{98} (1889), no.~2a, 473--484
  (german).

\bibitem[Ste90]{Stefan1890:Diffusion}
\bysame, \emph{Über die verdampfung und die auflösung als vorgänge der
  diffusion}, Ann. Phys. \textbf{277} (1890), no.~12, 725--747 (german),
  Original publication: \emph{S.-B. Wien. Akad. Math. Natur.} \textbf{98}
  (1889), 1418--1442.

\bibitem[Ste91]{Stefan1891:IceFormation}
\bysame, \emph{Über die theorie der eisbildung, insbesondere über die
  eisbildung im polarmeere}, Ann. Phys. Chem. \textbf{42} (1891), 269--286
  (german), Original publication: \emph{S.-B. Wien. Akad. Math. Natur.}
  \textbf{98} (1889), 965--983.

\bibitem[Vis96]{Visintin:BookPhaseTransition}
Augusto Visintin, \emph{Models of phase transitions}, Prog. Nonlinear Differ.
  Equ. Appl., vol.~28, Boston: Birkh{\"a}user, 1996.

\end{thebibliography}

\end{document}